\documentclass[11pt]{article}
\usepackage[utf8]{inputenc}
\usepackage{amsmath}
\usepackage{amssymb}  
\usepackage{mathrsfs}
\usepackage{tabularx}
\usepackage{booktabs}
\usepackage{enumerate}
\usepackage{amsthm}
\usepackage{color}
\usepackage{xcolor}
\usepackage{graphicx}
\usepackage{chngcntr}
\usepackage{enumitem}
\usepackage[sort]{natbib}
\usepackage{url}
\usepackage{ulem}
\usepackage{graphicx}
\usepackage{geometry}
\usepackage{varioref}
\usepackage{etoolbox,fontawesome}
\usepackage[section]{algorithm}
\usepackage{hyperref}
\usepackage{setspace}

\usepackage[capitalise,nameinlink
    ,noabbrev]{cleveref}
  \hypersetup{
    colorlinks=true,
    linkcolor=blue,
    citecolor=blue
}
\AtBeginEnvironment{thebibliography}{\def\url#1{\href{#1}{\faExternalLink}}}

\counterwithin{figure}{section}
\counterwithin{table}{section}
\usepackage{autonum} 

\newtheorem{theorem}{Theorem}[section]
\newtheorem{assumption}{Assumption}[section]
\newtheorem{definition}{Definition}[section]
\newtheorem{proposition}{Proposition}[section]
\newtheorem{example}{Example}[section]
\newtheorem{remark}{Remark}[section]
\newtheorem{lemma}{Lemma}[section]

\newcommand{\op}{o_{\mathbb{P}}}
\newcommand{\Op}{O_{\mathbb{P}}}

\newcommand{\E}{\mathbb{E}}
\newcommand{\F}{\mathcal{F}} 
\newcommand{\FF}{\mathcal{F}}
\newcommand{\B}{\mathcal{B}}
\newcommand{\R}{\mathbb{R}}
\newcommand{\mf}{\mathbf}
\newcommand{\bs}{\boldsymbol}
\newcommand{\lf}{\lfloor}
\newcommand{\rf}{\rfloor}
\newcommand{\proj}{\mathcal{P}}
\newcommand{\I}{\mathcal{I}}
\newcommand{\T}{\top}
\newcommand{\lt}{\left}
\newcommand{\rt}{\right}
\newcommand{\pp}{\mathbb{P}}

\numberwithin{equation}{section}

\definecolor{darkgreen}{rgb}{0.0, 0.5, 0.0}
\definecolor{ashgrey}{rgb}{0.7, 0.75, 0.71}

\newcommand{\nb}{\lf nb_n \rf}
 
\usepackage{xr}
\makeatletter

\title{\bf Detecting long-range dependence for time-varying  linear models}
\date{}
\author{ \small Lujia Bai \\
\small Center for Statistical Science, \\
	\small Department of Industrial Engineering, \\
	\small Tsinghua University\\
\and \small Weichi Wu\\
\small Center for Statistical Science, \\
	\small Department of Industrial Engineering, \\
	\small Tsinghua University\\}
\begin{document}
\maketitle
\begin{abstract}
We consider the problem of testing for long-range dependence in time-varying coefficient regression models, where the covariates and errors are locally stationary, allowing complex temporal dynamics and heteroscedasticity.  We develop KPSS, R/S, V/S, and K/S-type statistics based on the nonparametric residuals. Under the null hypothesis, the local alternatives as well as the fixed alternatives, we derive the limiting distributions of the test statistics. As the four types of test statistics could degenerate when the time-varying mean, variance, long-run variance of errors, covariates, and the intercept lie in certain hyperplanes, we show the bootstrap-assisted tests are consistent under both degenerate and non-degenerate scenarios. In particular, in the presence of covariates the exact local asymptotic power of the bootstrap-assisted tests can enjoy the same order as that of the classical KPSS test of long memory for strictly stationary series. The asymptotic theory is built on a new Gaussian approximation technique for locally stationary long-memory processes with short-memory covariates, which is of independent interest. The effectiveness of our tests is demonstrated by extensive simulation studies and real data analysis.
\end{abstract}

\footnotetext[1]{E-mail addresses: \href{blj20@mails.tsinghua.edu.cn}{blj20@mails.tsinghua.edu.cn}(L.Bai), \href{wuweichi@mail.tsinghua.edu.cn}{wuweichi@mail.tsinghua.edu.cn}(W.Wu)}

\section{Introduction}
Consider the time-varying coefficient linear model 
\begin{align}\label{model1}
    y_{i,n}=\mf x_{i,n}^{\T}\bs \beta(t_i)+e_{i,n},\quad i=1,2,\cdots n,
\end{align}
where the covariate vector $\mf x_{i,n}=(1,x_{i,2,n},...,x_{i,p,n})^\top$ is a $p$-dimensional {\it short-range dependent} (SRD) locally stationary time series and $y_{i,n}$ is the response variable, $t_i = i/n$. At each time point $t_i$, we only observe one realization $(\mf x_{i, n},y_{i,n})$ and no repeated measurement is available. The time-varying regression coefficient function $\bs \beta(\cdot)$ is a $p$-dimensional function with each coordinate a smooth function on $[0,1]$ and  the zero mean error process $(e_{i,n})$ is a possibly {\it long-range dependent} (LRD) or {\it long-memory} time series. More precisely, we assume $(e_{i,n})$ is a locally stationary $I(d)$ process 
i.e., for $1\leq i\leq n$
\begin{align}\label{Locallyerror}
   (1-\mathcal B)^{d} e_{i,n}=u_{i,n},
\end{align}
where $\mathcal B$ is the lag operator, $d\in [0,1/2)$ is the {\it long-memory parameter} and  $u_{i,n}$ is a SRD or {\it short-memory} locally stationary process. The strict definitions of locally stationary and long-memory processes are deferred to \cref{sec:model}. The error model \eqref{Locallyerror} naturally generalizes classical stationary SRD and LRD processes by allowing their generating mechanism to vary with time. Observe that  $(e_{i,n})$ will reduce to the SRD process $(u_{i,n})$ if $d=0$ and will be a LRD process if $d \in (0, 1/2)$. In fact,  when $(u_{i,n})$ is stationary, \eqref{Locallyerror} allows the classical stationary long-memory processes (e.g. FARIMA-GARCH models), which have found extensive application in hydrology (\cite{ZHANG2011121}, \cite{koutsoyiannis2013hydrology}), economics and finance (\cite{caporale2013long}, \cite{unemployment2016}) and many other fields since first introduced by \cite{hurst1951long}. Moreover, model \eqref{model1} admits heteroscedasticity, i.e., the dependence of $u_{i,n}$ on $(\mf x_{r,n})_{r=1}^n$, see \cref{sec:model} for more details.

The time-varying regression model \eqref{model1} with time series errors has attracted enormous attention, see for instance 
\cite{fan2000simultaneous}, 
\cite{zhou2010simultaneous} 
and \cite{chen2018jbes} 
where the errors are assumed to be SRD, and  \cite{kulik2012conditional}, \cite{BeranLongmemory} and \cite{ferreira2018estimation} where LRD errors are considered.   The aforementioned research reveals that nonparametric estimators of the time-varying coefficient ${\bs \beta}(\cdot)$ possess distinct properties under the two scenarios, $d=0$ and $0<d<1/2$.  When $d=0$, consider the local linear estimator of the multivariate coefficient function $ {\bs \beta}(\cdot)$ using the kernel function $K(\cdot)$ and the bandwidth $b_n$, of which the asymptotic behavior rests on the distributions of $(\mf x_{i,n})$ and $(e_{i,n})$. In particular, the order of the deviation $|\hat {\bs \beta}(\cdot)-\bs \beta(\cdot)|$ is determined by the long-memory parameter $d$. 
For $d=0$ and $t\in (0,1)$, \cite{zhou2010simultaneous} shows that  under mild conditions, 
\begin{align}\label{srdbeta}
    \sqrt{nb_n}(\hat {\bs \beta}(t)-{\bs \beta}(t)-b_n^2{\bs \beta}^{\prime\prime}(t)\mu_2/2)\Rightarrow N(0, \phi_0{\bs \Sigma}(t)),
\end{align}
where $\mu_2$ and $\phi_0$ are constants determined by $K(\cdot)$ and ${\bs \Sigma}(t)$ is  determined by the moments of the process $(\mf x_{i,n}e_{i,n})_{i=1}^n$. Meanwhile, for $d>0$ and $p=1$,
Theorem 7.22 in \cite{BeranLongmemory} shows that for stationary $e_{i,n}$ under regularity conditions,
\begin{align}\label{lrdbeta}
    (nb_n)^{1/2-d}(\hat {\bs \beta}(t)-{\bs \beta}(t)-b_n^2{\bs \beta}^{\prime\prime}(t)\mu_2/2)\Rightarrow N(0,V(d)),
\end{align}
where $V(d) = 2c_f \Gamma(1-2d) \sin(\pi d) \int_{-1}^1\int_{-1}^1 K(x)K(y)|x-y|^{2d-1}dxdy$, and $c_f$ is a constant related to the spectral density of errors. Equation \eqref{lrdbeta} shows that for $d>0$ the convergence rate of $\hat {\bs \beta}(t)$ is $(nb_n)^{d-1/2}$, which is much slower than the well-known $(nb_n)^{-1/2}$ convergence rate as given by \eqref{srdbeta} when $d=0$.  Therefore, a crucial problem 
of the statistical inference of model \eqref{model1} is to test
\begin{align}\label{WC-hypo1}
    H_0: d=0~~~\text{versus}~~~ H_A: 0<d<1/2.
\end{align}

The testing problem \eqref{WC-hypo1} for  \eqref{model1} is closely related to the existing tests of `spurious long memory', which refers to the phenomenon that in the presence of regime changes, level shifts or certain deterministic trends, a short memory process could exhibit many properties of a long-memory process, known as the  `spurious long-memory' effects, see for example \cite{giraitis2001testing}, \cite{qu2011test}, and \cite{mccloskey2013memory}. These findings motivate the tests for distinguishing genuine and
 spurious long memory. Among others, 
 \cite{qu2011test}, \cite{preu2013} and \cite{SIBBERTSEN201833} consider testing the null hypothesis of stationary long memory against spurious long memory. Meanwhile, several tests have been introduced to test the null hypothesis of spurious long memory, for which a prevailing approach is to assume a specific and parametric form of non-stationarity, see \cite{shao2006aos}, \cite{harris2008testing}, 
 and \cite{davis2013consistency} among others. Recently, there has been growing interest in detecting long memory in the presence of general non-stationarity, see for example \cite{Detectinglrdependence} which considered locally stationary moving average formulation. 
 In practice, by testing \eqref{WC-hypo1} for model \eqref{model1}, we are able to identify a new type of `spurious long memory' resulting from the {\it misspecification in the conditional mean}. See our data analysis in \cref{subsec:HKdata} where we apply our method to the Hong Kong circulatory and
respiratory data.

The goal of the present work is to test the {\it hypothesis \eqref{WC-hypo1}} under complex and general temporal dynamics, assuming that $(\mathbf x_{i,n})$ and $(u_{i,n})$  belong to the flexible class of locally stationary processes generated by smoothly changing underlying mechanisms. Although some related literature has studied hypothesis \eqref{WC-hypo1} for linear regression models with deterministic covariates, see for example \cite{harris2008testing},
to the best of the authors' knowledge, this work is the first instance investigating testing \eqref{WC-hypo1} for \eqref{model1} in the presence of time series covariates. In the literature, KPSS (Kwiatkowski, Phillips, Schmidt, and Shin, see  \cite{lee1996power}),
R/S (range over standard deviation, see  \cite{hurst1951long}
), V/S (rescaled variance, see \cite{giraitis2003rescaled}), and K/S (which has a  limiting distribution of Kolmogoroff-Smirnoff form, see \cite{lima2004robustness}) tests have been widely used for long memory detection in stationary processes. In this paper, we develop new KPSS, R/S, V/S and K/S-type tests tailored to the non-stationary time series time-varying regression problem \eqref{model1}. The limiting distributions of the test statistics under the null hypothesis, the local and fixed alternatives are then derived.  Our results differ from their stationary counterparts due to the following reasons. (1) Under the null hypothesis, it is well-known that the KPSS, R/S, V/S and K/S tests are built on the partial sum process whose convergence rate is $
n^{-1/2}$. However, the nonparametric estimate $\hat {\bs \beta}(\cdot)$ induces stochastic errors much larger than $n^{-1/2}$ which will lead to different limiting distributions as well as possible degeneracy. (2) Due to the non-stationary errors and covariates, the partial sum processes cannot be approximated by processes with stationary increments, which makes the test statistics non-pivotal. The major contributions of the paper lie in the following three aspects. 

Firstly, our methods are applicable to the locally stationary time series regression, which has found considerable attention in various related fields, see for instance \cite{vogt2012}, \cite{you2019}, \cite{zhou2010simultaneous} and many others. 
 In particular,  the flexible locally stationary framework allows  the error processes to display conditional and unconditional heteroscedasticity that has been increasingly investigated (see \cite{harris2017adaptive} and \cite{cavaliere2020adaptive}) in the context of long-memory models. Both the evolving distributional properties of the locally stationary data and the long-memory properties pose long-standing challenges to the inference of time-varying coefficient linear model \eqref{model1} due to the lack of general Gaussian approximation 
techniques for non-stationary long-memory processes.  For stationary long-memory processes, Gaussian approximation has been studied by for example \cite{dehling1989empirical}. 
Recently, \cite{wu2018} developed Gaussian approximation schemes for a class of locally stationary long-memory linear processes. However, their results cannot accommodate regression problems with time series covariates which requires the analysis of distributional properties of the partial sum process of $(\mathbf x_{i,n} e_{i,n})_{i=1}^n$.  In this paper, we address this issue via a further Gaussian approximation theorem, allowing  $(\mathbf x_{i,n})$ 
and $(e_{i,n})$ to be non-stationary SRD and LRD processes, respectively, with flexible dependence between them.  

Secondly, we develop effective bootstrap approaches which circumvent the difficult estimation of the non-pivotal limiting distributions of the test statistics  under time series non-stationarity. In particular,  the test statistics could degenerate when the time-varying mean, variance, long-run variance of errors and covariates, and the intercept lie in  certain hyperplanes whose geometry cannot be directly identified from the data. Importantly, regardless of the degeneracy of test statistics, our bootstrap procedures are consistent and possess good finite sample properties. 
Furthermore, we show that the exact local asymptotic power of the four types of bootstrap-assisted tests can reach the order $O( \log^{-1} n)$ in the presence of time series covariates, no matter whether the test statistics degenerate under the null hypothesis. This rate coincides with \cite{shao2007local} which studies a similar problem of testing the SRD null hypothesis against LRD local alternatives for  strictly stationary series without covariates.

The rest of the paper is organized as follows. 
Section \ref{notation} introduces necessary notation.
Section \ref{sec:model} formally states the non-stationary LRD model and the related assumptions. Section \ref{sec:4} provides
the test statistics, and establishes the asymptotic results  
via the new Gaussian approximation theory for the product of non-stationary SRD and LRD processes. 
Section \ref{sec:bootstrap} discusses
the  bootstrap algorithms. Section \ref{finite} reports the simulation results and the analysis of Hong Kong circulatory and respiratory data. Section \ref{conclude} provides a brief concluding remark.   In \cref{sec:appendA}, we provide the proof of the new Gaussian approximation theory. 
Detailed proofs, the literature review of KPSS and related tests, the implementation details including the selection of tuning parameters, additional simulation results, data analysis results (including COVID-19 data), and additional algorithms are relegated to the supplement.

\section{Notation}\label{notation}
For a matrix $\mf A = (a_{ij})_{1\leq i \leq n, 1\leq j \leq m} \in \R^{n\times m}$, let $|\mf A| = (\sum_{j=1}^m \sum_{i=1}^n a_{ij}^2)^{1/2}$ and write $\mf A \geq 0$ if $\mf A$ is semi-positive definite. Let $(\mf A)_{(1,1)}$ denote the element of $\mf A$ in the first column and first row.  Notice that when $m=1$, $\mf A$ is a vector. For  $\mf A\geq 0$ with eigendecomposition $\mf A = \mf Q\mf D \mf Q^{\T}$ with  orthonormal matrix $\mf Q$  and  diagonal matrix $\mf D$, the root of $\mf A$  is defined by $\mf A^{1/2} = \mf Q \mf D^{1/2} \mf Q^{\T}$, where $\mf D^{1/2}$ is the elementwise root of $\mf D$. Let $\mf I_p$ denote the $p$-dimensional identity matrix. For a  random matrix $\mf A$, for $q \geq 1$, let $\|\mf A \|_q  = (\E|\mf A|^q)^{1/q}$ denote the $\mathcal{L}^q$-norm of the random variable  $|\mathbf{A}|$  and write $\|\cdot\|=\|\cdot\|_2$ for short. Write $\mf A \in \mathcal{L}^q$ if $\|\mf A \|_q < \infty$.
For a function $f(\cdot)$, write $f \in C[0,1]$ if $f$ is continuous over $[0,1]$, $f \in C^p[0,1]$ if the $p_{th}$ order derivative of $f$ is continuous over $[0,1]$. Write $\mf A \in C[0,1]$ and $\mf A \in C^p[0,1]$ if each element $A_{ij}(\cdot)$ in $\mf A(\cdot)$ is in $C[0,1]$ and $C^p[0,1]$, respectively.
For any kernel function $K(\cdot)$, let $K^*(\cdot)$ denote the jackknife equivalent kernel $2 \sqrt{2} K(\sqrt{2}x) - K(x)$. Denote by $\lfloor x\rfloor$ the largest integer smaller or equal to $x$. For any two positive real sequences $a_n$ and $b_n$, write  $a_n \asymp b_n$ if $\exists\  0 <c < C<\infty$ such that $c<\liminf_{n \to \infty} \frac{a_n}{b_n}<\limsup_{n \to \infty} \frac{a_n}{b_n}<C$. Let $t \wedge s$ denote the smaller value in  $t$ and $s$. Let `$\Rightarrow$' denote convergence in distribution. Write $\lambda$ as Lebesgue measure on $[0,1]$. Let `$:=$' denote `defined as'.

\section{Model assumptions}\label{sec:model}
We start by introducing the time series time-varying regression model \eqref{model1} in detail. Recall model \eqref{model1}  has the following form
 \begin{equation}
       y_{i,n}=\mathbf{x}_{i,n}^{\T} \boldsymbol{\beta}(t_i)+e_{i,n}, \quad i=1, \ldots, n, ~~
      \text{ where } ~ (1-\B)^d e_{i,n} = u_{i,n}, \quad d \in [0,1/2).
  \end{equation}

We assume that the process $(u_{i,n})_{i=-\infty}^n$  and the covariate process $(\mf x_{i,n})_{i=1}^n$ have the form 
  \begin{align}\label{uxmodel}
      u_{i,n} = H(t_i,\FF_i), \quad \mf x_{i,n}=\mf W(t_i,\FF_i),
  \end{align}
  where $\FF_i=(\varepsilon_{-\infty},...,\varepsilon_i)$, $(\varepsilon_{i})_{i\in \mathbb Z}$ are $i.i.d.$ random variables, $H$ and $\mf W=(W_1,...,W_p)^\top$ are measurable functions such that $H:(-\infty,1]\times \mathbb R^{\mathbb Z}\rightarrow \mathbb R$,  $W_s:[0,1]\times \mathbb R^{\mathbb Z}\rightarrow \mathbb R$, $2\leq s\leq p$, while $W_1$ is fixed to be $1$ corresponding to the intercept of the regression.  Define $(\varepsilon^{\prime}_i)_{i \in \mathbb{Z}}$ as an $i.i.d.$ copy of $(\varepsilon_i)_{i \in \mathbb{Z}}$, and for $j \geq 0$, let $\F^{*}_j = (\F_{-1},\varepsilon^{\prime}_0, \varepsilon_1,\cdots, \varepsilon_{j-1}, \varepsilon_{j})$. For any  (vector) process $\mf L(t,\FF_i)$, it is said to be $\mathcal L^q$ stochastic Lipschitz continuous in the interval $\mathcal I$ (denoted by $\mathbf L\in \mathrm{Lip}_q(\mathcal I)$) if for $t_1,t_2\in \mathcal I$, there exists a constant $M>0$ such that
  \begin{align}
      \|\mathbf L(t_1,\FF_0)-\mathbf L(t_2,\FF_0)\|_q\leq M|t_1-t_2|.
  \end{align}
We say the process $\mf L(t,\FF_i)$ is {\it locally stationary} (LS) on $\mathcal I$ if $\mf L(t,\FF_i)\in \mathrm{Lip}_q(\mathcal I)$ for some $q\geq 2$. Write $\mathrm{Lip}_q=\mathrm{Lip}_q([0,1])$ for short. The locally stationary process offers a flexible nonparametric device to
characterise the complex temporal dynamics of the error and covariate processes in \eqref{uxmodel}, 
which is based on Bernoulli shift processes and leads to a general framework of nonlinear processes, see \cite{wu2005nonlinear}. Other formulations of locally stationary processes include \cite{dahlhaus1997} and \cite{nason2000wavelet}. 
See \cite{dahlhaus2019bej} for a comprehensive review. The physical dependence measure of the nonlinear filter $\mf L \in \mathcal{L}^q$ ($q>0$) over the interval $\mathcal I$
    is defined by
    \begin{equation}
      \delta_q(\mf L, k, \mathcal I) = \underset{t \in \mathcal {I}}{\sup}\|\mf L(t,\F_k) - \mf L(t, \F_k^*) \|_q.
    \end{equation}
The physical dependence measure $\delta_q(\mathbf L,k, \mathcal I)$ quantifies the influence of the input $\varepsilon_0$ on the output $\mf L(t,\FF_k)$ over the interval $\mathcal I$. Observe that $\delta_q(\mf L,k, \mathcal I)=0$ if $k<0$. Write $\delta_q(\mf L,k)=\delta_q(\mf L,k, [0,1])$ for short. We proceed to define SRD and LRD non-stationary processes for non-stationary time series.

\begin{definition}\label{def:SRDLRD}
The univariate process $( L(t,\FF_i))_{i=-\infty}^\infty, t\in \mathcal I$ is said to be SRD  if  
$$ \sum_{j=-\infty}^{\infty} \sup_{t,s \in \mathcal I} \left|\operatorname{Cov}\left( L(t, \FF_0),   L(s, \FF_{j})\right)\right|<\infty, $$ 
and a LRD process otherwise. The $p$-dimensional vector process $( \mf L(t,\FF_i))_{i=-\infty}^\infty, t\in \mathcal I$ is SRD if each of its component is SRD, and is otherwise LRD. 

\end{definition}

Definition \ref{def:SRDLRD} distinguishes the SRD and LRD by the uniform summability of covariance, which naturally extends the traditional definition of long memory in second-order stationary univariate processes, see for example Condition III in Chapter 2 of \cite{pipiras2017long}.  The uniform long-memory definition has been introduced to define LRD and SRD non-stationary time series in \cite{wu2018}.
Various definitions of LRD stationary processes could be found in  \cite{pipiras2017long} 
and definitions of LRD locally stationary processes are discussed in \cite{BERAN2009900}, \cite{Detectinglrdependence} and \cite{ferreira2018estimation} among others.
In this paper, we posit the following assumptions.
  \begin{assumption}\label{assumptionHp}
The zero mean SRD process $(u_{i,n})_{i=-\infty}^{n}$ satisfies
    \begin{enumerate}[label=(a\arabic*)', itemsep=2pt,topsep=0pt,parsep=0pt]
      \item   $H(t,\FF_0) \in \mathrm{Lip}_2(-\infty,1]$, $\E H(t,\FF_0)=0$, and 
   $ \sup_{ t \in (-\infty,1] }\left\|H\left(t, \mathcal{F}_{0}\right)\right\|_{4}<\infty
   $.\label{A:H_long}
   \item $ \delta_4(H, k,(-\infty,1]) = O(\chi^k)$ $\text{for some}~ \chi \in (0,1)$. \label{A:H_delta_long}
      \item Define the long-run variance function as
     $
      \sigma^{2}_H(t):=\sum_{k=-\infty}^{\infty} \operatorname{Cov}\left(H\left(t, \mathcal{F}_{0}\right), H\left(t, \mathcal{F}_{k}\right)\right)$,  $t \in (-\infty, 1],
    $
     which satisfies that $\underset{t \in (-\infty,1] }{\inf} \sigma_H^2(t)> 0$ and  $\underset{t \in (-\infty,1] }{\sup}\sigma_H^2(t) < \infty.$ \label{A:H_lrv}
      \item $\sigma_H^2 (\cdot)$ is twice continuously differentiable on $[0,1]$. \label{A:H_smooth_long}
     \end{enumerate} 
  \end{assumption}
Condition \ref{A:H_long} imposes the assumptions of local stationarity and finite forth moment on the innovations $(u_{i,n})$.  Condition \ref{A:H_long} will be satisfied if $\sup_{t \in (-\infty, 1]} \|\frac{\partial}{\partial t}H(t,\FF_0)\| <\infty$. Condition \ref{A:H_delta_long} ensures that the innovations $(u_{i,n})$ are SRD and satisfy geometric measure contraction (GMC). 
Conditions \ref{A:H_lrv} and \ref{A:H_smooth_long} guarantee that the innovations $(u_{i,n})$ have a finite, non-degenerate and smooth long-run variance. \par 
Notice that under the null hypothesis $d=0$, the error process $(e_{i,n})$ reduces to $(u_{i,n})$, which indicates that $(e_{i,n})$ is a SRD process. When $d > 0$, $(e_{i,n})$ is generated by a binomial weighted combination of $u_{i,n}=H(t_i,\FF_i)$  starting from the infinite past ($i=-\infty$). To stress that $(e_{i,n})$ is a LRD process under the alternative hypothesis with respect to $d$, in the remaining of this article we write $e_{i,n}$ as $e_{i,n}^{(d)}$ when $d>0$, i.e., $e_{i,n}^{(d)}  = (1-B)^{-d} u_{i,n} = \sum_{k=0}^{\infty} \psi_k(d) u_{i-k}$, $\psi_j(d) = \Gamma (j+d) / [\Gamma(d) \Gamma (j+1)]$. We further write $e_{i,n}^{(d)} = H^{(d)}(t_i, \F_{i})$, where $ H^{(d)}(t, \F_{l}) = \sum_{k=0}^{\infty} \psi_k(d) H(t - t_k, \F_{l-k})$. The following Proposition \ref{Prop31} elaborates that the physical dependence measure of $(e^{(d)}_{i,n})$  relies on $d$. 

\begin{proposition}\label{Prop31}
Under Assumption \ref{assumptionHp},  we have uniformly for $l \geq 0$, $0 < d < 1/2$,
\begin{align}
  \delta_p(H^{(d)}, l, (-\infty, 1]) = O\{(1+l)^{d-1}\}.
\end{align}
\end{proposition}

Our formulation of $(e_{i,n})$ in \eqref{Locallyerror} allows for a wide class of  non-stationary SRD and LRD processes under $H_0$ and $H_A$, respectively, including the following examples.

\begin{example}[Linear  locally stationary process]\label{ex1}
Consider the time-varying FARIMA($p,d,q$) model $(0 < d<1/2)$ (recall that $u_{j,n}=H(t_j,\FF_j)$, $-\infty \leq j \leq n$)
\begin{align}
   &(1-\mathcal{B})^{d}e_{i,n}=u_{i,n}, ~ i = 1,2,\cdots, n,\text{with }\Phi^{p}(\mathcal{B}, t)H(t,\FF_j)=  \Theta^{q}(\mathcal{B}, t) \varepsilon_{j},\nonumber
\end{align}
where $t \in (-\infty,1]$, $j\in \mathbb{Z}$, $\Phi^{p}(z, t)=1+\phi_{1}(t) z+\cdots+\phi_{p}(t) z^{p}$ and $\Theta^{q}(z, t)=1+\theta_{1}(t) z+\cdots+\theta_{q}(t) z^{q}$ are
polynomials with degrees $p$ and $q$, and the random variables $(\varepsilon_{i})_{i\in\mathbb Z}$ are $i.i.d.$ with mean $0$ and
variance $1$.  Assume that  for $t \in (-\infty,1]$, $\{\phi_i(t), 1\leq i \leq p\}$ and $\{\theta_j(t), 1 \leq j \leq q\}$ are twice differentiable,  $\Phi^{p}(z, t)$ and $\Theta^{q}(z, t)$ do not share the same roots, and $\Phi^{p}(z, t)$ does not have roots in the unit disk $\{|z| \leq 1\}$. Then there exists real-valued differentiable functions $(a_i(t))_{i\geq 0}$ such that $A(z,t) =\Theta^{q}(z, t)/\Phi^{p}(z, t) = \sum_{i=0}^{\infty} a_{i}(t) z^i$, where for $t\in (-\infty ,1]$, $|a_j(t)|$ and $|a_j^{\prime}(t)|$ are summable. Consequently,  we have the $\text{MA}(\infty)$ representation 
 \begin{equation}\label{eq:linear}
 e_{i,n}=(1-\mathcal B)^{-d}u_{i,n}=(1-\mathcal B)^{-d}\sum_{j=0}^\infty a_j(t_i)\mathcal B^j\varepsilon_{i}:=\sum_{j=0}^{\infty} b_{j,i}\varepsilon_{i-j}.
  \end{equation}
  
  When $d=0$, it follows that $b_{j,i} = a_j(t_i)$.
    Thus, $(e_{i,n})$ is SRD according to Definition \ref{def:SRDLRD}. When $d > 0$, by Lemma 3.2 of \cite{KOKOSZKA199519},  $b_{j,i} =\sum_{l=0}^j \psi_l a_{j-l}(t_{i-l})= L_i(j)j^{d-1}$, where $L_i(\cdot)$ is a slowly varying function for each $i$.
  Suppose for $1 \leq i \leq n$, $|L_i(j)| \leq L(j)$ for some slowly varying function $L(\cdot)$.  Noticing that $\operatorname{Cov}(e_{i,n}, e_{i+k,n}) = \sum_{j=0}^{\infty} b_{j,i}b_{j+k, i+k}$. By Proposition 2.2.9 in \cite{pipiras2017long}, $  \sup_{1 \leq i \leq n}|\operatorname{Cov}(e_{i,n}, e_{i+k,n})|$  is of order $k^{2d-1}$ for all $i \in \mathbb{Z}$. Hence, \eqref{eq:linear} is LRD according to Definition \ref{def:SRDLRD}.
  \end{example}

  \begin{example}[Nonlinear locally stationary process]
  Consider the time-varying ARFIMA($p,d,q$)-GARCH($1,1$) process $(1-\B)^d e_{i,n} =u_{i,n}, 1\leq i\leq n$,  where $u_{j,n}=H(t_j,\FF_j)$ and
  \begin{align}
 \Phi^{p}(\mathcal{B}, t)H(t, \FF_j)= \Theta^{q}(\mathcal{B}, t)v_j(t), ~~ v_j(t) =\varepsilon_{j} \sigma_{j}(t),~~j\in \mathbb Z, ~t\in (-\infty,1],\nonumber
\end{align}
where $\sigma^2_{j}(t) = c(t)+\alpha(t)v_{j-1 }^{2}(t) + \beta(t) \sigma^2_{j-1}(t)$, $(\varepsilon_{i})_{i\in\mathbb Z}$ are $i.i.d.$ random variables with mean 0 and variance 1, $c(t), \alpha(t), \beta(t)$ are smooth non-negative functions, and $\Phi^{p}(z, t)=1+\phi_{1}(t) z+\cdots+\phi_{p}(t) z^{p}$ and $\Theta^{q}(z, t)=1+\theta_{1}(t) z+\cdots+\theta_{q}(t) z^{q}$ are
polynomials with degrees $p$ and $q$.  Assume that  for $t \in (-\infty,1]$, $\{\phi_i(t), 1\leq i \leq p\}$, $\{\theta_j(t), 1 \leq j \leq q\}$, $c(t)$, $\alpha(t)$ and $\beta(t)$ are twice differentiable,  $\Phi^{p}(z, t)$ and $\Theta^{q}(z, t)$ do not share the same roots, and $\Phi^{p}(z, t)$ does not have roots in the unit disk $\{|z| \leq 1\}$, $(\varepsilon_{i})_{i\in\mathbb Z} \in \mathcal{L}^{8}$,  
$\sup _{t \in (-\infty,1]} c(t) <\infty$ and $\sup _{t \in (-\infty,1] }\left\|\alpha (t) \varepsilon_{t}^{2}+\beta(t)\right\|_{4}<1$. Then, by Example 2 of \cite{wu2011gaussian},
\ $(u_{i,n})_{i=-\infty}^{n}$ satisfy \ref{A:H_delta_long} of Assumption \ref{assumptionHp}. When $d > 0$,
Definition \ref{def:SRDLRD} can be verified similarly as given in Example \ref{ex1}, since $(v_{i}(t))_{i\in \mathbb Z }$ are white noises.

  \end{example}

\section{Main results} \label{Test} \label{sec:4}


Since $(e_{i,n})$ is not observable in  \eqref{model1}, we propose to test $H_0$ based on nonparametric residuals. Specifically, we adopt the local linear approach (see for instance \cite{fan1993local} and \cite{fan1996local}) to estimate ${\bs \beta}(t)$ in \eqref{model1}, i.e.,
\begin{equation}\label{eq:loclin}
(\hat{\boldsymbol{\beta}}_{b_{n}}(t), \hat{\boldsymbol{\beta}}_{b_{n}}^{\prime}(t))=\underset{\boldsymbol \eta_{0},\boldsymbol \eta_{1} \in \mathbb{R}^{p}}{\arg \min}\sum_{i=1}^{n}\{y_{i, n}-\mathbf{x}_{i,n}^{\T} \boldsymbol \eta_{0}- \mathbf{x}_{i,n}^{\T}  \boldsymbol \eta_{1}(t_{i}-t)\}^{2} K_{b_{n}}(t_{i}-t),
\end{equation}
where $K(\cdot)$ is a kernel function with finite support $[-1,1]$ and $b_n$ is a bandwidth and $K_{b_n}(\cdot)=K(\cdot/b_n)$. To further eliminate the bias term involving $\boldsymbol \beta^{\prime \prime }(\cdot)$, we use the jackknife bias-corrected estimator in \cite{wu2007inference} : 
\begin{equation}
  \tilde{\boldsymbol{\beta}}_{b_{n}}(t)=2 \hat{\boldsymbol{\beta}}_{b_{n} / \sqrt{2}}(t)-\hat{\boldsymbol{\beta}}_{b_{n}}(t).\label{eq:jack}
  \end{equation}
Then, we obtain the nonparametric residuals $(\tilde e_{i,n})$, i.e.,
$
  \tilde e_{i,n} = y_{i,n} - \mf x_{i,n}^{\T} \tilde{\boldsymbol \beta}(t_i).\nonumber
$ For simplicity, define $\tilde S_{r,n} =\sum_{i=\lfloor nb_n \rfloor+1}^r \tilde{e}_{i,n}$, $r=\lfloor nb_n \rfloor+1,\cdots n-\lfloor nb_n\rfloor$.  We consider four well-known types of partial sum based test statistics, which are KPSS, R/S, V/S and K/S-type tests built on $(\tilde e_{i,n})$. 

\begin{enumerate}[itemsep=2pt,topsep=0pt,parsep=0pt]
    \item KPSS-type statistic \begin{equation}\label{eq:KPSS}
    T_n = \frac{1}{n(n - 2\lf nb_n\rf)}\sum_{r=\lfloor nb_n \rfloor+1}^{n-\lfloor nb_n\rfloor} \left(\tilde S_{r,n}\right)^2.\end{equation}
    \item R/S-type statistic $
 Q_n = \max_{\lf nb_n \rf + 1 \leq k \leq n - \lf nb_n \rf } 
 \tilde S_{k,n} - \min_{\lf nb_n \rf + 1 \leq k \leq n - \lf nb_n \rf } \tilde S_{k,n}.
$
 \item V/S-type statistic $
  M_n = \frac{1}{n(n - 2\lf nb_n\rf)}\left\{\sum_{k=\lf nb_n \rf + 1 }^{n-\lf nb_n \rf } \tilde S_{k,n} ^2 - \frac{1}{n - 2\lf nb_n \rf}\left(\sum_{k=\lf nb_n \rf + 1}^{n-\lf nb_n \rf } \tilde S_{k,n} \right)^2\right\}.
$
\item K/S-type statistic  $
    G_n =  \max_{\lf nb_n \rf + 1 \leq k \leq n - \lf nb_n \rf } 
 \left|\tilde S_{k,n} \right|.
$
\end{enumerate}
The above four types of tests have been widely applied to the detection of long memory and many other important problems (e.g. unit root testing) for stationary time series. We refer to Section \ref{SectionKPSS} of the online supplement for the complete literature review and applications. 
To the best of our knowledge, all the existing work on the KPSS, R/S, V/S, and K/S tests considers the statistics based on the original series $(e_{i,n})$ or parametric residuals (e.g., the residuals obtained by the removal of the sample mean), and is therefore not applicable to the time-varying coefficient model \eqref{model1}.
Meanwhile, it is well-known that the nonparametric estimators have a slower convergence rate than the corresponding parametric estimators. 
Therefore, the asymptotic properties of our nonparametric residual-based KPSS and related tests will be very different from their parametric residual-based or original series-based counterparts; in fact we show that the tests can degenerate  under certain scenarios (see Theorem \ref{thm:null_dist}) and thus bootstrap procedures adaptive to the possible degeneracy are proposed for implementation, see Algorithms \ref{algorithm}, also Algorithms \ref*{trend_algorithms} and \ref*{algorithms} of the online supplement. In this article, we use the term `{\it KPSS and related tests}' to represent the four types of tests. 

\subsection{Assumptions}
For the sake of brevity, in this section we only discuss the KPSS-type test statistic in detail, and summarize the results of KPSS-related test statistics  in Remark \ref{rm:RS_limit}. In order to investigate the asymptotic properties of $T_n$ defined by  \eqref{eq:KPSS} in the presence of time series covariates, we introduce the following assumptions. 

   \begin{assumption}
    The kernel function $K(\cdot)$ is continuous, symmetric and supported on $[-1,1]$. 
    \label{A:K}
 \end{assumption}

 \begin{assumption}
 Each coordinate of $\bs \beta(\cdot)$, i.e., $\beta_i(\cdot)$ for $i = 1,\cdots,p$, lies in $C^3[0,1]$.\label{A:beta}
 \end{assumption}  

\begin{assumption}\label{Ass-U} 
Let $\mf U(t, \FF_i)  = H(t, \FF_i)\mf W(t, \FF_i)$, $i = 1, \cdots, n$, s.t.
\begin{enumerate}[label=(A\arabic*)] 
  \item $\mf U(t, \FF_i) \in \mathrm{Lip}_2$, $\sup_{t \in [0,1]} \|\mf U(t, \FF_i) \|_4< \infty.$ \label{A:U}
   \item Short-range dependence: $\delta_4(\mf U, k) = O(\chi^k)~ \text{for some} ~\chi \in (0,1)$. \label{A:U_delta}
    \item The smallest eigenvalue of 
    $  {\bs \Sigma}(t):=\sum_{j=-\infty}^{\infty} \mathrm{Cov}\left\{\mathbf{U}\left(t, \mathcal{F}_{0}\right), \mathbf{U}\left(t, \mathcal{F}_{j}\right)\right\}$, $ t\in [0,1],
     $(i.e., the long-run covariance matrix) 
     is bounded away from 0 on $[0,1]$.\label{A:U_lrv}

 \end{enumerate}
\end{assumption}
Assumption \ref{Ass-U} is standard for local linear time series regression, see for instance \cite{zhou2010simultaneous}. Notice that the first element of $\mf U(t, \FF_i)$ is $H(t, \FF_i)$. When $p=1$ (the time-varying trend model), Assumption \ref{Ass-U} reduces to \ref{A:H_lrv} in Assumption \ref{assumptionHp} with $t \in [0,1]$.

Define $\mf M(t) := \mathbb{E} (\mathbf{W}(t,\mathcal{F}_0)\mathbf{W}(t,\mathcal{F}_0)^{\T})$, $t \in [0,1]$. Notice that the first element of $\mf W$ is $1$. Write $\boldsymbol{ \mu}_W(\cdot) := (1, \mu_{W,2}(\cdot), \cdots, \mu_{W,p}(\cdot))^{\T} :=\E(\mf W(t,\FF_0))$. For $p \geq 2$,  Let $x_{i,j, n}$ denote the $j$th element in $\mf x_{i,n}$, and $\mf x_{i,n}^{(-1)} := (x_{i,2,n}, \cdots, x_{i,p,n})^{\T}$. Define $\bs  \mu_W^{(-1)}(\cdot) := (\mu_{W, 2}(\cdot), \cdots, \mu_{W,p}(\cdot))^{\T}$,  $\mf W^{(-1)}(\cdot, \cdot) := (W_2(\cdot, \cdot),\cdots, W_p(\cdot, \cdot))^\top$.

\begin{assumption}\label{Ass-W}   
  The following conditions hold for the covariates when $p\geq 2$:
\begin{enumerate}[label=(B\arabic*),itemsep=2pt,topsep=0pt,parsep=0pt] 
  \item The smallest eigenvalue of
$\mf M(t)$
is bounded away from 0 on $[0, 1]$. \label{A:Mt}
\item {$\mf M(\cdot) \in C^1[0,1]$, $\boldsymbol \mu_W (\cdot) \in C^1[0,1]$}.\label{A:Mt_smooth}
\item  $\mathbf{W}\left(t, \mathcal{F}_{i}\right) \in \mathrm{Lip}_{2}, \text { and } \sup _{t \in [0,1]}\left\|\mathbf{W}\left(t, \mathcal{F}_{i}\right)\right\|_{8}<\infty$, $i = 1, \cdots, n$.\label{A:W}
\item $\delta_{8}(\mf W^{(-1)},k)=O(\chi^k)$~ \text{for some}~ $ \chi \in (0,1)$.\label{A:W_delta}
\item $\E(H(t_j,\FF_j)|\mf W(t_j,\FF_j)) = 0$ for $j = 1,2,\cdots,n,$ a.s..\label{A:HW}
\end{enumerate}
\end{assumption}
  Condition \ref{A:Mt}  ensures that there is no multicolinearity among the explanatory variables.  Assumption \ref{A:Mt_smooth} guarantees that the $\mf M(\cdot)$ and $\bs \mu_W(\cdot)$ have continuous derivatives. Assumption \ref{A:W} requires that the covariates are locally stationary.
 Condition \ref{A:W_delta} imposes that each component of $\mf W^{(-1)}(\cdot, \cdot)$ is SRD. Condition \ref{A:HW} assumes that $\mf x_{i,n}$ is a $p$-dimensional random vector uncorrelated with innovations,  which is necessary for model identification.  Our assumptions are very mild in the sense that we allow nonlinearity and heteroscedasticity for the covariates and errors, as well as the correlation between $\mf x_{i,n}$ and $e_{i,n}$. Assumptions \ref{Ass-U} and \ref{Ass-W} can be verified using similar arguments in \cite{wu2009quantile}. When $p = 1$, we use the convention that $\mf x_{i,n}^{(-1)} = \emptyset$, $\mf W^{(-1)}(\cdot, \cdot)  = \emptyset$, and $\bs \mu_W^{(-1)}(\cdot) = \emptyset$, $\mf M(\cdot)=\bs \mu_W(\cdot)=1$. Therefore Assumption \ref{Ass-W} always hold in this case.
\begin{assumption}\label{nondeg:null}
  $\lambda(\boldsymbol \Sigma^{1/2}(u) \mf{M}^{-1}(u)\boldsymbol \mu^{\T}_W(u) \neq (\sigma_H(u), 0, \cdots, 0)^{\T}) > 0$.
\end{assumption}
The test statistics will degenerate under the null hypothesis when the parameters of the regression model \eqref{model1} lie in  the hyperplane  $\{u\in[0,1]: \boldsymbol \Sigma^{1/2}(u) \mf{M}^{-1}(u)\boldsymbol \mu^{\T}_W(u)= (\sigma_H(u), 0, \cdots, 0)^{\T}\}$. \cref{nondeg:null} excludes such situation.

\subsection{Asymptotic theory}
The following theorem establishes the asymptotic distribution of the KPSS-type statistic $T_n$ \eqref{eq:KPSS} under the null hypothesis.

\begin{theorem}\label{thm:null_dist}
  Let Assumptions \ref{A:K}, \ref{A:beta}, \ref{Ass-U} and  \ref{Ass-W} be satisfied, assuming $nb_n^6 \to 0$, $nb_n^{7/2}/(\log n)^4\to \infty$, we have that under the null hypothesis:
\par
(i) If  Assumption \ref{nondeg:null} holds, 
    \begin{align} \label{Tnlimit}
         T_n\Rightarrow \int_0^1 U^2(t) dt,
    \end{align}
    where $U(t)$ is a zero mean continuous Gaussian process with covariance function
    \begin{align}
      \E(U( r)U(s)) =: \gamma( r,s) &=\int_0^{ r\wedge s}\sigma_H^2 (u) du  - 2 \int_0^{ r\wedge s}\{\boldsymbol \mu^{\T}_W(u)\mf{M}^{-1}(u)\boldsymbol \Sigma^{1/2}(u)\}_1\sigma_H (u) du\notag\\ 
      &+  \int_0^{ r\wedge s} \boldsymbol \mu^{\T}_W(u)\mf{M}^{-1}(u)\boldsymbol \Sigma(u)\mf{M}^{-1}(u)\boldsymbol \mu_W(u) du, \quad  r,s \in [0,1],  
         \end{align}
         where  $\{\cdot\}_1$ denotes the first element of a vector.
  \par (ii) If  Assumption \ref{nondeg:null} fails, then
      $
        s_1^{-1}(T_n/b_n-s_2 )\Rightarrow \chi^2_1, 
     $
      where $s_1$ and $s_2$ are constants, i.e., $s_1 = 2\sigma^2_H(0) \int_0^1 \left(\int_{v-1}^1 K^*(t) dt\right)^2 dv$, and $s_2 = 2\int_{0}^1 \sigma^2_H(t) dt \int_0^1 \left(\int_{v}^1 K^*(t) dt\right)^2 dv$.

\end{theorem}

Theorem \ref{thm:null_dist} reveals that for the time-varying coefficient model with time series covariates, the limiting distribution of  $T_n$ depends on the time-varying mean and covariance matrix of the covariates $\mf x_{i,n}$ as well as the long-run covariance matrix of $\mf x_{i,n} e_{i,n}$. Theorem \ref{thm:null_dist} is very general since it posits neither the specific form of heteroscedasticity nor the parametric form of the errors. When Assumption \ref{nondeg:null} is violated, (ii) shows that  $T_n$  degenerates with asymptotic variance $2s_1^2 b_n^2$. After standardization, it converges to $\chi^2_1$ in distribution. 
 The results of R/S, V/S and K/S follow similarly.  An  important scenario that $T_n$ degenerates under $H_0$ (i.e., $d=0$) is the following time-varying trend model corresponding to $p=1$,  i.e., 
  \begin{equation}
    \begin{aligned}
     y_{i, n}= \beta_1\left(t_i \right) + e_{i, n}, \quad i=1, \ldots, n, ~\text{with}~
     (1-\B)^d e_{i,n} = u_{i,n}, \quad d \in [0,1/2).
    \end{aligned}
    \label{eq:long_memory1}
 \end{equation}

\begin{remark}
 Theorem \ref{thm:null_dist} and other theoretical results in this paper are valid for $p=1$, with $\bs \beta(t)$ replaced by $\beta_1(t)$, $\bs \Sigma(t)$ replaced by $\sigma_H^2(t)$, $\bs \mu_W(t)$ replaced by $1$.
\end{remark}

\subsubsection{Gaussian approximation}
Since the KPSS and related test statistics are constructed based on the partial sum process, to derive the asymptotic properties of the test statistics under the alternative hypothesis, we first study the Gaussian approximation of $\sum_{i=1}^r \mf x_{i,n} e_{i,n}^{(d)}$, $1 \leq r \leq n$, which is the partial sum of the product of a SRD and a LRD time series under $H_A$. 
Though  Gaussian approximation theory for stationary processes (
see for instance
\cite{wu2007strong},
\cite{dehling1989empirical},
 \cite{wu2006invariance} and the reference therein) has been successfully established and widely applied to many fields of statistics, there are only a few results of Gaussian approximation for locally stationary processes. Among them, \cite{wu2011gaussian} established a flexible Gaussian approximation framework for locally stationary SRD processes, which has served as a fundamental key to the inference of SRD (piecewise) locally stationary processes and functional time series, see for instance \cite{chen2015function} and \cite{wu2018gradient}.  
\cite{wu2018} proposed a Gaussian approximation scheme for a class of locally stationary linear LRD processes. However, all the existing Gaussian approximation approaches are not applicable to the partial sum process of the product series  $ (\mf x_{i,n} e_{i,n}^{(d)})$, which serves as the crucial ingredient for establishing the limiting distribution of $T_n$ under $H_A$.
To this end, we shall provide a general Gaussian approximation theorem for the product of LRD and SRD processes. In the remaining of this paper, let $ d_n = c/ \log n$, where $c$ is a positive constant. We substitute $d$ with $d_n$  to differentiate the notation under the fixed alternatives $(d>0)$ and that under the local alternatives $(d=c/\log n)$. 
 
\begin{theorem}\label{thm:fixlocal}
  Under Assumptions \ref{assumptionHp} and \ref{Ass-W}, on a richer probability space, we have:\par
  (i) There exists $\mf R_{k, n}=\sum_{j=0}^{\infty} \boldsymbol \mu_W(t_k)\psi_j(d)\sigma_H \left(t_{k-j}\right) v_{k-j}$, where the random variables $(v_i)_{i \in \mathbb Z}$ are $i.i.d.$ $N(0,1)$, s.t.
   \begin{align}
    \max_{\nb + 1 \leq r \leq n - \nb} \left|\sum_{i = \nb +1}^r (\mf x_{i,n} e_{i,n}^{(d)} - \mf{R}_{i,n}) \right|  = \Op(\sqrt{n}(\log n)^d + n^{1+\alpha_0(d-1/2)}),
  \end{align}
   where $ \alpha_0  \in (1, 4/3) $ and therefore $n^{1+\alpha_0(d-1/2)}  = o(n^{d+1/2})$.
 \par
  (ii) Further under Assumption \ref{Ass-U}, there exists a sequence of Gaussian processes $$\mf{\tilde R}_{i,n} = \sum_{j = 1}^{\infty} \boldsymbol \mu_W(t_i) \psi_{j}(d_n)\sigma_H(t_{i-j}) V_{i - j ,1} + \bs \Sigma^{1/2}(t_i) \mf V_i,$$ where $\mf V_i$, $1 \leq i \leq n$, are $i.i.d.$ $N(\mf 0, \mf I_p)$, s.t.
  \begin{align}
    \max_{\nb + 1 \leq r \leq n - \nb} \left|\sum_{i = \nb +1}^r (\mf x_{i,n} e_{i,n}^{(d_n)} - \mf{\tilde R}_{i,n}) \right|  = \op(n^{1/2}).
  \end{align}
\end{theorem}

Since $\|{\mf R}_{i,n}\| \asymp n^{d+1/2}$ and $\|\tilde {\mf R}_{i,n}\| \asymp n^{1/2}$, the approximation errors of Theorem \ref{thm:fixlocal} (i) and (ii) are asymptotically negligible. It can be also verified that the process $(\mf R_{k,n})_{k=1}^n$ is a locally stationary LRD Gaussian process defined by Definition \ref{def:SRDLRD}.

\begin{remark}
 The results of (i) consist of two parts. The rate $\sqrt{n} (\log n)^d$ is due to  the approximation of $\max_{1\leq s
    \leq n}|\sum_{k=1}^s(\bs x_{k,n} - \bs\mu_W(t_k))e_{k,n}^{(d)}|$, see \cref{prop:5.2}. The rate $n^{1+\alpha_0(d-1/2)}$ is due to  the approximation of $\max_{1\leq s
    \leq n}|\sum_{k=1}^s\bs \mu_W(t_k)e_{k,n}^{(d)}-\mf R_{k,n}|$, which  extends Theorem 2 in \cite{wu2018}, see \cref{lm:alt_gaussian}. Specifically, we allow the driving shocks $(u_{i,n})_{i=-\infty}^n$ to be both dependent and heteroscedastic, while \cite{wu2018} assumed $(u_{i,n})_{i=-\infty}^n$ to be  independent. 
    
\end{remark}
   
It is worth pointing out that (ii) is not an direct consequence of (i). Letting $d=d_n$ in (i), the rate $\sqrt n(\log n)^d$ in (i) will eventually lead to a trivial bound, i.e.,  $\Op(\sqrt{n})$, which is of the same order as the partial sum $\sum_{i = \nb +1}^r \mf x_{i,n} e^{(d_n)}_{i,n}$, $\nb +1\leq r\leq n-\nb $.


Based on the Gaussian approximation result, we proceed to study the limiting distributions of KPSS and related statistics under the fixed and local alternatives for the time-varying coefficient model \eqref{model1} with time series covariates. For the sake of brevity, we focus on the KPSS-type statistics. The results for R/S, V/S, and K/S-type statistics can be derived similarly and are summarized in Remark \ref{rm:RS_limit} and \cref*{sub:limits} in the online supplement.
\subsubsection{Fixed alternatives}\label{fixalternative}
In the following theorem, we establish the asymptotic distribution of the KPSS-type statistic $T_n$ \eqref{eq:KPSS} under the fixed alternatives.
     \begin{theorem}\label{thm:alt_approx}
       Under Assumptions \ref{assumptionHp}, \ref{A:K},  \ref{A:beta} and  \ref{Ass-W},  assuming $nb_n^4/(\log n)^2 \to \infty, n b_n^6 \to 0$, and 
      $\lambda(|\bs \mu_W^{(-1)}(\cdot))| \neq  0) > 0$,
     then we have under $H_A$ with long-memory parameter $d$,
    \begin{align}
      T_{n}\Gamma^2(d+1) /n^{2d } \Rightarrow  \int_0^1 U^2_d(t) dt,
    \end{align}
    where $U_d(t)$ is a zero mean continuous Gaussian process with covariance function
      \begin{align}
       \E(U_d(r)U_d(s)) :=  \gamma_d(r,s) = \int_{-\infty}^{r \wedge s}\sigma^2_H(v)\lambda_d(r,v)\lambda_d(s,v) dv , \quad r, s \in [0,1],
      \end{align}
      where for $v \leq u \in [0,1]$,
      $
        \lambda_d(u,v) =  d\int_{(-v)_+}^{(u-v)_+}t^{d-1}(\check M_W(t+v)-1) dt
       $, 
      $\check M_W(t) = \boldsymbol \mu_W^{\T}(t) \mf{M}^{-1}(t)\boldsymbol \mu_W(t)$, $t \in [0,1]$. 
     \end{theorem}
     Theorem \ref{thm:alt_approx} proves that the test statistic diverges to infinity at the rate of $n^{2d}$.
       Furthermore, from Theorem \ref{thm:alt_approx}, we observe that the limiting distribution of $T_n$ under the fixed alternatives is independent of ${\bs \Sigma}(t)$ except $\sigma^2_H(t)$, which is its $(1,1)$ component, while  under the null hypothesis the limiting distribution relies on all the components of ${\bs \Sigma}(t)$, see  Theorem \ref{thm:null_dist}. This is because when $d>0$, the stochastic fluctuation of the SRD components $(\mathbf x_{i,n})$ is asymptotic negligible compared to that  of $(e^{(d)}_{i,n})$.
     
      Straightforward calculation shows that  $\check M_W(t)\geq 1$. The condition $\lambda(|\bs \mu_W^{(-1)}(\cdot))| \neq  0) > 0$ ensures
     $\check M_W(t)>1$ in an interval with positive length such that $\lambda_d(u,v)>0$ and excludes the scenario that all  the stochastic covariates  are zero mean during the whole period (see  \cref{Remark-degenerate} for detailed discussion of such scenario) as well as the time-varying trend model,i.e., \eqref{eq:long_memory1} with $p=1$. When $\lambda_d(u,v)=0$, we can show that $T_n=\op(n^{2d})$, i.e., the test statistic is degenerate under $H_A$. 

\begin{remark}
  The error model \eqref{Locallyerror} is in fact a Type I fractional $I(d)$ process when $d>0$. If a locally stationary Type II fractional $I(d)$ error process is considered, i.e., $(1- \B)^d e_{i,n} = u_{i,n}\mf 1(i \geq 1)$, the limiting distribution of $T_n$ is almost the same except that the lower bound of the integral in $\gamma_d(r,s)$ is $0$ instead of $-\infty$. We refer to \cite{marinucci1999alternative} for the definition of Type I and Type II fractional $I(d)$ processes. 
\end{remark}

\subsubsection{Local alternatives}
The following theorem presents the asymptotic distribution of the KPSS-type statistic $T_n$ \eqref{eq:KPSS} under the local alternatives.
\begin{theorem}\label{thm:local}
  Let the conditions of  \cref{thm:alt_approx} and Assumption  \ref{Ass-U} hold.
 Then under $H_A$ with $d_n=c/\log n$ for a constant $c>0$, we have 
   \begin{align}
    T_n \Rightarrow \int_0^1 (U^{\circ}(t))^2dt,
   \end{align}
  where $U^{\circ}(t)$ is a zero mean  continuous Gaussian process with covariance function 
  \begin{align}
   \E(U^{\circ}(r)U^{\circ}(s))=: \gamma^{\circ} (r,s) &=  \check \gamma(r, s) + \gamma(r, s) +  2\tilde \gamma(r, s) , \quad r, s \in [0,1],
  \end{align}
  where $\gamma(r, s)$ is defined in \cref{thm:null_dist}, and
  \begin{align}
  &\tilde \gamma(r, s) = (e^c - 1)\int_0^{r\wedge s} \sigma_H(t)(\{\boldsymbol \mu_W^{\T}(t) \mf M^{-1}(t) \bs \Sigma^{1/2}(t)\}_1 - \sigma_H(t))(\check M_W(t) - 1) dt, \\
   &\check \gamma(r, s) = (e^c -1)^2\int_0^{r\wedge s}  \sigma^2_H(t) (\check M_W(t)-1)^2 dt, 
         \end{align}
         where $\{\cdot\}_1$ and $\check M_W(t)$ are as defined in \cref{thm:null_dist} and \cref{thm:alt_approx}, respectively.
\end{theorem}

From \cref{thm:local}, we shall see that under the local alternatives $d_n=c/\log n$, the KPSS-type statistic converges to a  distribution depending on the mean and covariance matrix of $(\mf x_{i,n})$, the long-run covariance matrix of $(\mf x_{i,n}e_{i,n})$, 
as well as the parameter $c$. Careful examination of the proof of Theorem \ref{thm:local} shows that $T_n$ will converge to the limit in Theorem \ref{thm:null_dist} if $d_n\log n=o(1)$, indicating that the exact local power of the KPSS-type test is $O(\log^{-1} n)$ for model \eqref{model1}. For the  time-varying trend model \eqref{eq:long_memory1}, it can be shown that the test statistic is degenerate (since $\lambda(|\bs \mu_W^{(-1)}(\cdot))| \neq  0) =0$) under local alternatives $d_n = c/\log n$, i.e., $T_n=\op(1)$. 
  The results for  R/S, K/S and V/S-type tests follow similarly. 
  \begin{remark}\label{rm:RS_limit}
The limiting  behavior of  R/S, V/S and K/S-type statistics defined in \cref{sec:4} under \cref{nondeg:null} can be derived by Theorems \ref{thm:null_dist}, \ref{thm:alt_approx} and \ref{thm:local} as well as an application of continuous mapping theorem.  Their limiting distributions are functions of $U(t)$, $U_d(t)$, $U^{\circ}(t)$ defined therein,
 see \cref*{sub:limits} in the online supplement for the exact forms. 
\end{remark}
\section{The bootstrap-assisted procedure} \label{sec:diff}\label{sec:bootstrap}
\cref{Test} shows that under the null hypothesis, the limiting distributions of KPSS and related test statistics are functions of the Gaussian process $U(t)$ which involves parameters $\boldsymbol \mu_W(t), \mf M(t),  \bs \Sigma(t)$ (or $\sigma^2_H(t)$ when $p=1$). Furthermore, the magnitude and specific form of the limiting distributions depend on whether \cref{nondeg:null} holds, which is usually unknown in practice. Therefore, it's impossible to obtain the  critical values by directly simulating the Gaussian process $U(t)$. In this section, we provide a consistent bootstrap approach \cref{algorithm} which mimics the asymptotic behavior of $T_n$ under the null hypothesis no matter whether \cref{nondeg:null} is satisfied and yields valid simulated critical values. In addition, the estimation of $\bs \mu_W(\cdot)$ is not required in our proposed bootstrap tests. More precisely, we employ $\mf x_{i,n}^\top \hat{\mf M}^{-1}(t)$ instead of $\hat{\bs \mu}^\top_W(t) \hat{\mf M}^{-1}(t)$, where $\hat{\bs \mu}_W(t)$ stands for any consistent estimator of $\bs \mu_W(t)$, and the validity of the former construction can be easily verified by noting that  in \eqref{eq:Gk} of \cref{algorithm}, the Gaussian multiplier is independent of $ \mf x_{i,n}$, $\hat {\mf M}(\cdot)$ and $\hat {\bs \Sigma}(\cdot)$ where the consistent estimators $\hat{\mf M}(t)$ and $\hat {\bs \Sigma}(t)$ will be discussed later. Thus, the difference between \eqref{eq:Gk} and the counterpart with $\bs \mu_W(t)$ can be controlled by the convolution of standard Gaussian multipliers and the partial sum of the zero mean SRD process containing $(\mf x_{i,n}-\bs \mu_W(t_i))$. We only discuss \cref{algorithm} for the KPSS-type test when $p\geq 2$ in detail in this section.
The other algorithms, including \cref*{trend_algorithms} in the online supplement for $p=1$ case corresponding to the time-varying trend model \eqref{eq:long_memory1},  and \cref*{algorithms}  for R/S, V/S, K/S-type tests are moved to the supplement. 

 
 \begin{algorithm}[!ht]
      \caption{The KPSS-type test for time-varying coefficient models}
      1. Select the smoothing parameters $m$, $b_n$, and $\tau_n$, according to  \cref*{sel}.\\
      2. Calculate $\tilde{e}_{i,n} = y_{i,n}  -\mathbf{x}_{i,n}^{\T}\tilde{{\bs \beta}}(t_i), i = 1, 2, \cdots, n$, 
        using local linear regression \eqref{eq:loclin} and jackknife correction \eqref{eq:jack}. Compute the KPSS-type statistic $T_n$ in \eqref{eq:KPSS}. \\
       3. Calculate $\hat{\mf{M}}(t)$ and $\hat{{\bs \Sigma}}(t)$, the estimators of $\mf M(t)$ and $\bs\Sigma(t)$ defined in \eqref{estiM} and \eqref{eq:diff_correct}, respectively.\\ 
       4. Generate B (say 2000) $i.i.d.$ copies of $N(\mf 0, \mf I_p)$ vectors $\mf V^{(r)}_i=(V^{(r)}_{i,1},...,V^{(r)}_{i,p})^\top$, $1\leq r\leq B$, then calculate (notice that $\hat \sigma^2_{H}(t)=(\hat {\bs \Sigma}(t))_{1,1} $)
        \begin{align}
          \tilde G^{(r)}_k=-\sum_{j=1}^n\left(\frac{1}{nb_n}\sum_{ i=\lf nb_n \rf+1 }^k \mf x_{i,n}^{\T} \hat{\mf{M}}^{-1}(t_i)  K_{b_n}^*(t_i - t_j)\right) \hat {\bs \Sigma}^{1/2}(t_j)\mf{V}^{(r)}_j + \sum_{ i=\lf nb_n \rf+1}^k \hat \sigma_{H}(t_i)V^{(r)}_{i,1},\label{eq:Gk}
        \end{align}
         and the bootstrap version of the KPSS-type statistic \eqref{eq:KPSS}
          \begin{align}\label{eq:bootstrap}
          \tilde T^{(r)}_{n} = \frac{1}{n(n - 2\lf nb_n\rf)}\sum_{s=\lf nb_n \rf + 1}^{ n - \lf nb_n\rf} \left(\sum_{k = \lf nb_n \rf + 1}^s \tilde G^{(r)}_k\right)^2.
        \end{align}
         5. Let $\tilde{T}_{n,(1)} \leq \tilde{T}_{n,(2)} \leq \cdots  \leq \tilde{T}_{n,(B)}$ be the ordered statistics of $\{\tilde{T}_{n}^{(r)}\}_{r=1}^B$.
         Reject $H_0$ at level $\alpha$ if $T_n >\tilde{T}_{n,(\lfloor B(1-\alpha)\rfloor)}$. Let $B^* = \max\{r: \tilde{T}_{n,(r)} \leq T_n\}$. Then the $p$-value of the KPSS-type test is $1-B^*/B$.
    \label{algorithm}
\end{algorithm}

To implement \Cref{algorithm}, we need to obtain the estimators $\hat {\mf M}(t)$ and $\hat {\bs \Sigma}(t)$. For $\hat {\mf M}(t)$ we propose the following estimator that uses directly observed covariates $\mf x_{i,n}$:
\begin{equation}\label{estiM}
    \hat{\mf M}(t) = \frac{1}{n \eta_n}\sum_{i=1}^n   \mathbf{x}_i \mathbf{x}_i^{\T} K_{\eta_n}(t_i - t^*) ,
\end{equation}
where $t^* = \max\{\eta_n,\min(t,1-\eta_n)\}$ for some bandwidth  $\eta_n \to 0$, $n\eta_n^2 \to \infty$.  Under Assumptions \ref{A:K} and \ref{Ass-W}, after a careful investigation of Lemma 6 of \cite{zhou2010simultaneous}, we have $\sup_{t \in[\eta_n,1-\eta_n]}|\hat{\mf{M}}(t)- \mf M(t)|=\op(1)$, i.e. $\hat{\mf{M}}(t)$ is uniformly consistent, see \cref*{lm:scblemma6} in the online supplement for details. \par 
 
 
 Observing that $\bs \Sigma(\cdot)$ depends on the unobserved error process $(e_{i,n})$.  Therefore, a common approach to estimate $\bs \Sigma(\cdot)$ is to utilize $\hat e_{i,n}'s$ from the local linear fit \eqref{eq:loclin}, 
 see for instance \cite{zhou2010simultaneous} and   \cite{vogt2015detecting}. Such plug-in estimate will yield test results that are sensitive to the choices of $b_n$ under $H_0$, and even more sensitive under $H_A$. We also find through our extensive numerical studies (which are not reported in this paper due to limited space) that the power of the test using the plug-in estimator of long-run covariance is often unsatisfactory. Therefore, we adopt a difference-based estimator that does not involves $\hat e_{i,n}'s$.

%


 
  For $p=1$ we recommend the difference statistic proposed in (4.7) of  Section 4.2 of \cite{dette2019detecting} for $\sigma^2_H(t)$ which is built on the difference of $y_{i,n}$. For $p \geq 2$, it can be shown that a direct extension of the estimator based on the difference of $(\mf x_{i,n}y_{i,n}$) is  asymptotically biased. Therefore, we adopt the following bias-corrected difference-based estimator proposed by  \cite{lrvunpublish}. Let $\mf Q_{k, m}=\sum_{i=k}^{k+m-1} \mf x_{i,n}y_{i,n}$, for $t \in [m/n,1-m/n]$, 
\begin{align}
    \bs \Delta_{j}=\frac{\mf Q_{j-m+1, m}- \mf Q_{j+1, m}}{m},\quad \acute{{\bs \Sigma}}(t)=\sum_{j=m}^{n-m} \frac{m \bs \Delta_{j}\bs \Delta_{j}^{\T}}{2}\omega(t, j),\label{eq:diff_based}
\end{align}
where $
      \omega(t, i)= K_{\tau_n}\left(t_i-t\right) / \sum_{i=1}^{n} K_{\tau_{n}} \left(t_i-t\right)
$ for some bandwidth $\tau_n$ and the kernel function $K(\cdot)$ with support $(-1,1)$.
For $t \in [0,m/n)$, set $\acute{{\bs \Sigma}}(t) = \acute{{\bs \Sigma}}(m/n)$. For $t \in (1-m/n,1]$, set $\acute{{\bs \Sigma}}(t) = \acute{{\bs \Sigma}}(1-m/n)$.  The bias-corrected difference-based estimator $\hat {\bs \Sigma}(t)$ for $t \in [0,1]$ is then defined as:
     \begin{align}
         \hat {\bs \Sigma}(t) = \acute {\bs \Sigma}(t) - \breve {\bs \Sigma}(t), ~~~ \text{where}~~\breve {\bs \Sigma}(t)= \sum_{j=m}^{n-m} \frac{m\hat{\mf A}_{j} \hat{\mf A}_{j}^{\T}}{2}\omega(t, j),\label{eq:diff_correct}
     \end{align}
where $
    \hat{\mf A}_{j} = \frac{1}{m}\sum_{i= j-m+1}^{j} (\mf x_{i,n} \mf x_{i,n}^{\T}\breve {\bs \beta}(t_i)-\mf x_{i+m, n} \mf x_{i+m, n}^{\T}\breve {\bs \beta}(t_{i+m}))$, 
    $
        \breve {\bs \beta}(t) = \bs \Omega^{-1}(t)\bs \varpi (t),
$
    where $\bs\Omega(t),\bs \varpi (t)$ are the smoothed versions of $\acute{\bs \Delta}_{j} := \frac{1}{m}\sum_{i=j-m+1}^{j} \tilde{\mf X}_{i,m}\tilde{\mf X}_{i,m}^{\T}$ and $\breve {\bs \Delta}_{j} := \frac{1}{m}\sum_{i=j-m+1}^{j} \tilde{\mf X}_{i,m}^{\T}\tilde{\mf Y}_{i,m}$, i.e., 
    $
         \bs\Omega(t) = \sum_{j=m}^{n-m} \acute {\bs \Delta}_{j}\omega(t^*, j)/2,\nonumber$ and $\bs \varpi (t) = \sum_{j=m}^{n-m} \breve {\bs \Delta}_{j} \omega(t^*, j)/2.
     $
To make \cref{algorithm} operational, it's necessary to select smoothing parameters  $\eta_n$, $\tau_n$ and $m$ for $\hat {\mf M}(t)$ and  $\hat {\bs \Sigma}(t)$.  The selection of smoothing parameters  is postponed to \cref*{sel} of the online supplement. 
\subsection{The limiting behavior of the bootstrap tests}\label{bootstrap_limit}

In this section,  we shall show the  asymptotic correctness of the bootstrap test \cref{algorithm}. 
We shall also prove that
the power of \cref{algorithm}  against $d > 0$
can be no less than $(n/m)^{2d}$,
and that the exact local power of \cref{algorithm} can achieve the order $O(\log^{-1} n)$. \par
 We start by defining the long-run {\it cross} covariance vector between the locally stationary processes $ (\mf U(t,\FF_i))$ and $(H(t,\FF_j))$. 

\begin{definition}\label{def:SUH}
 Define the long-run cross covariance vector $\mf s_{UH}(t) \in \mathbb R^p$ by
  \begin{align}
  \mf s_{UH}(t) = \sum_{j=-\infty}^{\infty} \mathrm{Cov}(\mf U(t, \FF_0), H(t, \FF_j)), \quad t \in [0, 1].\nonumber
  \end{align}
\end{definition}
When $p=1$, $\mf s_{UH}(t)$ degenerates into $\sigma_H^2(t)$.
 We assume the following conditions for $\hat{\bs \Sigma}(t)$ used in \cref{algorithm}. Write $\hat{\bs \Sigma}(t)$ defined in \eqref{eq:diff_correct} as $\hat{\bs \Sigma}_d(t)$ under the fixed alternatives and $\hat{\bs \Sigma}_{d_n}(t)$  under the local alternatives $d_n$. Let $\mathcal{I} = [\gamma_n, 1-\gamma_n] \subset (0,1)$, where $\gamma_{n}=\tau_{n}+(m+1) / n$. 
\begin{assumption}\label{ass:lrv} 
The long-run variance estimator satisfies the following conditions\par
 (i) Under the null hypothesis,
    \begin{equation}
      \sup _{t \in \mathcal{I}}\lt|\hat{{\bs \Sigma}}(t)-{\bs \Sigma}(t)\rt|= \op(b_n/\log^2 n).\nonumber
      \end{equation} \par
       (ii) Under the fixed alternatives,
    \begin{equation}
      \sup_{t \in \I}\left| m^{-2d}\hat{{\bs \Sigma}}_d(t) - {\bs \Sigma}_d(t)\right|  = \op(1),\nonumber
      \end{equation}
      where ${\bs \Sigma}_d(t) = \kappa_2(d)\sigma_H^2(t) \boldsymbol \mu_W(t)\bs \mu^{\T}_W(t)$, $t \in [0,1]$, and $\kappa_2(d) = \Gamma^{-2}(d+1)\int_{0}^{\infty}(t^d - (t-1)_+^d)(2t^d - (t-1)_+^d - (t+1)^d) dt$.\par
       (iii) Under the local alternatives $d_n = c/\log n$, $m = \lf n^{\alpha_1} \rf$, $\alpha_1 \in (0,1)$,
       \begin{align}   
   \sup_{t \in \I}\left|\hat{{\bs \Sigma}}_{d_n}(t) - \check {\bs \Sigma}(t)\right|   = \op(1),\nonumber
 \end{align}
  where $\check {\bs \Sigma}(t) :=  {\bs \Sigma}(t) + (e^{c\alpha_1}-1)^2\sigma_H^2(t) \boldsymbol \mu_W(t)\bs {\mu}^{\T}_W(t) + (e^{c\alpha_1}-1)(\mf s_{UH}(t) \boldsymbol \mu_W^{\T}(t) +\boldsymbol \mu_W(t)\mf s^{\T}_{UH}(t))$.
\end{assumption}
It can be shown that the both the plug-in estimator of \cite{zhou2010simultaneous} and the bias-corrected estimator \eqref{eq:diff_correct} satisfies this condition under suitable bandwidth conditions following Theorem 3.1 of \cite{wu2006invariance} and the chaining argument of Propostion B.1 of  \cite{dette2018change}.  \par Let  $\tilde T_n$  denote the bootstrap statistic \eqref{eq:bootstrap}  generated in one iteration.  Recall the definitions of $U(t)$, $s_1$, $s_2$  in \cref{thm:null_dist}. \cref{thm:bootstrap_null} gives the limiting distributions of bootstrap statistic $\tilde T_n$.

\begin{theorem}[Bootstrap under null]\label{thm:bootstrap_null}
Assume the conditions \ref{A:K}, \ref{A:beta}, \ref{Ass-U}, \ref{Ass-W} and \ref{ass:lrv} hold,  
$nb_n^{7/2}/(\log n)^4\to \infty$, $nb_n^6 \to 0$, $\eta_n \to 0$, $n\eta_n^2 \to \infty$. Then, under the null hypothesis, we have \par
(i) if \cref{nondeg:null} holds, then
$ \tilde T_{n} \Rightarrow \int_0^1 U^2(t) dt$
.\par
(ii) if \cref{nondeg:null} doesn't hold, then 
$ s_1^{-1}(\tilde T_n/b_n-s_2 )\Rightarrow \chi^2_1$.

\end{theorem}

    
    Combining with \cref{thm:null_dist}, \cref{thm:bootstrap_null} indicates that the bootstrap test \cref{algorithm} is asymptotically of level $\alpha$ {\it no matter whether \cref{nondeg:null} is satisfied}. We proceed to investigate the behavior of the bootstrap statistic $\tilde T_n$ under fixed and local alternatives.
Let $\tilde U_d(t)$, $\check U(t)$ be zero mean continuous Gaussian processes, with the covariance structures defined in the same way as that of $U(t)$ in \eqref{Tnlimit}, where $\bs \Sigma(t)$ is replaced by $\bs \Sigma_d(t)$  and $\breve{\bs \Sigma}(t)$ in \cref{ass:lrv}, respectively. Let $\sigma_{Hd}^2(t) := (\bs \Sigma_d(t))_{(1,1)}$, $\check \sigma_{H}^2(t) := (\check{\bs \Sigma}(t))_{(1,1)}.$
  \begin{theorem}[Bootstrap under  alternatives]\label{thm:bootstrapA}
    Under the conditions of \cref{thm:bootstrap_null},\par
(i)   Suppose $\lambda({\bs \Sigma}^{1/2}_d(t)\mf{M}^{-1}(t) \boldsymbol \mu_W(t) \neq (\sigma_{Hd}(t), 0, \cdots, 0)) > 0$ under the fixed alternatives $d>0$. Then, we have
    $$m^{-2d} \tilde T_{n} \Rightarrow \int_0^1 \tilde U^2_d(t) dt,$$ where  
    $\tilde U_d(t)$ is a zero-mean continuous Gaussian process with covariance function 
\begin{align}
     \E(\tilde U_d(r)\tilde U_d(s))
     &= \int_0^{r \wedge s }\boldsymbol \mu_W^{\T}(t) \mf{M}^{-1}(t) {\bs \Sigma}_d(t)\mf{M}^{-1}(t) \boldsymbol \mu_W(t) dt \nonumber
     \\ & - 2 \int_{0}^{r \wedge s } \{\boldsymbol \mu_W^{\T}(t) \mf M^{-1}(t)  {\bs \Sigma}_d^{1/2}(t)\}_1  \sigma_{Hd}(t) dt
     + \int_{0}^{r \wedge s} \sigma_{Hd}^2(t) dt, \quad r,s \in [0,1]. \nonumber
\end{align}

    \par
    (ii) Suppose $\lambda(\check{\bs \Sigma}^{1/2}(t)\mf{M}^{-1}(t) \boldsymbol \mu_W(t) \neq (\check \sigma_{H}(t), 0, \cdots, 0)) > 0$. For the local alternatives $d_n = c /\log n$ with some positive constant $c$, we have
    $$ \tilde T_{n} \Rightarrow \int_0^1 \check U^2(t) dt,$$
     where $\check U(t)$ is a zero-mean continuous Gaussian process with covariance function
     \begin{align}
     \E(\check U(r)\check U(s)) &= \int_0^{r \wedge s } \boldsymbol \mu_W^{\T}(t) \mf{M}^{-1}(t) \check {\bs \Sigma}(t)\mf{M}^{-1}(t) \boldsymbol \mu_W(t) dt \\ &- 2 \int_{0}^{r \wedge s } \{\boldsymbol \mu_W^{\T}(t) \mf M^{-1}(t)  \check {\bs \Sigma}^{1/2}(t)\}_1  \check \sigma_{H}(t) dt 
           + \int_{0}^{r \wedge s} \check \sigma_{H}^2(t) dt, \quad r,s \in [0,1].
     \end{align}


    \end{theorem}
  
\cref{thm:bootstrapA} (i) gives the limiting distribution of $\tilde T_{n}/ m^{2d}$, which means the critical values generated by \cref{algorithm} for a level $\alpha$ test diverges at the rate of $m^{2d}$.  Thus, the bootstrap-assisted test is consistent since \cref{thm:alt_approx} demonstrates that the KPSS-type test statistic $T_n$ diverges at the rate $n^{2d}$ which is much faster than  $m^{2d}$.  Further together with \cref{thm:null_dist}, \cref{thm:bootstrapA} (i) shows that the bootstrap test \cref{algorithm} is  asymptotically correct. Notice that the condition $\lambda({\bs \Sigma}_d(t)\mf{M}^{-1}(t) \boldsymbol \mu_W(t) \neq (\sigma_{Hd}(t), 0, \cdots, 0)) > 0$ prevents the degeneracy of bootstrap statistics. If it is violated, $\tilde T_n= o(m^{2d})$ which will yield higher power than when the condition is fulfilled. 

On the other hand, \cref{thm:bootstrapA} (ii) and \cref{thm:local} indicate that the bootstrap test \cref{algorithm} is able to detect the local alternatives at the rate of $\log^{-1} n$. Observe that under \cref{ass:lrv} (iii) as $c\rightarrow  0$, $|\check{\bs \Sigma}(t)-\bs \Sigma(t)|\rightarrow 0$ and the covariance structure of $\check{U}(t)$ will converge to the covariance structure of $U(t)$. Therefore, the bootstrap test \cref{algorithm} has no power when $d_n=o( \log^{-1} n)$, indicating that the proposed test has the exact local power of $O(\log^{-1} n)$ under the condition of \cref{thm:bootstrapA}. For stationary time series with unknown constant mean, \cite{shao2007local} has also proved that the KPSS test for long memory has the exact local power $O(\log^{-1} n)$.
 
An  important scenario that $T_n$ degenerates under both null and alternatives is the  time-varying trend model \eqref{eq:long_memory1}. For this case we provide the bootstrap-assisted KPSS test in \cref*{trend_algorithms} of the online supplement, which is the $p=1$ version of \cref{algorithm} with $\hat {\mf M}(t) = 1$, and $\hat{\bs \Sigma}(t) = \hat \sigma_H^2(t)$ using the difference statistic proposed by \cite{dette2019detecting}.
\cref{thm:bootstrap_null} (ii) ensures that the level $\alpha$ critical value generated by \cref*{trend_algorithms} converges to the $\alpha_{th}$ quantile of the test statistic $T_n$ under the null hypothesis. Under the alternative hypothesis, the following proposition investigates the power of the test implemented via \cref*{trend_algorithms}.
     \begin{proposition}\label{deg:alt} Let $\tilde T_n$ denote the KPSS-type bootstrap statistic generated from \cref*{trend_algorithms}. Under Assumptions \ref{assumptionHp}, \ref{A:K}, \ref{A:beta}, \ref{ass:lrv}, and the bandwidth conditions 
     $nb_n^4/(\log n)^2 \to \infty,  b_n \to 0$, $nb_n/m \to \infty$, we have the following results: \par
      (i) Under the fixed alternatives $d > 0$,
     $
        \lim_{n \to \infty} P\left( T_n >  \tilde T_n\right) = 1.
  $\par
      (ii) Suppose $m = \lf n^{\alpha_1} \rf$, $nb_n = n^{\beta}$ for some $\alpha_1, \beta \in (0,1)$. Then under local alternatives with 
      $d = d_n = c/\log n$ for a sufficiently large constant $c$,   
      $
        \lim_{n \to \infty} P\left( T_n >  \tilde T_n\right) = 1.
      $
    \end{proposition}
 In addition to the KPSS test, \cref{trend_algorithms} also provides the bootstrap-assisted V/S, R/S and K/S tests for the time-varying trend model \eqref{eq:long_memory1}. Similar conclusions as given in \cref{deg:alt} hold for the power of these tests.

  \begin{remark}\label{Remark-degenerate}
  For $p\geq 2$, $T_n$ degenerates under the fixed alternatives if and only if all the stochastic covariates are of mean zero (i.e., $\mu_W^{(-1)}(t) = 0$), see 
  \cref{fixalternative}. Under this condition, we can show that $T_n/b_n$ diverges at the rate of $(nb_n)^{2d}$. Combining with \cref{thm:bootstrapA} (i), our bootstrap tests \cref{algorithm} and \cref{algorithms}   are consistent if $b_n (nb_n/m)^{2d}\to \infty$. If $\sum_{i \in \mathbb Z} \mathrm{Cov} (H(t, \FF_0), H(t, \FF_i)W^{(-1)}(t, \FF_i)) = 0$, $T_n$ degenerates under $H_0$ and $H_A$ if and only if $\mu_W^{(-1)}(t) = 0$. In this case, one could implement \cref{algorithm} and \cref{algorithms}  for R/S, V/S, K/S-type tests  via modifying the difference-based long-run covariance estimator:
  set $\hat{\bs \Sigma}(t)_{(1,l)} = 0$ and $\hat{\bs \Sigma}(t)_{(l,1)} = 0$  for $l=2, \cdots, p$. Then, by a further investigation of the proof to \cref{thm:bootstrapA},  $\tilde T_n/b_n$ is $\Op(m^{2d})$ and thus the tests are consistent when $m/(nb_n) \to 0$.
  In addition, similar to \eqref{deg:alt}, we can also show that the local power of the tests with the modified estimator is of order $O(\log^{-1} n)$.
      \end{remark}



\section{Finite sample performance}\label{finite}


In the following simulation studies and data analysis, we examine the size and power performance of the bootstrap-assisted KPSS and related tests
\cref{algorithm}, \cref*{trend_algorithms} and \cref*{algorithms} in the online supplement. The number of bootstrap samples is $B = 2000$ and the number of replications is $1000$. The parameters $\mf M(t)$, ${\bs \Sigma}(t)$, $\sigma^2_H(t)$ are estimated by $\hat{\mf M}(t)$, $\hat {\bs \Sigma}(t)$, $\hat \sigma^2_H(t)$ in \cref{sec:diff}, 
with all the smoothing parameters selected by the methods advocated in \cref*{sel} in the online supplement. 
Let
\begin{align}
    \FF_j =( \cdots, \zeta_{j-1}, \zeta_{j}), \quad \mathcal{G}_j = (\cdots, \varepsilon_{j-1}, \varepsilon_{j}),\quad  j=-\infty, \cdots, n, \quad \nonumber
\end{align}
where $(\varepsilon_{l})_{l\in \mathbb Z}, (\zeta_{l})_{l\in \mathbb Z}$ are $i.i.d.$ $ N(0,1)$.
 We consider the following time-varying coefficient model, 
$$
y_{i, n}=\beta_{1}(t_i)+\beta_{2}(t_i) x_{i, n}+e_{i,n}, \quad i=1, \ldots, n,
$$
where $\beta_{1}(t)=4 \sin (\pi t) $, $\beta_{2}(t)=4 \exp \{-2 \left(t-0.5\right)^{2}\} $, $x_{i, n}=W(t_i, \mathcal{F}_{i})$, and $ u_{j,n} = H(t_j,\mathcal{F}_j, \mathcal{G}_j)$.
First, we consider the following independent model:
\begin{enumerate}[label=(\roman*),itemsep=2pt,topsep=0pt,parsep=0pt] 
  \item  
  Let $W\left(t, \mathcal{F}_{i}\right)= (0.25 + 0.25\cos(2\pi t))W(t, \mathcal{F}_{i-1})+ 0.25\zeta_{i} + (t-0.5)^2$
  , $ H(t, \mathcal{G}_{i})=(0.35 - 0.4(t-0.5)^2) H(t, \mathcal{G}_{i-1})+ 0.8 \varepsilon_{i}$. \label{M0}
\end{enumerate}
Second, we consider the following heteroscedastic model:$$
    H(t,\mathcal{F}_{i},\mathcal{G}_{i}) =  B\left(t,\mathcal{G}_{i}\right)\sqrt{1+W^2(t, \mathcal{F}_{i})},$$ where $
     W\left(t, \mathcal{F}_{i}\right)= (0.1 + 0.1\cos(2\pi t))W(t, \mathcal{F}_{i-1})+ 0.2\zeta_{i} + 0.7(t-0.5)^2,
$ and
 $B\left(t,\mathcal{G}_{i}\right)$ is as considered in the following linear and nonlinear scenarios.
\begin{enumerate}[label=(ii.\arabic*),itemsep=2pt,topsep=0pt,parsep=0pt] 
    \item 
    Linear errors: \label{M1}$
         B(t, \mathcal{G}_{i})=(0.3 - 0.4(t-0.5)^2) B(t, \mathcal{G}_{i-1})+ 0.8\varepsilon_{i}.
    $
    \item \label{M2} 
    Nonlinear errors:
    $
        B(t, \mathcal{G}_{i})=(0.15 - 0.4(t-0.5)^2) B(t, \mathcal{G}_{i-1})+ 0.8G(t, \mathcal{G}_{i})$,$ G(t, \mathcal{G}_{i}) = \varepsilon_i \sigma_i(t),
    $
    where $\sigma^2_i(t) = 0.9+0.1\cos(\pi/3 + 2\pi t) + (0.1+0.2t)G^2(t, \mathcal{G}_{i-1})+ (0.1 + 0.2t)\sigma^2_{i-1}(t).$ 
\end{enumerate}
 Observe that models \ref{M1} and \ref{M2} are heteroscedastic models with locally stationary AR(1) and locally stationary GARCH(1,1) errors, respectively. \cref{tb:bn} summarizes the performance of our proposed bootstrap-assisted KPSS, R/S, V/S and K/S-type tests for long memory in models \ref{M1} and \ref{M2} with different $b_n's$. We relegate the simulated sizes of model \ref{M0} with different $b_n's$ to \cref*{tb:m0size1} of the online supplement.
The empirical sizes of all the four tests are close to their nominal levels and are quite stable when $b_n$ changes within a reasonably wide range. Also, \cref*{tb:Tn} in the online supplement reports the simulated Type I error of the proposed tests with respect to increasing sample sizes. As shown in \cref*{tb:Tn}, our  procedures for smoothing parameter selection including GCV and MV selection 
 as well as the difference-based long-run variance estimator work very well in the sense that the simulated sizes of all four tests are quite close to their nominal levels in different sample sizes. 
\begin{table}[!ht]
\centering
 \setlength{\tabcolsep}{3pt}  
\small
\begin{tabular}{ccccccccc|cccccccccc}
  \hline
&   \ref{M1}& & & & & &  &  &  \ref{M2}& & & & & &  & \\ 
\hline
      & \multicolumn{2}{c}{KPSS} & \multicolumn{2}{c}{R/S} & \multicolumn{2}{c}{V/S} & \multicolumn{2}{c|}{K/S}&    \multicolumn{2}{c}{KPSS} & \multicolumn{2}{c}{R/S} & \multicolumn{2}{c}{V/S} & \multicolumn{2}{c}{K/S} \\
\hline
$b_n$ & 5\%         & 10\%       & 5\%        & 10\%      & 5\%        & 10\%      & 5\%        & 10\%     
 & 5\%         & 10\%       & 5\%        & 10\%      & 5\%        & 10\%      & 5\%        & 10\%      \\
  \hline
  0.15 & 4.7 & 8.6 & 5.5 & 9.8 & 6.4 & 10.0 & 4.8 & 10.8 &    4.1 & 9.2 & 4.8 & 9.4 & 5.0 & 9.7 & 3.9 & 8.6 \\ 
  0.175 & 5.5 & 10.4 & 5.5 & 11.1 & 4.8 & 9.6 & 5.1 & 9.3 &    5.1 & 8.9 & 5.0 & 9.3 & 5.4 & 10.4 & 5.3 & 9.9 \\ 
  0.2 & 5.2 & 9.9 & 5.9 & 10.6 & 6.1 & 9.3 & 4.8 & 8.4 &    5.4 & 9.8 & 4.9 & 10.9 & 5.4 & 10.7 & 4.6 & 9.5 \\ 
  0.225 & 5.4 & 10.2 & 4.8 & 9.9 & 4.8 & 9.5 & 5.0 & 8.9 &    5.0 & 9.5 & 5.7 & 10.3 & 4.4 & 8.9 & 4.7 & 10.4 \\ 
\hline
\end{tabular}
 \caption{Simulated sizes (in \%) of KPSS, R/S, V/S and K/S-type tests for model \ref{M1} and \ref{M2} with the sample size $1000$, $m$ and $\tau_n$ determined by MV selection.}
\label{tb:bn}
\end{table}

\begin{figure}
  \centering
  \includegraphics[width= 0.45\linewidth]{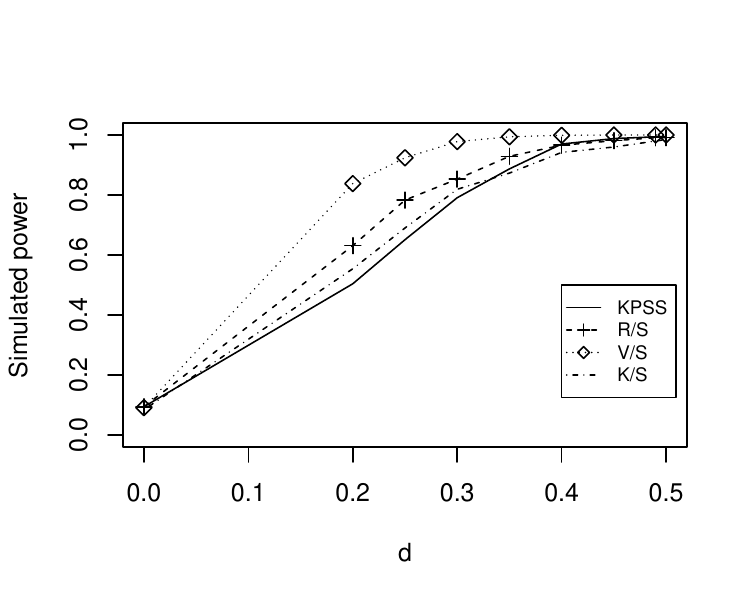}
  \includegraphics[width= 0.45\linewidth]{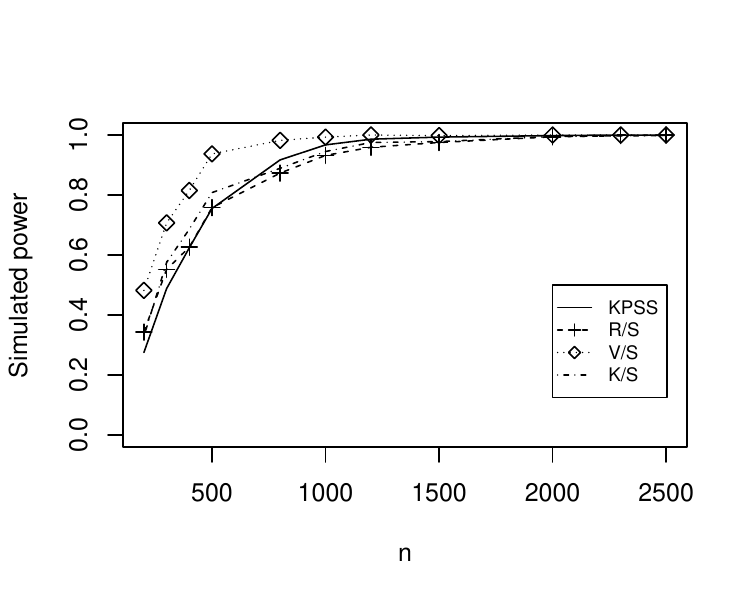}
  \caption{Simulated powers of KPSS and related tests with nominal level $0.1$. 
  Left: $n=1500$ and $d$ increases from $0$ to $0.5$; Right: $d = 0.4$ and $n$ increases from $200$ to $2500$.}
  \label{fig:power2}
\end{figure}

Recall  the long-memory error process of \eqref{Locallyerror}, which can be written as  $e_{i,n}^{(d)} = (1-\B)^{-d}u_{i,n}$ where $\B$ is the lag operator. \cref{fig:power2} displays the power performance of the KPSS and related tests for model \ref{M1} with nominal level $0.1$. 
The left panel reports simulated rejection rates of KPSS, R/S, V/S, and K/S-type tests as the long memory parameter $d$ increases from $0$ to $0.5$ with sample size $1500$. The power of all four KPSS and related tests increases to 1 as $d$ approaches to $1/2$, while the V/S-type test remains the most powerful among all tests. The right panel depicts  the power performance of the four tests as the sample size grows from $200$ to $2500$ when $d = 0.4$. The figure shows that the power of each test increases to $1$ as the sample size grows, and the V/S-type test performs the best among the four tests whenever the sample size is larger than $400$. The power performance of models \ref{M0} and \ref{M2} are shown in  \cref*{fig:power1} and \cref*{fig:power3} of the online supplement where the V/S-type test achieves the best power under most alternatives and sample sizes considered.  

 \subsection{Analysis of Hong Kong hospital data}\label{subsec:HKdata}
 
Hong Kong circulatory and respiratory data  contains daily measurements of pollutants and daily hospital admissions in Hong Kong between January 1st, 1994 and December 31st, 1995. The dataset has been studied by 
\cite{fan2000simultaneous}, 
\cite{zhou2010simultaneous}, and \cite{wu2018gradient} among others. They investigated the relationship between the levels of pollutants and the total number of hospital admissions of circulation and respiration,see \cref*{fig:hk_data} in \cref*{hk} of  the supplement for the observed series. Under the assumption of $i.i.d.$ observations, \cite{fan2000simultaneous} claimed that sulphur dioxide ($\text{SO}_2$) is not significant. 
Using the non-stationary model, \cite{zhou2010simultaneous} found all three pollutants ($\text{SO}_2$, nitrogen dioxide ($\text{NO}_2$) and dust) are significant. In both cases, they assumed the observations were SRD. We shall examine this assumption via the KPSS and related tests.  Consider the following time-varying coefficient model,
 \begin{equation}\label{eq:hk_model}
   y_{i,n}=\beta_{1}(t_i)+\sum_{p=2}^{4} \beta_{p}(t_i) x_{i, p, n}+\varepsilon_{i}, \quad \quad i=1, \ldots, n,
 \end{equation}
 where $(y_{i,n})$ is the series of daily total number of hospital admissions of circulation and respiration and $(x_{i, p, n})$, $p = 2, 3, 4$, represent the series of daily levels  (in micrograms per cubic meter) of $\text{SO}_2$, $\text{NO}_2$ and dust, respectively. The sample size is $n = 2 \times 365 = 730$.
We first investigate whether the covariates $(x_{i, p, n})$, $p = 2, 3, 4$ are SRD series. 

We select the smoothing parameters $b_n$, $\eta_n$, $m$ and $\tau_n$ through the methods provided in \cref*{sel} in the online supplement and summarize the selected parameters in \cref*{tb:hk_par} in \cref*{hk} of the online supplement. We test for long memory in the three pollutants series via KPSS and related tests, presenting the $p$-values in \cref{tb:hk}. For each pollutant series we fail to reject it is SRD at the significance level $0.05$.  

  \begin{table}[ht]
 \centering
 \begin{tabular}{rrrrr|rrrrr}
   \hline
   & KPSS & R/S & V/S & K/S&  & KPSS & R/S & V/S & K/S\\ 
   \hline
   $\text{SO}_2$  & 0.358 & 0.076 & 0.087 & 0.543 & Dust & 0.668 & 0.103 & 0.124 & 0.663\\
  $\text{NO}_2$ & 0.476 & 0.818 & 0.232 & 0.748 &  \eqref{eq:hk_model} & 0.461 & 0.633 & 0.549 & 0.649 \\ 
 \hline
 \end{tabular} 
 \caption{The $p$-values for  $\text{SO}_2$, $\text{NO}_2$, dust and daily total number of hospital admissions modeled by \eqref{eq:hk_model}.}
 \label{tb:hk}
 \end{table}
 To test for long memory in the daily hospital admissions, we consider two approaches. The first approach is to model the hospital admissions by the time-varying trend model \eqref{eq:long_memory1} and implement the KPSS and related tests, i.e., the \cref*{trend_algorithms} in the online supplement. The KPSS test yields p-values $0.004$, and other tests yield p-values smaller than $2 \times 10^{-4}$.  All four tests reject the null hypothesis of short memory at the significance of $0.05$. The second approach is to model the hospital admissions via model \eqref{eq:hk_model} taking into account three pollutants ($\text{SO}_2, \text{NO}_2$ and dust) and apply the KPSS and related tests, i.e., the \cref{algorithm} and \cref*{algorithms} in the online supplement. The large $p$-values in the right panel of \cref{tb:hk} show that the KPSS and related tests fail to reject that the total number of hospital admissions is SRD at the significance level $0.05$. Although both of the two approaches are asymptotically correct, the misspecification of regression models tends to cause `spurious long memory' in finite samples. Therefore, our results conclude that the SRD assumption for \eqref{eq:hk_model} adopted by 
 \cite{fan2000simultaneous}, 
 \cite{zhou2010simultaneous},  \cite{wu2018gradient} and many others is reasonable.

\section{Conclusion and future work}\label{conclude}
 This paper develops bootstrap-assisted KPSS, R/S, K/S, and V/S-type nonparametric tests to detect long memory in time-varying coefficient linear models where the covariates and errors are allowed to be locally stationary and heteroscedastic. Under the null hypothesis, the fixed and local alternatives, we derive the limiting distributions of those test statistics and bootstrap statistics. In particular, we identify the conditions under which  the test statistics degenerate. Although such conditions cannot be directly identified from real data in practice, our proposed bootstrap-assisted KPSS and related tests are always asymptotically correct regardless of such conditions. We also establish the theory of Gaussian approximation to the partial sum process of the product of non-stationary SRD and LRD time series, which is of
separate interest and useful for a large class of problems in the analysis of (time-varying) linear models with LRD errors and SRD covariates.

 A comprehensive Monte Carlo study supports that our proposed KPSS and related tests have good size and power performance in finite samples and are robust to the choices of smoothing parameters. The tests are applied to Hong Kong circulatory and respiratory data, recognizing `spurious long memory' due to the misspecification in the conditional mean. Therefore, our proposed methods can be employed as regression diagnostics.  In \cref*{sec:covid} of the online supplement, we also apply the proposed tests to the COVID-19 data and identify that the time series of log cumulative confirmed cases of Japan and Ireland are LRD, while the time series of log cumulative confirmed deaths of Japan and Ireland are both SRD.

Recent studies have considered LRD models with time-varying long-memory parameter  $d(t)$, see for instance \cite{Detectinglrdependence} and \cite{ferreira2018estimation}. Extra simulations in \cref*{tvd} of the online supplement evidence that our proposed testing procedures are still consistent against the alternatives with time-varying non-negative $d(t)$, i.e., $d(t)>0$ for $t$ in a sub-interval of $[0,1]$. The derivation of the theoretical behavior of the test statistics and the bootstrap procedure with time-varying $d(t)$ is left for rewarding future work.
The extension of the proposed tests to $d \geq 1/2$ or $d<0$ for the null hypothesis (see also \cite{wu2006invariance}, \cite{Ioa2021}) is also challenging and meaningful. 

\appendix
\section{Proof of \texorpdfstring{\cref{thm:fixlocal}}{Theorem 4.2}}\label{sec:appendA}
This subsection provides the proof of \cref{thm:fixlocal}. We first present the following propositions which are needed for the proof. The \cref{prop:5.2} approximates the partial sum  of the product series by the partial sum of a LRD process weighted by the {\it expectation} of a SRD process. In \cref{lm:alt_gaussian}, we establish the Gaussian approximation scheme for the vector partial sum process of $(\boldsymbol \mu_W(t_i) e_{i,n}^{(d)})$. \cref{prop:5.3} is a non-trivial extension of \cref{prop:5.2} under the local alternatives. The proofs of \cref{lm:alt_gaussian} and \cref{prop:5.3} are postponed to  \cref*{sec:appendG} of the online supplements.

  \begin{proposition}\label{prop:5.2}
       Under Assumptions \ref{assumptionHp} and \ref{Ass-W},  we have
       \begin{align}
         \max_{\lf nb_n \rf + 1 \leq r \leq n- \lf  nb_n \rf }\left|\sum_{ i=\lf nb_n \rf+1}^r \mf x_{i,n} e_{i,n}^{(d)}-\sum_{ i=\lf nb_n \rf+1}^r \boldsymbol \mu_W(t_i) e_{i,n}^{(d)}\right| 
         &= \Op (\sqrt{n}(\log n)^d).
         \end{align}
     \end{proposition}
     In \cref{prop:5.2}, the rate $\Op (\sqrt{n}(\log n)^d)$ is due to the complicated dependence between the covariate process and the driving shocks $(u_{i,n})$. When $(\mf x_{i,n})$ and $(e_{i,n}^{(d)})$ are independent, the bound can be further sharpened to $\Op(\sqrt{n})$. The discrepancy manifests  the subtle effect of heteroscedasticity under the fixed alternatives. 
     
     
     \begin{remark}
      The approximation result in \cref{prop:5.2} is frequently considered in the context of regression with LRD errors. With arguments given in the proof of \cref{prop:5.2}, we can show that for any deterministic function $v(\cdot): \mathbb R^p\rightarrow \mathbb R$,  the  {\normalfont partial sum process} $\sum_{ i=\lf nb_n \rf+1}^r  v(\mf x_{i,n}) e_{i,n}^{(d)}$ can be approximated by $\sum_{ i=\lf nb_n \rf+1}^r \E( v(\mf x_{i,n})) e_{i,n}^{(d)}$, i.e.,  \begin{align}\max_{\lf nb_n \rf + 1 \leq r \leq n- \lf  nb_n \rf }\left|\sum_{ i=\lf nb_n \rf+1}^r  v(\mf x_{i,n}) e_{i,n}^{(d)}-\sum_{ i=\lf nb_n \rf+1}^r \E( v(\mf x_{i,n})) e_{i,n}^{(d)}\right| 
      = \Op (\sqrt{n}(\log n)^d).
      \end{align}
     Similar to \cref{prop:5.2}, if $\mf x_{i,n}$ is independent of $e_{i,n}$, the approximation error will be reduced to $\Op(\sqrt n)$.  As a consequence, our result is also in line with the result in  Section 7.2.3 of \cite{BeranLongmemory} for the quantity $\sum_{ i=\lf nb_n \rf+1}^{n-\lf nb_n\rf}  v(\mf x_{i,n}) e_{i,n}^{(d)}$. Furthermore, our proof extends that of Proposition 1 and 2 in \cite{kulik2012conditional} in a non-trivial way, since they focus on the {\normalfont partial sum} assuming   $i.i.d.$ covariates $\{ \mf x_{i,n} \}$ and $i.i.d.$ errors $\{e_{i,n}\}$. \cite{kulik2012conditional} utilized their Propositions 1 and 2 to estimate the conditional variance in the heteroscedastic model \eqref{model1}.
     \end{remark}
     \begin{proposition}\label{lm:alt_gaussian}
     Under Assumptions \ref{assumptionHp} and \ref{Ass-W}, on a possibly richer probability space, there exists a sequence of $i.i.d.$. standard normal $\{v_i\}_{i\in \mathbb Z}$ and $\mf R_{k, n}=\sum_{j=0}^{\infty} \boldsymbol \mu_W(t_k)\psi_j \sigma_H \left(t_{k-j}\right) v_{k-j}$, $1\leq k\leq n$, such that
     \begin{align}
       \max _{1 \leq s \leq n}\left|\sum_{k=1}^{s}\left(\boldsymbol \mu_W(t_k) e_{k,n}^{(d)}-\mf R_{k, n}\right)\right|=O_{p}\left(n^{1+\alpha_0(d-1/2)}\right).
     \end{align}
     where $ \alpha_0  \in (1, 4/3) $ and therefore $n^{1+\alpha_0(d-1/2)}  = o(n^{d+1/2})$.
   \end{proposition}

\begin{proposition}\label{prop:5.3}
Under Assumptions \ref{assumptionHp} and \ref{Ass-W},  we have
  \begin{align}
    \max_{\nb + 1 \leq r \leq n - \nb} \Big|\sum_{i = \nb +1}^r \mf x_{i,n} e_{i,n}^{(d_n)} - \sum_{i= \nb +1}^r \left\{\mf x_{i,n} e_{i,n} + \boldsymbol \mu_W(t_i) (e_{i,n}^{(d_n)} - e_{i,n})\right\} \Big|  = \op(\sqrt{n}).
   \end{align}
\end{proposition}

In the following proofs for the sake of simplicity, we will omit the index $n$ in $e_{i,n}, \mf x_{i,n}, y_{i,n}, u_{i,n}$ and use $\psi_j$ to represent $\psi_j(d)$ when we discuss the fixed alternatives and to represent $\psi_j(d_n)$ for the theory of the local alternatives.
 \subsection{Proof of \texorpdfstring{\cref{thm:fixlocal}}{Theorem 4.2}}
Result  (i) follows from \cref{lm:alt_gaussian} and \cref{prop:5.2}. With regard to (ii),
  observe that
  \begin{align}
  \sum_{i= \nb +1}^r \left\{\mf x_i e_i + {\bs \mu}_W(t_i) (e_i^{(d_n)} - e_i)\right\} = \sum_{i= \nb +1}^r  \mf x_i e_i + \sum_{j=1}^{\infty} \sum_{i= \nb +1}^r{\bs \mu}_W(t_i) \psi_j u_{i-j}.
  \end{align}
  The proof follows from \cref{lm:alt_gaussian},   \cref{prop:5.3} and \eqref{eq:p_gaussian} of the online supplement. \hfill $\Box$

\subsection{Proof of \texorpdfstring{\cref{prop:5.2}}{Proposition G.2}}
 Let $\tilde m = M \log n$,  $\tilde{e}^{(d)}_{i,\tilde m} = \sum_{j= \tilde m+1}^{\infty}\psi_j u_{i-j}$, $\tilde{\mf{x}}_{i,\tilde m} = \E(\mf x_i| \varepsilon_i,\cdots,\varepsilon_{i-\tilde m})$.
 Firstly, we can approximate $\sum_{ i=\lf nb_n \rf+1}^r \mf x_i e_i^{(d)}$ by $\sum_{ i=\lf nb_n \rf+1}^r \tilde{\mf{x}}_{i,\tilde m} e_i^{(d)}$ in that by \eqref{eq:marx} of the online supplement,
  \begin{align}
   \left\| \max_{\lf nb_n \rf + 1 \leq r \leq n- \lf  nb_n \rf }\left| \sum_{ i=\lf nb_n \rf+1}^r  (\mf x_i -  \tilde{\mf{x}}_{i,\tilde m} )e_i^{(d)}\right|  \right\| \leq \sum_{ i=\lf nb_n \rf+1}^{n - \lf nb_n\rf}  \| \mf x_i -  \tilde{\mf{x}}_{i,\tilde m}\|_4 \| e_i^{(d)}\|_4 = O(n\chi^{\tilde m}).
   \label{eq:mlong1}
  \end{align}
  For every fixed $j = 1,2, \cdots, \tilde m$, notice that
   $(\tilde{\mf{x}}_{i,\tilde m} u_{i-j})_{i=1}^n$ is a SRD  sequence similar to $(\mf x_iu_i)_{i=1}^n$. Therefore, we have
  \begin{align}
   \left \|\max_{\lf nb_n \rf + 1 \leq r \leq n- \lf  nb_n \rf }\left|\sum_{ i=\lf nb_n \rf+1}^r \tilde{\mf{x}}_{i,\tilde m} (e_i^{(d)}- \tilde{e}^{(d)}_{i,\tilde m}) \right| \right \| &\leq \sum_{j=0}^{\tilde m} \psi_j \left \|\max_{\lf nb_n \rf + 1 \leq r \leq n- \lf  nb_n \rf } \left|\sum_{ i=\lf nb_n \rf+1}^r \tilde{\mf{x}}_{i,\tilde m} u_{i-j} \right| \right \|\\ & = O(\sqrt{n}\tilde m^d).
   \label{eq:mlong11}
  \end{align}
 For $k \in \mathbb Z$, we define the projection operator as $\proj_k \cdot = \E(\cdot|\FF_k) -\E(\cdot|\FF_{k-1})$. Then,  we have the following decomposition
 \begin{align}
   \sum_{ i=\lf nb_n \rf+1}^r \tilde{\mf{x}}_{i,\tilde m} \tilde e_{i,\tilde m}^{(d)} &= \sum_{l = 0}^{\tilde m} \sum_{i = \lf nb_n \rf + 1}^r \proj_{i-l} \left (\tilde{\mf{x}}_{i,\tilde m} \tilde{e}^{(d)}_{i,\tilde m}\right) + \sum_{i = \lf nb_n \rf + 1}^r \E(\tilde{\mf{x}}_{i,\tilde m} \tilde{e}^{(d)}_{i,\tilde m}|\FF_{i-\tilde m-1}):= T_1 + T_2.
   \label{eq:T1T2}
 \end{align}
 We proceed to show that the $T_2$  is the leading term. Applying Doob's inequality to the martingale $\sum_{i = \lf nb_n \rf + 1}^{r} \proj_{i-l} \left (\tilde{\mf{x}}_{i,\tilde m} \tilde{e}^{(d)}_{i,\tilde m} \right)$, we have
 \begin{align}
   \left \| \max_{\lf nb_n \rf + 1 \leq r \leq n- \lf  nb_n \rf }|T_1 | \right \| 
   & \leq 2 \sum_{l = 0}^{\tilde m}  \left\|\sum_{i = \lf nb_n \rf + 1}^{n -\lf nb_n\rf } \proj_{i-l} \left (\tilde{\mf{x}}_{i,\tilde m} \tilde{e}^{(d)}_{i,\tilde m}\right) \right\|.
 \end{align}
 Let $\tilde{\mf{x}}^{(i-l)}_{i,\tilde m}$, $\tilde{e}^{(d),(i-l)}_{i,\tilde m}$ denote the random variables replacing $\varepsilon_{i-l}$ in $\tilde{\mf{x}}_{i,\tilde m}$ and $\tilde{e}^{(d)}_{i,\tilde m}$ with its $i.i.d.$ copy. We have $\tilde{e}^{(d),(i-l)}_{i,\tilde m} = \tilde{e}^{(d)}_{i,\tilde m}$ for $l\leq \tilde m$,  following from the definition of $\tilde{e}^{(d)}_{i,\tilde m}$. By Jensen's inequality, for $l \leq \tilde m$, 
 \begin{align}
   \Big\|\sum_{i = \lf nb_n \rf + 1}^{n -\lf nb_n\rf } \proj_{i-l} \left (\tilde{\mf{x}}_{i,\tilde m} \tilde{e}^{(d)}_{i,\tilde m}\right) \Big\|^2 &= \sum_{i = \lf nb_n \rf + 1}^{n -\lf nb_n\rf } \left\| \proj_{i-l} \left (\tilde{\mf{x}}_{i,\tilde m} \tilde{e}^{(d)}_{i,\tilde m}\right) \right\|^2 \notag\\ 
   & \leq \sum_{i = \lf nb_n \rf + 1}^{n -\lf nb_n\rf } \left(\lt\|\tilde{\mf{x}}_{i,\tilde m}-\tilde{\mf{x}}^{(i-l)}_{i,\tilde m}\rt\|_4 \lt\|\tilde{e}^{(d)}_{i,\tilde m}\rt\|_4 +  \lt\|\tilde{\mf{x}}^{(i-l)}_{i,\tilde m}\rt\|_4 \lt\|\tilde{e}^{(d)}_{i,\tilde m} - \tilde{e}^{(d),(i-l)}_{i,\tilde m}\rt\|_4 \right)^2 \notag\\ 
   & = \sum_{i = \lf nb_n \rf + 1}^{n -\lf nb_n\rf } \left(\lt\|\tilde{\mf{x}}_{i,\tilde m}-\tilde{\mf{x}}^{(i-l)}_{i,\tilde m}\rt\|_4 \lt\|\tilde{e}^{(d)}_{i,\tilde m} \rt\|_4 \right)^2  
   = O(n \chi^{2l})\notag.
 \end{align}
 Hence, we have 
 \begin{align}
   \left \| \max_{\lf nb_n \rf + 1 \leq r \leq n- \lf  nb_n \rf }|T_1 | \right \| = O(\sqrt{n}).
   \label{eq:T1}
 \end{align}
  Since $\tilde e_{i,\tilde m}^{(d)}$ is $\FF_{i-\tilde m-1}$ measurable and $\tilde {\mf x}_{i,\tilde m} $ is independent of $\FF_{i-\tilde m-1}$, we have 
 \begin{align}
   T_2 = \sum_{ i=\lf nb_n \rf+1}^r \E(\tilde{\mf{x}}_{i,\tilde m}|\FF_{i-\tilde m-1}) \tilde{e}^{(d)}_{i,\tilde m} 
   =  \sum_{ i=\lf nb_n \rf+1}^r {\bs \mu}_W(t_i) \tilde{e}^{(d)}_{i,\tilde m}.
   \label{eq:T2}
 \end{align}
 Therefore, combining the results from \eqref{eq:mlong1}, \eqref{eq:mlong11}, \eqref{eq:T1T2}, \eqref{eq:T1} and \eqref{eq:T2}, we have 
 \begin{align}
   \left\| \max_{\lf nb_n \rf + 1 \leq r \leq n- \lf  nb_n \rf }\left| \sum_{ i=\lf nb_n \rf+1}^r \mf x_i e_i^{(d)} -  \sum_{i = \lf nb_n \rf + 1}^r {\bs \mu}_W(t_i) \tilde{e}^{(d)}_{i,\tilde m}\right| \right\| = O(\sqrt{n}\tilde m^d + n\chi^{\tilde m} + \sqrt{n}) = O(\sqrt{n}\tilde m^d).
 \end{align}
 Finally, since $\mathcal{P}_{k} \cdot=\mathbb{E}\left(\cdot \mid \mathcal{F}_{k}\right)-\mathbb{E}\left(\cdot \mid \mathcal{F}_{k-1}\right)$ and $e_i^{(d)} - \tilde{e}^{(d)}_{i,\tilde m} = \sum_{j = 0}^{\tilde m} \psi_{j} u_{i-j}$, it follows from Doob's maximal inequality and Burkholder's inequality that 
 \begin{align}
   &\left \| \max_{\lf nb_n \rf + 1 \leq r \leq n- \lf  nb_n \rf }\left|\sum_{ i=\lf nb_n \rf+1}^r {\bs \mu}_W(t_i) (e_i^{(d)} - \tilde{e}^{(d)}_{i,\tilde m})\right| \right\| \notag\\ 
   &\leq C_1 \sum_{l=0}^{\infty}  \sum_{j=0}^{\tilde m} \psi_j  \left \|\sum_{ i=\lf nb_n \rf+1}^{n - \lf nb_n\rf} {\bs \mu}_W(t_i)\proj_{i-l} u_{i-j} \right \|\notag\\ 
   & \leq  C_2 \sqrt{n} \left( \sum_{l=0}^{\tilde m}  \sum_{j=0}^l \psi_j \delta_2(H, l-j,  (-\infty, 1]) + \sum_{l=\tilde m+1}^{\infty}  \sum_{j=0}^{\tilde m} \psi_j \delta_2(H, l-j, (-\infty, 1] ) \right) \notag\\ 
  &= O(\sqrt{n} \tilde m^d)
   \label{eq:mlong2},
 \end{align}
 where $C_1$, $C_2$ are sufficiently large constants. The last inequality of \eqref{eq:mlong2} follows in that for $k < 0$, $\delta_2(H,k, (-\infty, 1]) = 0$, and the big $O$ follows from Lemma 3.2 in \cite{KOKOSZKA199519}.
 \hfill $\Box$



\section*{Acknowledgement} Weichi Wu is the corresponding author and gratefully acknowledges NSFC Young Program (no.
11901337) of China.

\newpage
\begin{center}
   {\Large \bf Supplement to ``Detecting long-range dependence for time-varying  linear models"}
\end{center}

We organize the supplementary material as follows: \cref{SectionKPSS} contains a literature review of KPSS and related test. \cref{sec:imple} provides the implementation details including the selection procedures for the smoothing parameters $m, b_n, \tau_n$ and $\eta_n$ in the bootstrap tests. \cref{sim} reports additional simulations of KPSS and related tests. \cref{sec:covid} analyses the COVID-19 dataset. \cref{hk} displays some details in analyzing Hong Kong circulatory and respiratory data. \cref{sec:algorithm} provides algorithms of KPSS and related tests under the time-varying trend model and  R/S, V/S and K/S-type tests under the time-varying coefficient model. \cref{sec1} gives  the limiting distributions of  R/S, V/S and K/S-type statistics under \cref{nondeg:null} and the proofs of the results in Sections \ref{sec:model} and \ref{Test} of the main article. In \cref{bootstrap}, we justify the proposed bootstrap procedures and offer the proofs of the results in Section \ref{sec:bootstrap} of the main article.\par
 Recall filtration $\FF_i=(\varepsilon_{-\infty},...,\varepsilon_i)$ for $i.i.d.$ random variables $(\varepsilon_{i})_{i\in \mathbb Z}$ and the projection operator $\proj_k \cdot = \E(\cdot|\FF_k) -\E(\cdot|\FF_{k-1})$ for $k \in \mathbb Z$.
Recall that $ e_{i,n}^{(d)} = \sum_{j=0}^{\infty}\psi_j(d) u_{i-j, n}$, $ e_{i,n}^{(d_n)} = \sum_{j=0}^{\infty}\psi_j(d_n) u_{i-j, n}$. For the sake of simplicity, we use $\psi_j$ to represent $\psi_j(d)$ when we discuss the fixed alternatives and $\psi_j(d_n)$ for the theory of the local alternatives. Recall $t_i = i/n$, and that $K^*(x)$ denotes the jackknife equivalent kernel $2 \sqrt{2} K(\sqrt{2}x) - K(x)$.  Let "$\Rightarrow$" denote weak convergence, and "$\leadsto$" denote the convergence of a process. 
Let $0 \times \infty = 0$, $a_n \sim b_n$ denote $\lim_{n \to \infty} a_n/b_n = 1$ for real sequences $a_n$ and $b_n$. For a random variable $X$ and a distribution $G$, $X \sim G$ is denoted by $X$ follows the distribution $G$.
Let $D[0,1]$ be the space of real functions on $[0,1]$ that are right-continuous and have left-hand limits (also named càdlàg functions). In the following proofs, we will omit the index $n$ in $e_{i,n}, \mf x_{i,n}, y_{i,n}, u_{i,n}$ for simplicity.

\appendix
\setcounter{section}{1}
\section{KPSS and related tests}\label{SectionKPSS}
The first test statistic is the KPSS-type statistic $T_n$ defined by \begin{equation}
    T_n = \frac{1}{n(n - 2\lf nb_n\rf)}\sum_{r=\lfloor nb_n \rfloor+1}^{n-\lfloor nb_n\rfloor} \left(\tilde S_{r,n}\right)^2.
\end{equation}
The KPSS test was first introduced by \cite{kwiatkowski1992testing} to test for the unit root in level  and trend stationary series, complementary to the ADF test. 
Besides the unit root problem, KPSS-type statistics have been widely and successfully applied to many important hypothesis testing problems, including testing for long memory against the null hypothesis of short memory, see for instance  \cite{lee1996power}. 
The same statistic was also used for detecting structural changes (see \cite{MacNeill1974} among others) and examining random walk components in functional time series (see \cite{kokoszka2016kpss}).
The exhaustive account of the applications of KPSS-type tests is almost impossible and we have only listed a small fraction here.   
\par 
The second test statistic $Q_n$ is the 
R/S-type statistic defined by 
\begin{align}
 Q_n = \max_{\lf nb_n \rf + 1 \leq k \leq n - \lf nb_n \rf } 
 \tilde S_{k,n} - \min_{\lf nb_n \rf + 1 \leq k \leq n - \lf nb_n \rf } \tilde S_{k,n}.\label{Qn}
\end{align}
The R/S test was first introduced by \cite{hurst1951long}. \cite{lo1989long} proposed a modified R/S test for long memory, which is robust to the short-range dependence of strictly stationary errors under null. Lo's test \cite{lo1989long} has been widely applied in finance, see \cite{cheung1993gold} and many others.

The third test statistic is the V/S-type statistic $M_n$ defined by 
\begin{align}
  M_n = \frac{1}{n(n - 2\lf nb_n\rf)}\left\{\sum_{k=\lf nb_n \rf + 1 }^{n-\lf nb_n \rf } \tilde S_k ^2 - \frac{1}{n - 2\lf nb_n \rf}\left(\sum_{k=\lf nb_n \rf + 1}^{n-\lf nb_n \rf } \tilde S_k \right)^2\right\}.\label{Mn}
\end{align}
The V/S test is proposed by \cite{giraitis2003rescaled}, where the authors found that the V/S test achieved better size and power performance than R/S and KPSS tests when applied to certain financial data.

The fourth test statistic is the K/S-type statistic $G_n$ defined by
\begin{align}
    G_n =  \max_{\lf nb_n \rf + 1 \leq k \leq n - \lf nb_n \rf } 
 \left|\tilde S_k \right|.\label{Gn}
\end{align} 
\cite{lima2004robustness} used the K/S statistic to test for long-range dependence and argued that its behavior under the alternative of long memory hypothesis was similar to that of R/S test.

\section{Implementation details}\label{sec:imple}\label{sel}


   In this section, we discuss the selection of proper smoothing parameters $m, b_n, \tau_n$ and $\eta_n$ for the implementation of the bootstrap tests \cref{algorithm}, \cref{trend_algorithms} and \cref{algorithms}. The criterion for choosing parameters for  \cref{trend_algorithms} is to let $p=1$ in the following schemes.
   
   To select $b_n$,  we adopt the Generalized Cross Validation (GCV) proposed by \cite{craven1978smoothing}. For the estimation of ${\bs \beta}(\cdot)$, we can write $\hat{\mf Y}(b) = \mf Q(b)\mf Y$ for some square matrix $\mf Q$, where $\mf Y = (y_{1,n}, \cdots, y_{n,n})^{\T}$, and $\hat{\mf Y}(b)=(\hat y_{1,n},...,\hat y_{n,n})^\top$ is the estimated value of $\mf Y$ via the bandwidth $b$, i.e., $\hat y_{i,n}=\mf x_{i,n}^\top \hat {\bs \beta}(t_i)$. Then we select $\hat b_n$ by
   \begin{equation}
     \hat{b}_{n}=\underset{b\in [b_L^*,b_U^*]}{\arg \min }\{\operatorname{GCV}(b)\}, \quad \quad  \operatorname{GCV}(b)=\frac{n^{-1}|\mathbf{Y}-\hat{\mathbf{Y}}|^{2}}{[1-\operatorname{tr}\{\mf Q(b)\} / n]^{2}},\label{GCV}
     \end{equation}
where the selection range $[b_L^*,b_U^*]$ are chosen as follows. The theoretical optimum bandwidth $b_n$ for the local linear estimation \eqref{eq:loclin}, as discussed in \cite{zhou2010simultaneous}, is 
   \begin{equation}
b_{n}^{*}=\left[\frac{\phi_{0} \int_{0}^{1} \operatorname{tr}\{{\bs \Sigma}(t)\} \mathrm{d} t}{\mu_{2}^{2} \int_{0}^{1}\left|\boldsymbol{\beta}^{\prime \prime}(t)\right|^{2} \mathrm{~d} t}\right]^{1 / 5} n^{-1 / 5}:=cn^{-1/5},\nonumber
\end{equation}
where $\mu_{2}=\int_{\mathbb{R}} x^{2} K(x) \mathrm{d} x $ and $ \phi_{0}=\int_{\mathbb{R}} K^{2}(x) \mathrm{d} x$. Let $b_n =  n^{-1/5}$, then we obtain the pilot estimator $\tilde{\bs \beta}^{\prime}(t)$ via the local linear estimation \eqref{eq:loclin}. Next letting $m = \lf n^{4/15} \rf$, $\tau_n = n^{-5/29} $, we obtain
the pilot estimator of $\bs \Sigma(t)$ via the difference-based approach in \cref{sec:diff}.
Thus for Epanechnikov kernel, the lower and upper bound of $b_n^{*}$ is given by 
\begin{align}
b_{L}^{*} =\hat  c n^{-1 / 4}, \quad b_{U}^{*} = \hat c n^{-1 / 6},\quad \hat c = \left[\frac{15 \sum_{i=1}^n \operatorname{tr}\{\hat{\bs \Sigma}(i /n)\}}{\sum_{i=\nb + 2}^{n - \nb}\left| \tilde {\bs \beta}^{\prime}(t_i)-\tilde {\bs \beta}^{\prime}(t_{i-1})\right|^{2} }\right]^{1 / 5}.
\end{align}
For the choice of $m$ and $\tau_n$ for estimating $\bs \Sigma(t)$ and $\sigma_H(t)$ in  \cref{sec:diff},  as a rule of thumb, we can simply choose $m^* = \lf n^{4/15} \rf$,$\tau_n^* =  n^{-5/29}$  under which \cref{ass:lrv} holds. We refer to the results of   \cite{lrvunpublish}. 
For refinement, we recommend the following extended minimum volatility (MV) method as proposed in Chapter 9 of \cite{politis1999subsampling} which works quite well in our empirical studies. The MV method has the advantage of robustness under complex dependence structures and does not depend on any parametric assumptions of the time series. To be concrete, we first propose a grid of possible block sizes and bandwidths $\{m_1, m_2,\cdots, m_{M_1}\}$, $\{\tau_1, \tau_2,\cdots, \tau_{M_2}\}$. Define the sample variance $s^2_{m_i,\tau_j}(t)$ of the bootstrap statistics as 
\begin{align}
    s^2_{m_i,\tau_j}(t) = \frac{1}{99} \sum_{i=1}^{100} \lt(\tilde T_{n, (i)} - \bar{\tilde T}_{n}\rt)^2,\nonumber
\end{align}
where $\tilde T_{n, (1)},..., \tilde T_{n,(100)}$ are the bootstrap statistics calculated  from 100 iterations of \cref{algorithm} with parameters $b_n$, $m_i$ and $\tau_j$, and $\bar{\tilde T}_{n} = \sum_{i=1}^{100} \tilde T_{n, (i)}/100. $
Then calculate 
   \begin{align}
    MV(i,j):= \underset{1 \leq k \leq n}{\max} \mathrm{SE}\lt\{\cup_{r=-1}^{1}\{s^2_{m_{i}, \tau_{j+r}}(t_k)\} \cup \cup_{r=-1}^{1}\{s^2_{m_{i+r}, \tau_{j}}(t_k)\}\rt\},\nonumber
    \end{align} where $\mathrm{SE}$ stands for the standard deviation, i.e. the maximand is
    \begin{align}
    \frac{1}{4}\left\{\sum_{r=-1,1}\left(s^2_{m_i,\tau_{j+r}}(t_k)-\overline{s^2}_{i,j}(t_k) \right)^2+\sum_{r=-1,1}\left(s^2_{m_{i+r},\tau_{j}}(t_k)-\overline{s^2}_{i,j}(t_k)\right)^2+\left(s^2_{m_i,\tau_{j}}(t_k)-\overline{s^2}_{i,j}(t_k) \right)^2\right\}^{1/2},
    \nonumber
    \end{align}
    where \begin{align}
        \overline{s^2}_{i,j}(t) = \frac{1}{5}\left(\sum_{r=-1,1}s^2_{m_i,\tau_{j+r}}(t) +\sum_{r=-1,1}s^2_{m_{i+r} ,\tau_{j}}(t)+s^2_{m_i,\tau_j}(t) \right).\nonumber
    \end{align} Then we select the pair $(m_{i^*},\tau_{j^*})$ where $(i^*,j^*)$ minimizes $MV(i,j)$. Finally, for $\eta_n$, as a rule of thumb, we recommend setting $\eta_n = b_n$, which works reasonably well in our Monte Carlo experiments. The choice of $\eta_n$ can be also refined by MV methods. Specifically, we can first propose a grid of possible bandwidths $\{\eta_1, \cdots, \eta_{M_3}\}$. Denoted by $\hat{\mf M}_{\eta_{i}}(t)$ the estimated covariance matrix via \eqref{estiM} using $\eta_i$,   $i = 1, 2 \cdots, M_3$,  and select  $\eta=\eta_{j^*}$ where $j^*$ is the  minimizer of the following criterion $V^\diamond(i)$,
    \begin{align}
     V^\diamond(i)=   \underset{1 \leq k \leq n}{\max}  \sum_{r = -2}^2\left| \hat{\mf M}_{\eta_{i+r}}(t_k)  -  \bar{ \hat{\mf M}}_{\eta_{i}}(t_k) \right|^2,\nonumber
    \end{align}
    where $\bar{ \hat{\mf M}}_{\eta_{i}}(t_k) = \sum_{r = -2}^2 \hat{\mf M}_{\eta_{i+r}}(t_k)/5$.

\section{Additional Simulations}\label{sim}

\subsection{Simulation results of the independent model \ref{M0}}
This subsection contains the simulation results on the sensibility of simulated sizes on sample sizes of model \ref{M0}, \ref{M1} and \ref{M2}(see \cref{tb:Tn}) and those of independent model \ref{M0}, including simulated sizes (see \cref{tb:m0size1}) and powers (see \cref{fig:power1}) with the selection procedure described in \cref{sel} in the main article.

\begin{table}[!ht]
\centering
\setlength \tabcolsep{4pt}
\small
\begin{tabular}{lrrrrrrrrrrrrrrrrr}
\hline
$n=$ &  \multicolumn{2}{l}{1000}& & & & &  &  &  &\multicolumn{2}{l}{1500}& & & &  & \\ 
\cline{2-9} \cline{11-18}
      & \multicolumn{2}{l}{KPSS} & \multicolumn{2}{l}{R/S} & \multicolumn{2}{l}{V/S} & \multicolumn{2}{l}{K/S}&   & \multicolumn{2}{l}{KPSS} & \multicolumn{2}{l}{R/S} & \multicolumn{2}{l}{V/S} & \multicolumn{2}{l}{K/S} \\
\cline{2-3} \cline{4-5}\cline{6-7}\cline{8-9}
\cline{11-12} \cline{13-14}\cline{15-16}\cline{17-18}
Model & 5\%         & 10\%       & 5\%        & 10\%      & 5\%        & 10\%      & 5\%        & 10\%     
& $\quad$  & 5\%         & 10\%       & 5\%        & 10\%      & 5\%        & 10\%      & 5\%        & 10\%      \\
 \hline
\ref{M0} & 5.8 & 11.5 & 6.9 & 10.5 & 5.8 & 10.6 & 4.2 & 7.9 &   & 6.0 & 11.3 & 4.7 & 9.0 & 5.8 & 9.2 & 4.6 & 10.0 \\ 
  \ref{M1} & 4.6 & 9.3 & 3.9 & 7.4 & 5.1 & 9.4 & 4.3 & 9.2 &   & 4.9 & 10.5 & 5.4 & 10.7 & 4.6 & 9.6 & 4.6 & 9.2 \\ 
  \ref{M2} & 4.7 & 10.0 & 5.6 & 9.9 & 5.1 & 9.8 & 5.4 & 9.9 &   & 5.0 & 10.2 & 4.9 & 10.0 & 4.7 & 9.9 & 4.2 & 9.3 \\ 
  \hline
\end{tabular}
\caption{Simulated Type I errors (in \%) of KPSS, R/S, V/S and K/S-type tests for model \ref{M0}, \ref{M1} and \ref{M2}. The bandwidths $m$ and $\tau_n$ are determined by MV selection. The bandwidth $b_n$ is selected by GCV.}
\label{tb:Tn}
\end{table}

\begin{table}[!ht]
  \centering
  \begin{tabular}{lrrrrrrrr}
    \hline
     & \multicolumn{2}{c}{KPSS} & \multicolumn{2}{c}{R/S} & \multicolumn{2}{c}{V/S} & \multicolumn{2}{c}{K/S} \\
\cmidrule(r){2-3} \cmidrule(r){4-5}\cmidrule(r){6-7}\cmidrule(r){8-9}
$b_n$ & 5\%         & 10\%       & 5\%        & 10\%      & 5\%        & 10\%      & 5\%        & 10\%      \\
\hline
  0.15 & 4.6 & 9.0 & 5.8 & 10.5 & 4.9 & 9.3 & 4.2 & 8.9 \\ 
  0.175 & 4.4 & 10.1 & 5.5 & 9.6 & 5.8 & 11.1 & 5.5 & 9.7 \\ 
  0.2 & 5.7 & 10.1 & 4.8 & 9.6 & 5.5 & 11.0 & 5.8 & 10.6 \\ 
  0.225 & 5.2 & 10.5 & 5.8 & 10.9 & 5.4 & 8.8 & 4.5 & 10.0 \\ 
     \hline
  \end{tabular}
  \caption{Simulated sizes (in \%) of KPSS, R/S, V/S and K/S tests for model \ref{M0}. The bandwidths $m$ and $\tau_n$ are determined by MV selection.}
     \label{tb:m0size1}
  \end{table}

\begin{figure}[!ht]
  \centering
  \includegraphics[width= 0.45\linewidth]{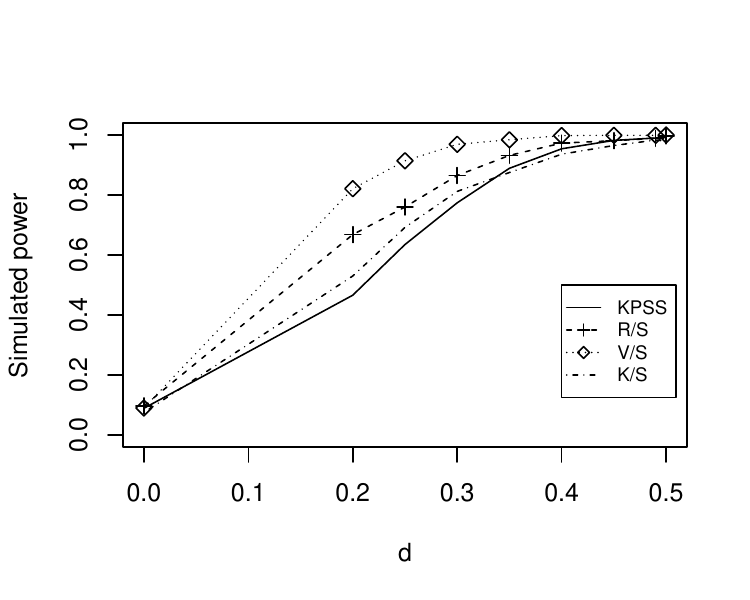}
  \includegraphics[width= 0.45\linewidth]{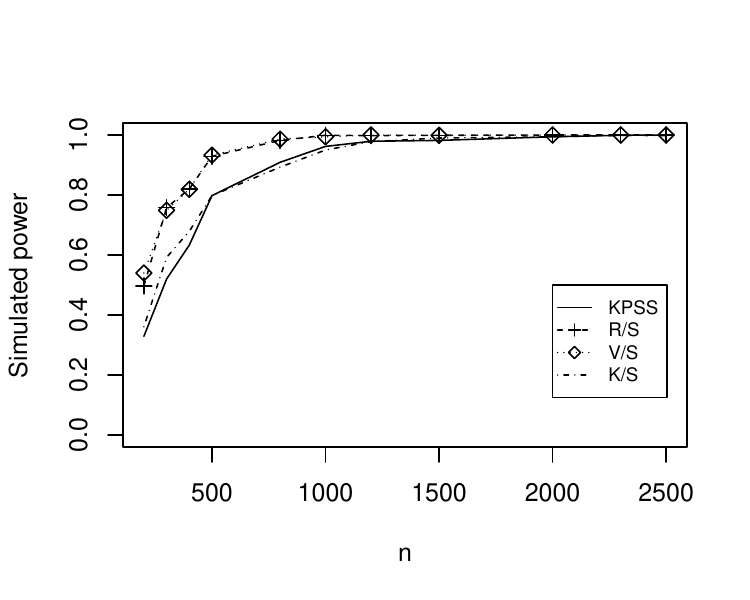}
  \caption{Simulated powers of KPSS and related tests for \ref{M0} with nominal level $0.1$. Left: $n=1500$ and $d$ increases from $0$ to $0.5$; Right: $d = 0.4$ and the sample size $n$ increases from $200$ to $2500$.}
  \label{fig:power1}
\end{figure}

\subsection{Simulation results with heavy-tail innovations}\label{heavy}
This subsection contains the simulation results on the sensibility of simulated sizes on sample sizes of model  \ref{M1} with heavy-tail $i.i.d.$ innovations $\varepsilon_i \sim t_8$.

\begin{table}[ht]
\centering
\ \begin{tabular}{lrrrrrrrr}
    \hline
     & \multicolumn{2}{c}{KPSS} & \multicolumn{2}{c}{R/S} & \multicolumn{2}{c}{V/S} & \multicolumn{2}{c}{K/S} \\
\cmidrule(r){2-3} \cmidrule(r){4-5}\cmidrule(r){6-7}\cmidrule(r){8-9}
$b_n$ & 5\%         & 10\%       & 5\%        & 10\%      & 5\%        & 10\%      & 5\%        & 10\%      \\
\hline
  0.15 & 5.1 & 10.6 & 5.0 & 10.1 & 5.0 & 9.0 & 4.5 & 9.7 \\ 
  0.175 & 4.9 & 10.1 & 4.4 & 8.5 & 5.6 & 9.9 & 4.7 & 9.7 \\ 
  0.2 & 5.0 & 9.9 & 5.0 & 9.3 & 6.1 & 10.4 & 5.2 & 9.6 \\ 
  0.225 & 5.5 & 10.4 & 5.2 & 10.3 & 6.4 & 12.2 & 4.8 & 10.9 \\ 
   \hline
\end{tabular}
\caption{Simulated sizes (in \%) of KPSS, R/S, V/S and K/S tests for model \ref{M1} with $i.i.d.$ $t_8$ innovations $\varepsilon_i$. The bandwidths $m$ and $\tau_n$ are determined by MV selection.}
\end{table}

\subsection{Power performance of model \ref{M2}}
 
\cref{fig:power3} depicts the power performance under data generating model \ref{M2}. As shown in the left panel, when the long memory parameter increases, the rejection rates of all the four KPSS and related tests grow to 1. The right panel reports the power performance as the sample size increases. It implies that KPSS and related tests have asymptotic power 1. 

\begin{figure}[!ht]
  \centering
  \includegraphics[width= 0.45\linewidth]{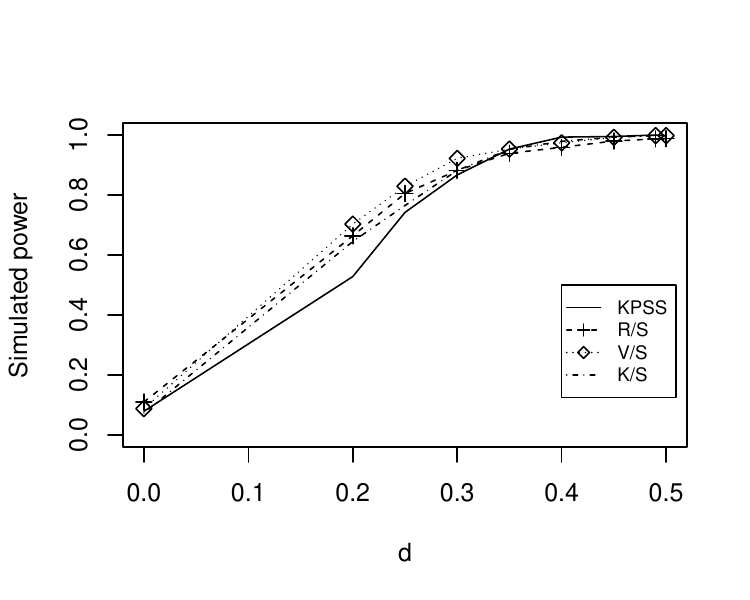}
  \includegraphics[width= 0.45\linewidth]{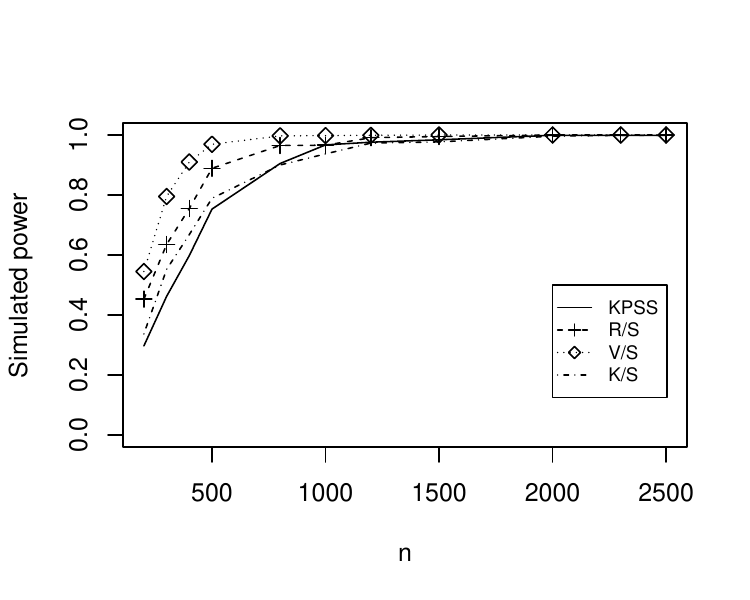}
  \caption{Simulated powers of KPSS and related tests for \ref{M2} with nominal level $0.1$. Left: $n=2500$ and $d$ increases from $0$ to $0.5$; Right: $d = 0.4$ and the sample size $n$ increases from $200$ to $3000$.}
  \label{fig:power3}
\end{figure}

\newpage

\subsection{Simulation results of time-varying \texorpdfstring{$d$}{d}}\label{tvd}

 Although our theory is established for $d$ as a constant, we examine numerically the  power performance of proposed tests as functions of $F=\int_{0}^{1} d(u) d u$.  In particular, we consider another configuration of $d$, i.e.
\begin{align}
   d_2(t) = 0.35 + 0.1 \cos(2 \pi t)  \quad t \in [0,1].
\end{align}
In \cref{fig:power4}, KPSS and related tests display good power performance under models \ref{M0}, \ref{M1} and \ref{M2}, in that the rejection rates of all tests grow to 1 as the sample size increases.

\begin{figure}[!ht]
  \centering
   \includegraphics[width= 0.32\linewidth]{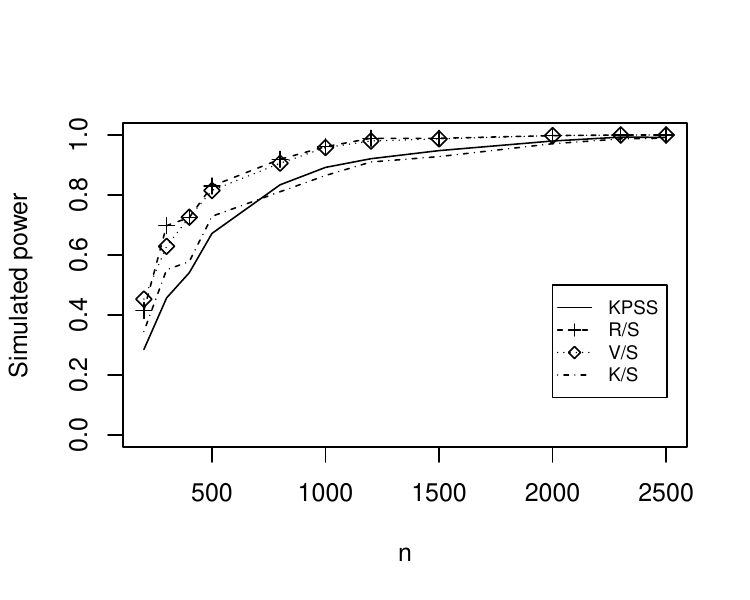}
  \includegraphics[width= 0.32\linewidth]{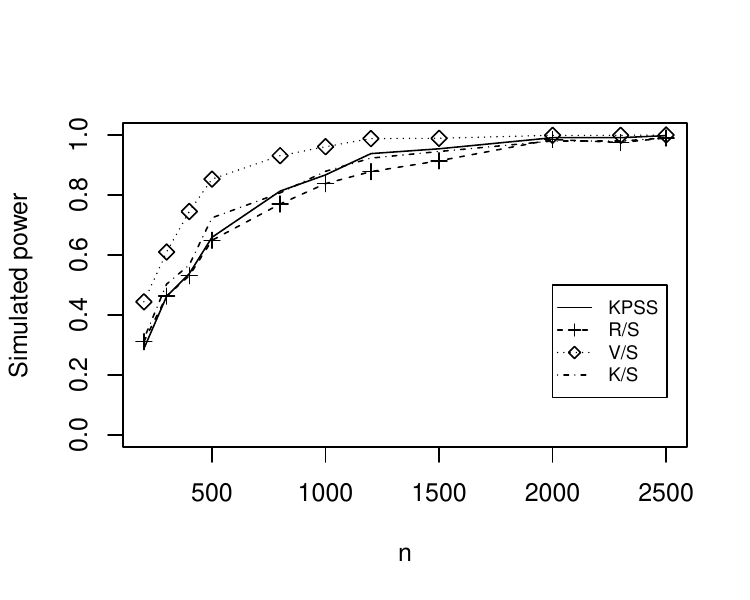}
  \includegraphics[width= 0.32\linewidth]{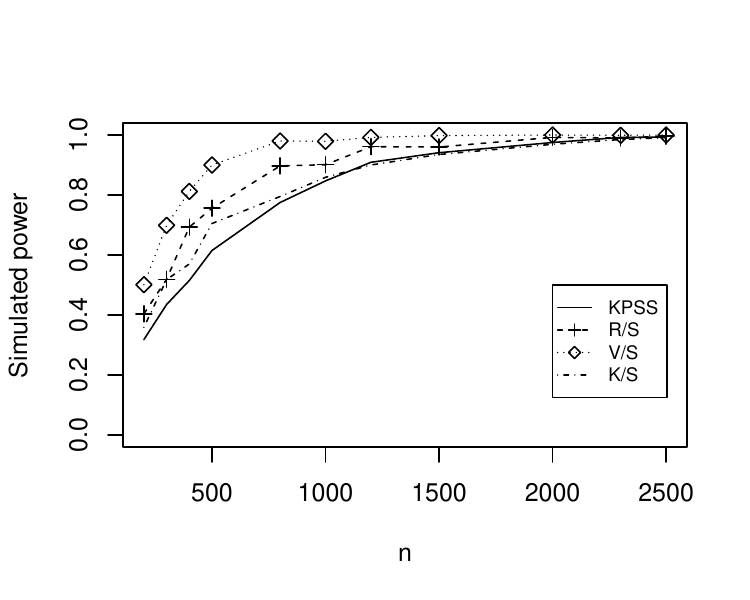}
  \caption{Simulated rejection rates for model \ref{M0}, \ref{M1} and \ref{M2} under the alternative  $d_2$, nominal level 0.1, starting from n = 200.}
  \label{fig:power4}
\end{figure}

\section{Analysis of the COVID-19 infection curve}\label{sec:covid}

We investigate the time series of the cumulative confirmed cases and deaths of COVID-19 in Japan and Ireland, all in log-scale. For each series, we consider the sub-series from the date when its number first exceeds $500$ to 10/06/2021. Our data is obtained from the European Centre for Disease Prevention and Control (ECDC), where the confirmed cases and deaths are updated daily.  The cumulative confirmed cases and deaths of COVID-19  has been modeled by a piecewise linear trend model in \cite{jiang2020time}. We consider the time-varying trend model \eqref{eq:long_memory1} and apply \cref{trend_algorithms} to testing whether the series of log cumulative confirmed cases and deaths of COVID-19 in Japan and Ireland are LRD.
We test for long memory in the two series of each country using the four KPSS and related tests with the critical values generated by $5000$ times of bootstrap.  For the smoothing parameters, we apply  MV criterion  to select $m$ between  $\lf \frac{6}{7} n^{4/15} \rf $ and $\lf \frac{12}{7} n^{4/15} \rf$ and $\tau_n$ between $\lf \frac{26}{29} n^{-5/29}\rf$ and $\lf \frac{34}{29} n^{-5/29}\rf$, where $n$ is the length of the time series. For the cumulative confirmed cases of Japan with $n=577$, $m$ is selected as $4$ for KPSS, R/S, V/S and K/S-type tests and $\tau_n$'s are chosen as $0.350, 0.350, 0.300, 0.350$, respectively. For the cumulative confirmed cases of Ireland with $n=567$, $m$ is selected as $4$ for KPSS, R/S, V/S and K/S-type tests and $\tau_n$'s are chosen as $0.350, 0.350, 0.300, 0.350$, respectively. For the cumulative confirmed deaths  of Japan with $n=524$, $m$ is selected as $4$ and $\tau_n$'s are chosen as $0.355$ for four KPSS and related tests. For the cumulative confirmed deaths of Ireland with $n=538$, $m$ is selected as $4$ for KPSS, and V/S-type tests, $9$ for R/S and K/S -type tests and $\tau_n$'s are chosen as $0.353, 0.303, 0.303, 0.353$, respectively.
By GCV criterion, we select $b_n$ as $0.096$, $0.085$ for cumulative confirmed cases and deaths of Japan, respectively, and $0.090$, $0.088$ for those of Ireland.  

\begin{figure}
  \centering
  \includegraphics[width = \linewidth]{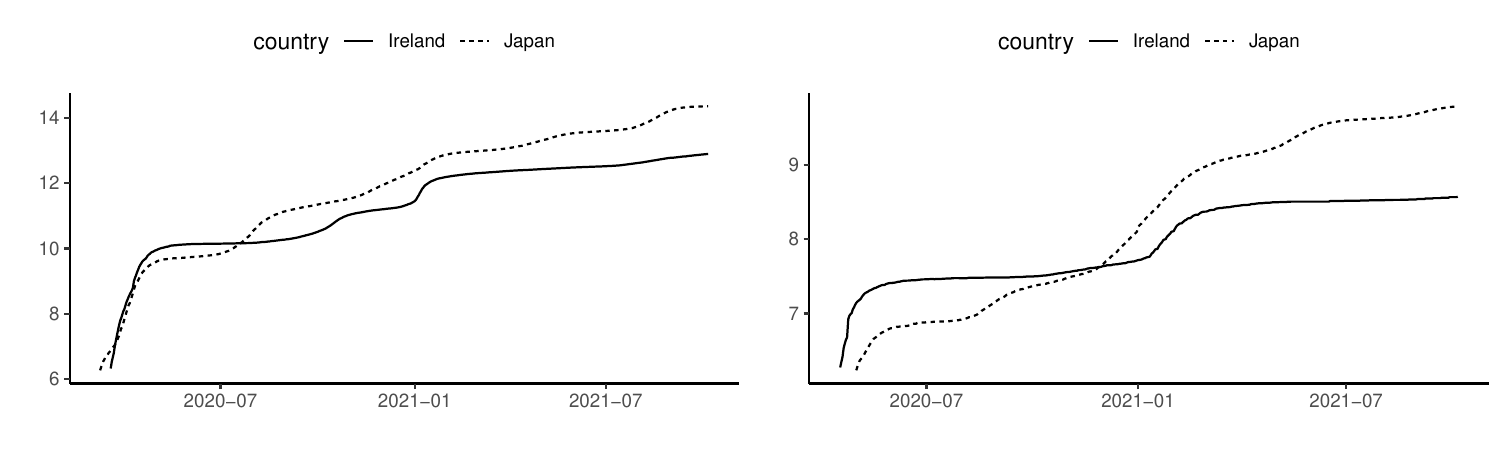}
  \caption{Cumulative confirmed cases (left)  and deaths (right) in log-scale of COVID-19 in Japan and Ireland}
  \label{Ireland_cum}
\end{figure}

Figure \ref{Ireland_cum} displays the time series of cumulative confirmed cases and deaths of Japan and Ireland in log-scale, respectively. The $p$-values of KPSS and related tests are  shown in \cref{tb:covid}. 
For the cumulative confirmed cases series, all four tests reject the null hypothesis at the significance of $0.05$, which indicates significant long-range dependence in the time series of log cumulative confirmed cases of COVID-19 in both countries. On the other hand, all the $p$-values of the four KPSS and related tests for the cumulative confirmed deaths series of COVID-19 exceed 0.05, which fails to reject that the series of log cumulative confirmed deaths of COVID-19 is SRD in either Japan or Ireland. 

\begin{table}[!ht]
\begin{minipage}[b]{.5\linewidth}
  \centering
  \begin{tabular}{c|c|c|c|c}
    \hline
    &KPSS&R/S&V/S&K/S\\ 
    \hline
     cases &0.0012&0.0096&$2\times 10^{-4}$& 0.0022\\
     deaths &0.4332&1&0.9996&0.8244\\
        \hline
  \end{tabular}
\end{minipage}
 \begin{minipage}[b]{.5\linewidth}
  \centering
  \begin{tabular}{c|c|c|c|c}
    \hline
    &KPSS&R/S&V/S&K/S\\ 
    \hline
     cases &0.0198&0.0276&0.0042&0.0032\\
     deaths &0.4524& 1&0.7986&0.998\\
        \hline
  \end{tabular}
 \end{minipage}
   \caption{ $p$-values of KPSS and related tests for the  cumulative confirmed cases and deaths in Japan (left panel) and Ireland (right panel).}
  \label{tb:covid}
\end{table}

\section{Details in analyzing Hong Kong circulatory and respiratory data}\label{hk}
\cref{fig:hk_data} shows the sample path of the covariates ($\text{SO}_2$, $\text{NO}_2$, dust) and the response (the total number of the hospital admissions) in model \eqref{eq:hk_model}. \cref{tb:hk_par} summarizes the smoothing parameters selected in KPSS and related tests when testing for long memory in the series of $\text{SO}_2$, $\text{NO}_2$, dust and total number of hospital admissions modeled by \eqref{eq:long_memory1} and by  \eqref{eq:hk_model}.

 \begin{figure}[!ht]
   \centering
    \includegraphics[width = 0.9\linewidth]{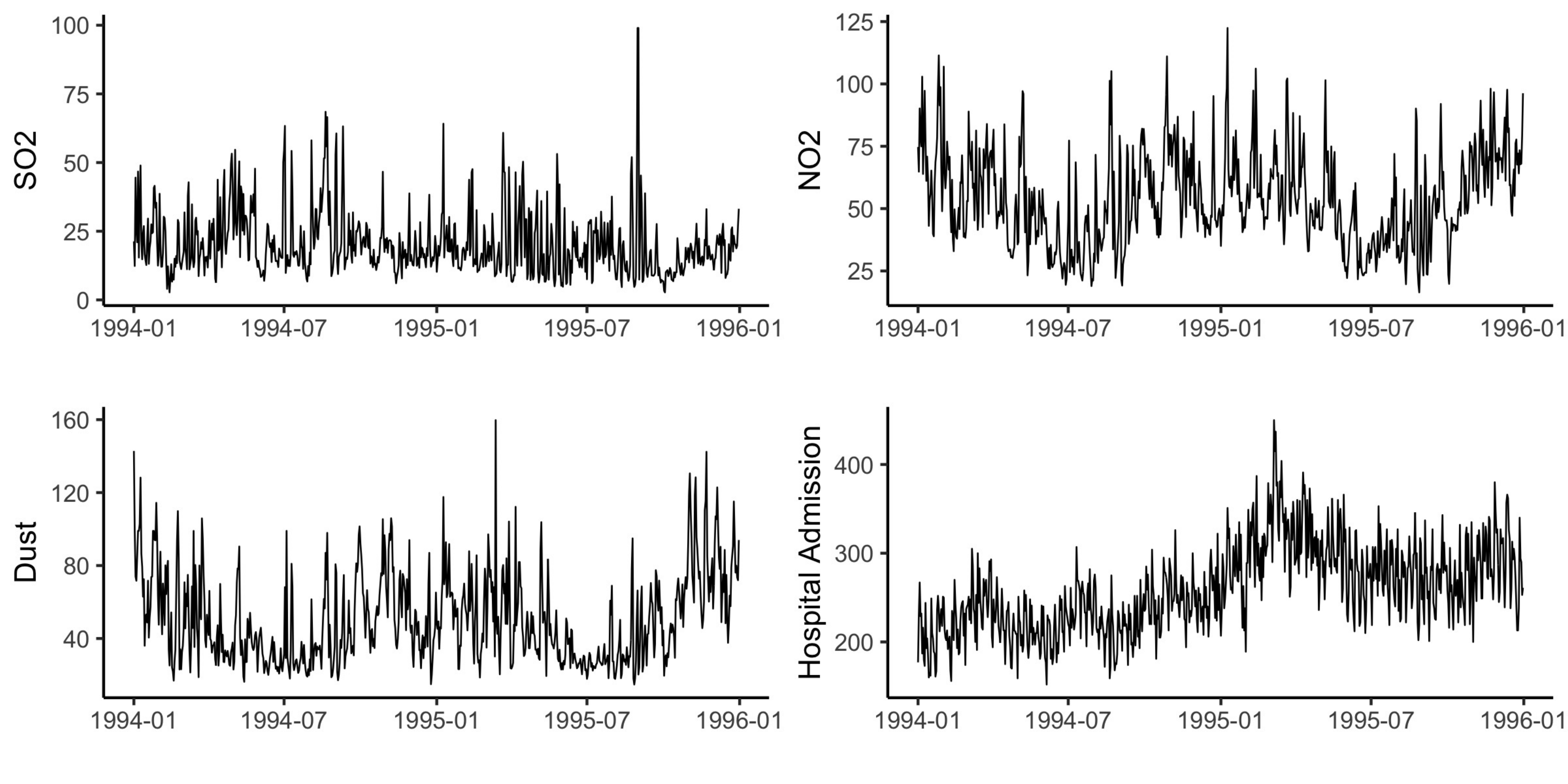}
    \caption{Sample paths of the time series of $\text{SO}_2$, $\text{NO}_2$, dust and the totla number of hospital admissions for the Hong Kong circulatory and respiratory data}
    \label{fig:hk_data}
 \end{figure}
 \begin{table}[ht]
 \centering
 \begin{tabular}{c|ccccc}
  \hline
  &  & KPSS & R/S & V/S & K/S\\
   \hline
  &  $b_n$ &($m$, $\tau_n$)  & ($m$, $\tau_n$) &($m$, $\tau_n$)& ($m$, $\tau_n$)\\ 
   \hline
  $\text{SO}_2$ & 0.250& (16, 0.338) & (16, 0.338) & (16, 0.338) & (16, 0.338) \\ 
  $\text{NO}_2$ & 0.149&  (11, 0.288) & (11, 0.338) & (11, 0.338) & (11, 0.288) \\ 
   Dust & 0.144&  (7, 0.288) & (6, 0.338) & (6, 0.338) & (7, 0.288) \\ 
  model \eqref{eq:long_memory1}&0.138 & (6, 0.338) & (6, 0.338) & (6, 0.338) & (6, 0.338) \\ 
  model \eqref{eq:hk_model} & 0.181& (6, 0.338) & (11, 0.288) & (9, 0.288) & (9, 0.338) \\
 \hline
 \end{tabular}
 \caption{Selected smoothing parameters of KPSS and related tests for $\text{SO}_2$, $\text{NO}_2$, dust and daily total number of hospital admissions modeled by \eqref{eq:long_memory1} and by  \eqref{eq:hk_model}, respectively.}
 \label{tb:hk_par}
 \end{table}

\section{Bootstrap Algorithms} \label{sec:algorithm}
\cref{trend_algorithms} presents the algorithms of KPSS and related tests for the time-varying trend model without time series covariates.  \cref{algorithms} presents the algorithms of R/S, V/S and K/S-type tests for the time-varying coefficient model, and  \cref{cor:bootstraps} investigates the limiting distributions of bootstrap statistics in R/S, V/S and K/S-type tests under null, fixed and the local alternatives. 
\begin{algorithm}
  \caption{The bootstrap procedure for KPSS and related tests for the time-varying trend model}
  \begin{enumerate}
        \item Select the window size $m$ and bandwidth $b_n, \tau_n$, according to the methods in \cref{sel}.
        \item Calculate $\tilde{e}_i = y_{i}  -
        \tilde \beta_1(t_i), i = 1, 2, \cdots, n$,  where $\tilde {\beta}_1$ is obtained by local linear regression \eqref{eq:loclin} with $p=1$ and jackknife correction \eqref{eq:jack}.  Then, compute  the KPSS-type statistic $T_n$ \eqref{eq:KPSS}, R/S-type statistic $Q_n$ in \eqref{Qn},  V/S-type statistic $M_n$ in \eqref{Mn},  K/S-type statistic $G_n$ in \eqref{Gn}.
        \item Calculate $\hat \sigma_{H}^2(t)$ using the estimator in (4.7) of Section 4.2 of \cite{dette2019detecting}.
        \item Generate B (say 2000) $i.i.d.$ copies of $N(0,1)$ variables $ V^{(r)}_i$, for $1 \leq r \leq B$, then calculate
        \begin{align}
          \tilde G^{(r)}_k= \sum_{i=\lfloor nb_n \rfloor+1}^k \hat \sigma_H (t_i) V^{(r)}_i - \frac{1}{n b_n}\sum_{i=\lfloor nb_n \rfloor+1}^k \sum_{j=1}^n \hat \sigma_H (t_j) V^{(r)}_j K_{b_n}(t_j - t_i).  \label{eq:G_k_trend}
        \end{align}
        and the bootstrap version of the KPSS-type statistic \eqref{eq:KPSS},
     \begin{align}\label{eq:bootstrap_trend}
          \tilde T^{(r)}_{n} = \frac{1}{n(n - 2\lf nb_n\rf)}\sum_{s=\lf nb_n \rf + 1}^{ n - \lf nb_n\rf} \left(\sum_{k = \lf nb_n \rf + 1}^s \tilde G^{(r)}_k\right)^2,
        \end{align}
         the bootstrap version of the R/S-type statistic,
      \begin{align}
      \widetilde{\mathrm{RS}}^{(r)}_{n} = \max_{\lf nb_n \rf + 1 \leq k \leq n - \lf nb_n \rf }
     \tilde G^{(r)}_k - \min_{\lf nb_n \rf + 1 \leq k \leq n - \lf nb_n \rf } \tilde G^{(r)}_k,
    \end{align}
     the bootstrap version of the V/S-type statistic,
      \begin{align}
      \widetilde{\mathrm{VS}}^{(r)}_{n} = \frac{1}{n(n - 2\lf nb_n\rf)}\left\{\sum_{k=\lf nb_n \rf + 1 }^{n-\lf nb_n \rf }  (\tilde G^{(r)}_k)^2 - \frac{1}{n - 2\lf nb_n \rf}\left(\sum_{k=\lf nb_n \rf + 1}^{n-\lf nb_n \rf }  \tilde G^{(r)}_k \right)^2\right\}, 
    \end{align}
    the bootstrap version of the K/S-type statistic,
      \begin{align}
      \widetilde{\mathrm{KS}}^{(r)}_{n} = \max_{\lf nb_n \rf + 1 \leq k \leq n - \lf nb_n \rf } 
      \left| \tilde G^{(r)}_k\right|.
    \end{align}
    \item Let $\tilde{T}_{n,(1)} \leq \tilde{T}_{n,(2)} \leq \cdots  \leq \tilde{T}_{n,(B)}$ be the ordered statistics of $\tilde{T}_{n}^{(r)}, r = 1,2,\cdots,B$. Let $\widetilde{\mathrm{RS}}_{n,(1)} \leq \widetilde{\mathrm{RS}}_{n,(2)} \leq \cdots  \leq \widetilde{\mathrm{RS}}_{n,(B)}$ be the ordered statistics of $\{\widetilde{\mathrm{RS}}_{n}^{(r)}\}_{r=1}^B$, $\widetilde{\mathrm{VS}}_{n,(1)} \leq \widetilde{\mathrm{VS}}_{n,(2)} \leq \cdots  \leq \widetilde{\mathrm{VS}}_{n,(B)}$ be the ordered statistics of $\{\widetilde{\mathrm{VS}}_{n}^{(r)}\}_{r=1}^B$, $\widetilde{\mathrm{KS}}_{n,(1)} \leq \widetilde{\mathrm{KS}}_{n,(2)} \leq \cdots  \leq \widetilde{\mathrm{KS}}_{n,(B)}$ be the ordered statistics of $\{\widetilde{\mathrm{KS}}_{n}^{(r)}\}_{r=1}^B$. Let $B^* = \max\{r: \tilde{T}_{n,(r)} \leq T_n\}$. 
    Let $B_{\mathrm{RS}}^* = \max\{r: \widetilde{\mathrm{RS}}_{n,(r)} \leq Q_n\}$, $B_{\mathrm{VS}}^* = \max\{r: \widetilde{\mathrm{VS}}_{n,(r)} \leq M_n\}$, $B_{\mathrm{KS}}^* = \max\{r: \widetilde{\mathrm{KS}}_{n,(r)} \leq G_n\}$. Then the $p$-value of KPSS-type test is $1-B^*/B$, the $p$-value of the R/S-type test is $1-B_{\mathrm{RS}}^*/B$, the $p$-value of the V/S-type test is $1-B_{\mathrm{VS}}^*/B$, and the $p$-value of the K/S-type test is $1-B_{\mathrm{KS}}^*/B$. Reject $H_0$ at the level of $\alpha$ for each type of test if  its $p$-value is smaller than $\alpha$.
    \end{enumerate}
    \label{trend_algorithms}
  \end{algorithm}
  \newpage

\begin{algorithm}
  \caption{The bootstrap procedure of R/S, V/S, K/S-type tests for time-varying coefficient models}
  \begin{enumerate}
    \item Select the window size $m$ and bandwidth $b_n, \tau_n$, according to the methods in \cref{sel}.
    \item Calculate $\tilde{e}_i = y_{i}  -\mathbf{x}_i^{\T}\tilde{{\bs \beta}}(t_i), i = 1, 2, \cdots, n$, where $\tilde{\bs \beta}$ is obtained using local linear regression \eqref{eq:loclin} and jackknife correction \eqref{eq:jack}. Then, compute R/S-type statistic $Q_n$ in \eqref{Qn},  V/S-type statistic $M_n$ in \eqref{Mn},  K/S-type statistic $G_n$ in \eqref{Gn}.
    \item Calculate $\hat{\mf{M}}(t)$ and $\hat{\bs \Sigma}(t)$ defined in \eqref{estiM} and \eqref{eq:diff_correct} of  \cref{sec:diff}, respectively.
    \item Generate B (say 2000) $i.i.d.$ copies of $N(\mf 0, \mf I_p)$ vectors $\mf V^{(r)}_i=(V^{(r)}_{i,1},...,V^{(r)}_{i,p})^\top$, for $1 \leq r \leq B$, then calculate (notice that $\hat \sigma^2_{H}(t)=(\hat {\bs \Sigma}(t))_{1,1} $)
    \begin{align}
      \tilde G^{(r)}_k=-\sum_{j=1}^n\left(\frac{1}{nb_n}\sum_{ i=\lf nb_n \rf+1 }^k \mf x_{i,n}^{\T} \hat{\mf{M}}^{-1}(t_i)  K_{b_n}^*(t_i - t_j)\right) \hat {\bs \Sigma}^{1/2}(t_j)\mf{V}^{(r)}_j + \sum_{ i=\lf nb_n \rf+1}^k \hat \sigma_{H}(t_i)V^{(r)}_{i,1},
    \end{align}
            and the bootstrap version of the R/S-type statistic,
      \begin{align}
      \widetilde{\mathrm{RS}}^{(r)}_{n} = \max_{\lf nb_n \rf + 1 \leq k \leq n - \lf nb_n \rf }
     \tilde G^{(r)}_k - \min_{\lf nb_n \rf + 1 \leq k \leq n - \lf nb_n \rf } \tilde G^{(r)}_k,
    \end{align}
    the bootstrap version of the V/S-type statistic,
      \begin{align}
      \widetilde{\mathrm{VS}}^{(r)}_{n} = \frac{1}{n(n - 2\lf nb_n\rf)}\left\{\sum_{k=\lf nb_n \rf + 1 }^{n-\lf nb_n \rf }  (\tilde G^{(r)}_k)^2 - \frac{1}{n - 2\lf nb_n \rf}\left(\sum_{k=\lf nb_n \rf + 1}^{n-\lf nb_n \rf }  \tilde G^{(r)}_k \right)^2\right\}, 
    \end{align}
    the bootstrap version of the K/S-type statistic,
      \begin{align}
      \widetilde{\mathrm{KS}}^{(r)}_{n} = \max_{\lf nb_n \rf + 1 \leq k \leq n - \lf nb_n \rf } 
      \left| \tilde G^{(r)}_k\right|.
    \end{align}
    \item Let $\widetilde{\mathrm{RS}}_{n,(1)} \leq \widetilde{\mathrm{RS}}_{n,(2)} \leq \cdots  \leq \widetilde{\mathrm{RS}}_{n,(B)}$ be the ordered statistics of $\{\widetilde{\mathrm{RS}}_{n}^{(r)}\}_{r=1}^B$, $\widetilde{\mathrm{VS}}_{n,(1)} \leq \widetilde{\mathrm{VS}}_{n,(2)} \leq \cdots  \leq \widetilde{\mathrm{VS}}_{n,(B)}$ be the ordered statistics of $\{\widetilde{\mathrm{VS}}_{n}^{(r)}\}_{r=1}^B$, $\widetilde{\mathrm{KS}}_{n,(1)} \leq \widetilde{\mathrm{KS}}_{n,(2)} \leq \cdots  \leq \widetilde{\mathrm{KS}}_{n,(B)}$ be the ordered statistics of $\{\widetilde{\mathrm{KS}}_{n}^{(r)}\}_{r=1}^B$.
    Let $B_{\mathrm{RS}}^* = \max\{r: \widetilde{\mathrm{RS}}_{n,(r)} \leq Q_n\}$, $B_{\mathrm{VS}}^* = \max\{r: \widetilde{\mathrm{VS}}_{n,(r)} \leq M_n\}$, $B_{\mathrm{KS}}^* = \max\{r: \widetilde{\mathrm{KS}}_{n,(r)} \leq G_n\}$. Then the $p$-value of the R/S-type test is $1-B_{\mathrm{RS}}^*/B$, the $p$-value of the V/S-type test is $1-B_{\mathrm{VS}}^*/B$, and the $p$-value of the K/S-type test is $1-B_{\mathrm{KS}}^*/B$.  Reject $H_0$ at the level of $\alpha$ for each type of test if  its $p$-value is smaller than $\alpha$.
    \end{enumerate}
    \label{algorithms}
 \end{algorithm}

  \begin{theorem}
    \label{cor:bootstraps}
    The bootstrap statistics $\widetilde{\mathrm{RS}}_{n},  \widetilde{\mathrm{VS}}_{n}, \widetilde{\mathrm{KS}}_{n}$ are defined in \cref{algorithms}.   Then, we have the following results \par
     (i) Under the conditions of \cref{thm:bootstrap_null}(i), we have under $H_0$ 
     \begin{align}
       \widetilde{\mathrm{RS}}_{n} /\sqrt{n} \Rightarrow \sup_{0 \leq t \leq 1} U(t) - \inf_{0 \leq t \leq 1} U(t), 
       \quad
           \widetilde{\mathrm{VS}}_{n} \Rightarrow  \int_0^1 U^2(t) dt - \left(\int_{0}^1 U(t)dt\right)^2,
        \quad
         \widetilde{\mathrm{KS}}_{n}/\sqrt{n} \Rightarrow \sup_{0 \leq t \leq 1}| U(t)|,
     \end{align}
     where $U(t)$ is as defined in \cref{thm:null_dist}. \par 
     (ii) For the fixed alternatives, under the conditions of \cref{thm:bootstrapA} (i), 
       we have
       \begin{align}
         &m^{-d} \widetilde{\mathrm{RS}}_{n} /\sqrt{n} \Rightarrow \sup_{0 \leq t \leq 1} \tilde U_d(t) - \inf_{0 \leq t \leq 1} \tilde U_d(t), \\ &
      m^{-2d} \widetilde{\mathrm{VS}}_{n} \Rightarrow  \int_0^1 \tilde U_d^2(t) dt - \left(\int_{0}^1 \tilde U_d(t)dt\right)^2,\quad
       m^{-d} \widetilde{\mathrm{KS}}_{n}/\sqrt{n} \Rightarrow \sup_{0 \leq t \leq 1}| \tilde U_d(t)|,
       \end{align}
       where $\tilde U_d(t)$ is as defined in (i) of \cref{thm:bootstrapA}.\par
       
       (iii) For the local alternatives $d_n = c /\log n$ with some constant  $c>0$, under the conditions of \cref{thm:bootstrapA} (ii), we have
        \begin{align}
         \widetilde{\mathrm{RS}}_{n} /\sqrt{n} \Rightarrow \sup_{0 \leq t \leq 1} \check U_{\alpha}(t) - \inf_{0 \leq t \leq 1} \check U_{\alpha}(t), 
         \quad
             \widetilde{\mathrm{VS}}_{n} \Rightarrow  \int_0^1 \check U_{\alpha}^2(t) dt - \left(\int_{0}^1 \check U_{\alpha}(t)dt\right)^2,
          \quad
           \widetilde{\mathrm{KS}}_{n}/\sqrt{n} \Rightarrow \sup_{0 \leq t \leq 1}| \check U_{\alpha}(t)|,
        \end{align}
        where $\check U_{\alpha}(t)$ is as defined in (ii) of \cref{thm:bootstrapA}.
   \end{theorem}
\cref{cor:bootstraps} follows from the proofs of  \cref{thm:bootstrap_null} and \cref{thm:bootstrapA} and continuous mapping theorem. Therefore, for the sake of brevity, we omit its proof. 

\section{Proofs and related results of Sections \ref{sec:model} and \ref{Test}}\label{sec1}

\subsection{Proof of \texorpdfstring{\cref{Prop31}}{Proposition 3.1}}
By Lemma 3.2 of \cite{KOKOSZKA199519}, under \cref{assumptionHp}, we have
\begin{align}
  \delta_p(H^{(d)}, l, (-\infty, 1]) 
 \leq \sum_{k=0}^{l} \psi_k(d) \delta_p(H,l-k, (-\infty, 1])  
 =O\{(1+l)^{d-1}\}.
  \end{align}



\subsection{Limiting distributions of R/S, V/S, and K/S-type statistics}\label{sub:limits}

The limiting  behavior of  R/S, V/S and K/S-type statistics defined in \cref{sec:4} under \cref{nondeg:null} can be derived by  \cref{thm:null_dist}, \cref{thm:alt_approx} and \cref{thm:local} as well as an application of continuous mapping theorem.  Recall the definitions of $U(t)$, $U_d(t)$, $U^{\circ}(t)$ in \cref{thm:null_dist}, \cref{thm:alt_approx} and \cref{thm:local}.

For the R/S-type statistic defined in \eqref{Qn}, under $H_0$, we have
$
  Q_n/\sqrt{n} \Rightarrow \sup_{0 \leq t \leq 1} U(t) - \inf_{0 \leq t \leq 1} U(t).
$
Under fixed alternatives with long-memory parameter $d$, we have
$
Q_n \Gamma(d+1)/n^{d+1/2} \Rightarrow \sup_{0 \leq t \leq 1} U_d(t) - \inf_{0 \leq t \leq 1} U_d(t),
$
and under local alternatives with $d_n = c / \log n$,
$
Q_n /\sqrt{n} \Rightarrow \sup_{0 \leq t \leq 1} U^{\circ}(t) - \inf_{0 \leq t \leq 1} U^{\circ}(t).
$
\par
For the V/S-type statistic defined in \eqref{Mn}, under $H_0$, we have
  $
    M_n \Rightarrow  \int_0^1 U^2(t) dt - \left(\int_{0}^1 U(t)dt\right)^2.
  $
 Under fixed alternatives with long memory parameter $d$, we have
$
    M_n  \Gamma^2(d+1)/n^{2d} \Rightarrow  \int_0^1 U^2_d(t) dt - \left(\int_{0}^1 U_d(t)dt\right)^2,
$
and under local alternatives with $d_n = c / \log n$,
$
M_n \Rightarrow  \int_0^1 U^{\circ,2}(t) dt - \left(\int_{0}^1 U^{\circ}(t)dt\right)^2.
$
\par
For the K/S-type statistic defined in \eqref{Gn}, under $H_0$, we have
$
G_n/\sqrt{n} \Rightarrow \sup_{0 \leq t \leq 1}| U(t)|.
$         
Under fixed alternatives with long-memory parameter $d$, we have
$
G_n \Gamma(d+1)/n^{d+1/2} \Rightarrow \sup_{0 \leq t \leq 1} |U_d(t)|,
$
and under local alternatives with $d_n = c / \log n$,
$
G_n/\sqrt{n} \Rightarrow \sup_{0 \leq t \leq 1}| U^{\circ}(t)|.
$
\subsection{Proof of \texorpdfstring{ \cref{thm:null_dist}}{Theorem 4.1}}\label{sec2}
Before proving \cref{thm:null_dist}, we study the covariance between $\mf x_i$ and $\mf x_{j}e_{j}$. 

\begin{proposition}\label{Propo-4-24-1}
  Let $\bar {\mf x}_i^{\T}=\mf x_i^{\T}-{\bs \mu}_W(t_i)$ be a $p$-dimensional vector with $j_{th}$ entry  $ \bar {x}_{i,j}$. Let $x_{i,l}$ be  $l_{th}$ entry of $\mf x_i$. Then under Assumption \ref{Ass-U} and \ref{Ass-W}, $1\leq l, k\leq p$, we have that
  \begin{align}
  \max_{1 \leq i,j\leq n, 1\leq k,l\leq p}| \E(\bar {x}_{i,l} x_{j,k}e_{j})|=O(\chi^{|i-j|}).
  \end{align}
\end{proposition}

\begin{proof}[Proof of \cref{Propo-4-24-1}]
  Under \cref{Ass-W}, $\bar {x}_{i,k} = \sum_{m \in \mathbb{Z}}\proj_{m}\{\bar {x}_{i,k}\}$,  $x_{j,k}e_{j} = \sum_{m \in \mathbb{Z}}\proj_{m}\{x_{j,k}e_{j}\}$. Then, with the orthogonality of $\proj_j$
   , we have
    \begin{align}
      |\E(\bar {x}_{i,l} x_{j,k}e_{j})| 
      &= \left| \E\left[\sum_{m \in \mathbb{Z}}\proj_{m}\{\bar {x}_{i, l}\} \proj_{m}\{x_{j,k} e_{j}\}\right]\right|
      \leq \sum_{m \in \mathbb{Z}} \delta_2(\mf W,i-m)\delta_2(\mf U,j - m) 
      = O(\chi^{|i-j|}).
    \end{align}
  The last equality follows from the SRD conditions in Assumptions \ref{Ass-U} and \ref{Ass-W}.
  \end{proof}

  \begin{lemma}\label{lm:scblemma6}
    Under the condition of \cref{thm:null_dist}, we have 
    \begin{align}
      \sup_{t \in \mathcal T} |\mf S_{n,0}(t) - \mf M(t)| = \Op(n^{-1/2}b_n^{-3/4} + b_n^2),
    \end{align}
    where $\mf S_{n,0} = \frac{1}{nb_n}\sum_{i=1}^n \mf x_i \mf x_i^{\T} K_{b_n}(t_i - t).$
  \end{lemma}
  \begin{proof}
    Similar to the proof in Lemma 6 of \cite{zhou2010simultaneous}, under Assumption \ref{A:W_delta}, we have
    \begin{align}
      \|\mf S_{n,0}(t)  - \E \{ \mf S_{n,0}(t)  \}\|_4 = O((nb_n)^{-1/2}).
    \end{align}
   By the chaining argument in Proposition B.1 in Section B.2 in \cite{dette2018change}, we have
   \begin{align}
   \left \|\sup_{t \in \mathcal T}|\mf S_{n,0}(t)  - \E \{ \mf S_{n,0}(t)\} \right \|_4 = O(n^{-1/2}b_n^{-3/4}).
   \end{align}
   Finally, under \ref{A:Mt_smooth} and Assumption \ref{A:K},  we have uniformly for $t \in \mathcal T$, that 
   \begin{align}
     \left| \E \{ \mf S_{n,0}(t) - \mf M(t)\right| & =  \left| \frac{1}{nb_n}\sum_{i=1}^n \mf M(t_i)K_{b_n}(t_i - t) - \mf M(t)\right| \\ 
     & = \left| \frac{1}{nb_n}\sum_{i=1}^n (\mf M(t) + \mf M^{\prime}(t) (t_i - t) + O(b_n^2))K_{b_n}(t_i - t) - \mf M(t)\right|
     & = O((nb_n)^{-1} + b_n^2).
   \end{align}
  \end{proof}
 
\subsubsection{Proof of \texorpdfstring{\cref{thm:null_dist}}{Theorem 4.1}  (i)}

Define $\mathcal T = [b_n ,1-b_n]$.
  Under Assumptions \ref{A:K}, \ref{A:beta}, \ref{Ass-U} and \ref{Ass-W},
  replacing Lemma 6 of \cite{zhou2010simultaneous} by \cref{lm:scblemma6} in the proof of Theorem 3 of \cite{zhou2010simultaneous} yields that
  \begin{align}
      \sup_{t\in \mathcal T}\left|\tilde {\bs \beta}_{b_n}(t)-{\bs \beta}(t)-\sum_{i=1}^n\frac{\mf{M}^{-1}(t)}{nb_n}\mf x_ie_iK^*_{b_n}(t_i-t)\right|=\Op(\rho_n'\chi_n'),
  \end{align}
  where $\rho_{n}^{\prime}=(n b_{n})^{-1 / 2} \log n +b_{n}^{2}$, $\chi^{\prime}_n = n^{-1 / 2} b_{n}^{-3/4}+b^2_{n}$, $K^*_{b_n}(t_i-t)=2 K_{\frac{b_n}{\sqrt 2}}(t_i-t)-K_{b_n}(t_i-t)$. Then, uniformly for $\lf nb_n \rf+1 \leq r\leq n - \lf nb_n \rf$, we have
  \begin{align}
  \sum_{ i=\lf nb_n \rf+1}^r\tilde e_i
  &=-\sum_{j=1}^n\left(\frac{1}{nb_n}\sum_{ i=\lf nb_n \rf+1 }^r\mf x_i^{\T}\mf{M}^{-1}(t_i)K_{b_n}^*\left(t_i-t_j\right)\right)\mf x_je_j+\sum_{ i=\lf nb_n \rf+1}^re_i+\Op(n\rho_n'\chi_n').\label{sumhate}
  \end{align}
  Define the following function $G^*(r)$, $\lf nb_n \rf+1 \leq r\leq n - \lf nb_n \rf$,
  \begin{align}\label{Gstar}
  G^*(r)=-\sum_{j=1}^n\left(\frac{1}{nb_n}\sum_{ i=\lf nb_n \rf+1 }^r{\bs \mu}_W^{\T}(t_i)\mf{M}^{-1}(t_i)K_{b_n}^*\left(t_i-t_j\right)\right)\mf x_je_j+\sum_{ i=\lf nb_n \rf+1}^re_i.
  \end{align}
  Then combining \eqref{sumhate} and \eqref{Gstar},  we have
  \begin{align}
  \max_{\lf nb_n \rf + 1 \leq r \leq n- \lf  nb_n \rf }\left|G^*(r)-\sum_{ i=\lf nb_n \rf+1}^r\tilde e_i\right|
  \leq \max_{\lf nb_n \rf + 1 \leq r \leq n- \lf  nb_n \rf } |\tilde M_r|+\Op(n \rho_n'\chi_n'),
  \label{eq:null_Gstar}
  \end{align}
  where
  \begin{align}
  \tilde M_r=\sum_{j=1}^n\left(\frac{1}{nb_n}\sum_{ i=\lf nb_n \rf+1 }^r\lt(\mf x^{\T}_i-{\bs \mu}_W^{\T}(t_i)\rt)\mf{M}^{-1}(t_i)K_{b_n}^*\left(t_i-t_j\right)\right)\mf x_je_j.
  \end{align}\par
  We shall show (i) the bound for $\max_{\lf nb_n \rf + 1 \leq r \leq n- \lf  nb_n \rf}| \tilde{M}_r|$, (ii) the asymptotic behavior of the process $G^{*}(r)$. We break the proof into several steps. Step 1 derives the maximum bound for $| \tilde{M}_r|$. The Gaussian approximation result of $G^{*}(r)$ is established in Step 2. In Step 3, we obtain the limiting distribution of $G^{*}(\lf nt \rf)/\sqrt{n}$ and its convergence with Skorohod topology on $D[0,1]$.
  
 \textbf{Step 1:} 
  We shall show that 
  \begin{align}
  \max_{\lf nb_n \rf + 1 \leq r \leq n- \lf  nb_n \rf}| \tilde{M}_r|=\Op(b_n^{-1}).\label{eq:step1}
  \end{align}
  
  Let $\bar {\mf x}_i^{\T}=\mf x_i^{\T}-{\bs \mu}_W(t_i)$ be a $p$-dimensional vector with $j_{th}$ entry  $ \bar {x}_{i,j}$. Let $x_{i,l}$ be  $l_{th}$ entry of $\mf x_i$.  For the sake of brevity, let $\mf L_s = \mf x_s e_s$ and $L_{s,k}$ be the $k_{th}$ element of $\mf L_s$, $M^{-1}_{l,k}(t)$ be the element in the $l_{th}$ row and $k_{th}$ column of $\mf{M}^{-1}(t)$, where $ 1\leq k,l \leq p$.
  Assumption \ref{Ass-W} guarantees $\sup_{t \in [0,1],1 \leq l,k \leq p}|M^{-1}_{l,k}(t)|$ is bounded.
  Consider the following m-dependent sequences \begin{align}
      \tilde{L}_{s, k, m}=\mathbb{E}\left(x_{s,k} e_{s} | \varepsilon_{s}, \dots, \varepsilon_{s-m}\right),~\tilde{\bar{x}}_{s,k, m}=\mathbb{E}\left(x_{s,k}-\E(x_{s,k})| \varepsilon_{s}, \dots, \varepsilon_{s-m}\right), 1\leq k\leq p.
  \end{align} 
  Further define
  \begin{align}
    \tilde M_r^{(m)}=\sum_{k=1}^p\sum_{l=1}^p\sum_{j=1}^n\left(\frac{1}{nb_n}\sum_{ i=\lf nb_n \rf+1 }^r\bar{x}_{i,l}M^{-1}_{l,k}(t_i)K_{b_n}^*\left(t_i-t_j\right)\right) \tilde{L}_{j, k, m},
  \end{align}
  and 
  \begin{align}
    \bar M_r^{(m)}=\sum_{k=1}^p\sum_{l=1}^p\sum_{j=1}^n\left(\frac{1}{nb_n}\sum_{ i=\lf nb_n \rf+1 }^r\tilde{\bar{x}}_{i,l,m} M^{-1}_{l,k}(t_i)K_{b_n}^*\left(t_i-t_j\right)\right) \tilde{L}_{j, k, m}.
  \end{align}
  Write $\tilde \FF_{s, s-j} = (\varepsilon_{s-j}, \cdots, \varepsilon_{s})$,  $\FF_{s, s-j} = (\FF_{s-j-1},\varepsilon^*_{s-j}, \cdots,\varepsilon_s)$, where $ \{\varepsilon^*_{i}\}_{i\in \mathbb Z}$ are the $i.i.d.$ copy of $ \{\varepsilon_i\}_{i\in \mathbb Z}$. Observe that 
  \begin{align}
    \tilde L_{s,k,m} -L_{s,k} = \sum_{j=m}^{\infty} \{ \E[L_{s,k}|
  \tilde \FF_{s,s-j}] - \E[L_{s,k}|\tilde \FF_{s,s-j-1}]\}
  \end{align}
  is the summation of martingale differences. Let $\tilde{L}_{s, k}^{(i-l)}$ denote changing $\varepsilon_{i-l}$ with $i.i.d.$ copy $\varepsilon^*_{i-l}$ in $\tilde{L}_{s, k}$. Under condition \ref{A:U_delta}, by triangle inequality (the first inequality \eqref{Lleq1}) and Jensen's inequality (the second inequality \eqref{Lleq2}), we have
  \begin{align}
    \|\tilde L_{s,k,m} -L_{s,k} \|_4 
    &\leq C \sum_{j=m}^{\infty} \left\| \E[L_{s,k}|
    \tilde \FF_{s,s-j}] - \E[L_{s,k}|\tilde \FF_{s,s-j-1}]\right\|_4 \label{Lleq1}\\ 
   &\leq C \sum_{j=m}^{\infty} \left\| L_{s,k}^{(s-j-1)} - L_{s,k} \right\|_4 \label{Lleq2}
    = O(\chi^{m}).\label{Lleq4}
  \end{align}
 Then, using Jensen's equality, we have 
  \begin{align}
    \|\proj_{s-j}(\tilde L_{s,k,m} -L_{s,k})\|_4 \leq 2\|\tilde L_{s,k,m} -L_{s,k} \|_4 = O(\chi^m).
  \end{align}
  At the same time, 
  \begin{align}
    \|\proj_{s-j}(\tilde L_{s,k,m} -L_{s,k})\|_4 \leq\|\proj_{s-j}\tilde L_{s,k,m}\|_4 + \| \proj_{s-j}L_{s,k}\|_4 \leq 2 \delta_4(U,j)=O(\chi^j).
  \end{align}
  Therefore, 
  \begin{align}
    \|\proj_{s-j}(\tilde L_{s,k,m} -L_{s,k})\|_{4} = O(\chi^{\max(j,m)}).
    \label{eq:marL}
  \end{align} 
  Similarly, we have 
  \begin{align}
    \|\tilde{\bar{x}}_{i,l,m} - \bar{x}_{i,l}\|_{4} = O(\chi^m),\quad 
    \|\proj_{s-j}(\tilde{\bar{x}}_{i,l,m} - \bar{x}_{i,l})\|_{4}  = O(\chi^{\max(j,m)}).
    \label{eq:marx}
  \end{align}
  
  As a consequence, by Burkholder's inequality and \eqref{eq:marL}, for some large constant $M$,
  \begin{align}
    &\max_{1\leq i \leq n}\left \|\sum_{j=1}^n K_{b_n}^*\left(t_i-t_j\right)(L_{j, k} - \tilde{L}_{j, k, m}) \right\|_4\\ &\leq M\max_{1\leq i \leq n}\sum_{l=0}^{\infty}\left  \| \sum_{j=1}^n K_{b_n}^*\left(t_i-t_j\right)\proj_{j-l}(L_{j, k} - \tilde{L}_{j, k, m}) \right\|_4  = O(\sqrt{nb_n}m\chi^m).
    \label{eq:ker_m_approx}
  \end{align} 
  Then, by Cauchy inequality and \eqref{eq:ker_m_approx}, it follows that
  \begin{align}
    &\left \|\max_{\lf nb_n \rf + 1 \leq r \leq n- \lf  nb_n \rf}| \tilde{M}_r - \tilde{M}_r^{(m)}| \right \| \notag\\
    &\leq \sum_{k=1}^p \left\{\max_{1 \leq i \leq n}\left \|\frac{1}{nb_n}\sum_{j=1}^n K_{b_n}^*\left(t_i-t_j\right)(L_{j, k} - \tilde{L}_{j, k, m} )\right \|_4 \times \sum_{ i=\lf nb_n \rf+1 }^{n - \lf nb_n \rf}  \sum_{l=1}^p \left \|\bar{x}_{i,l}M^{-1}_{l,k}(t_i)\right\|_4 \right\}\notag\\ 
    & = O\lt(p^2 \sqrt{n/b_n} m \chi^m \rt).
    \label{eq:m_approx1}
  \end{align}
  
  An elementary calculation using Burkholder's inequality shows that 
  \begin{align}
   \max_{1\leq i\leq n} \left\|\sum_{j=1}^n K_{b_n}^*\left(t_i-t_j\right)\tilde{L}_{j, k, m} \right\|_4 = O\left(\sqrt{n b_n}\right).
    \label{eq:ker_m}
  \end{align}
  Along with equation \eqref{eq:marx}, it's straightforward to show that
  \begin{align}
    &\left \|\max_{\lf nb_n \rf + 1 \leq r \leq n- \lf  nb_n \rf}\lt| \bar{M}_r^{(m)} - \tilde{M}_r^{(m)}\rt| \right \|\notag \notag\\
    & \leq  \frac{1}{nb_n} \sum_{k=1}^p \left \{ \sum_{ i=\lf nb_n \rf+1 }^{\lf n-nb_n \rf}\sum_{l=1}^p \left\|(\bar{x}_{i,l}-\tilde{\bar{x}}_{i,l,m})M^{-1}_{l,k}(t_i)\right\|_4 \times \max_{1 \leq i \leq n} \left\|\sum_{j=1}^n K_{b_n}^*\left(t_i-t_j\right)\tilde{L}_{j, k, m} \right\|_4 \right \} \\ 
    & = O\left(p^2 \sqrt{n/b_n} \chi^m\right).
    \label{eq:m_approx2}
  \end{align}
  Therefore,  $\bar{M}_r^{(m)}$ is an appropriate approximation of $\tilde{M}_r$, in that 
  \begin{align}
    \left \|\max_{\lf nb_n \rf + 1 \leq r \leq n- \lf  nb_n \rf}\lt| \bar{M}_r^{(m)} - \tilde{M}_r\rt| \right \| = O\left(p^2 \sqrt{n/b_n} m \chi^m\right).
    \label{eq:m_approx12}
  \end{align} 
  Using the argument similar to  \cref{Propo-4-24-1}, we have that
   \begin{align}
  \max_{\lf nb_n \rf + 1 \leq r \leq n- \lf  nb_n \rf }\E(\bar M_r^{(m)})=O\left \{p^2 \sum_{i=1}^n \sum_{j=1}^n 
  \chi^{|i-j|}/(nb_n)\right \} = O\left(p^2 b_n^{-1}\right).
  \label{eq:m_approx3}
   \end{align} 
   \par
  Now, we proceed to compute the order of $\max_{\lf nb_n \rf + 1 \leq r \leq n- \lf  nb_n \rf}| \bar{M}_r^{(m)} - \E(\bar{M}_r^{(m)})|$.\par
  Write $\bar a_{n,i}^{(m)} = \sum_{k=1}^p\sum_{l=1}^p \left(\frac{1}{nb_n}\tilde{\bar{x}}_{i,l,m} M^{-1}_{l,k}(t_i)\sum_{j=1}^n K_{b_n}^*\left(t_i-t_j\right)\right) \tilde{L}_{j, k, m}$, then $\bar{M}_r^{(m)} = \sum_{ i=\lf nb_n \rf+1}^r \bar a_{n,i}^{(m)}$.
  Observe that $\proj_{j-s}(\tilde{\bar{x}}_{j,l,m} \tilde L_{i,k,m}) = 0$, for $s > 2m$. Then, we have 
  \begin{align}
    \left \|\max_{\lf nb_n \rf + 1 \leq r \leq n- \lf  nb_n \rf}\lt| \bar{M}_r^{(m)} - \E(\bar{M}_r^{(m)})\rt| \right \| \leq  \sum_{s=0}^{2 m}\left\|\max _{\lf nb_n \rf + 1 \leq r \leq n- \lf  nb_n \rf}\left|\sum_{i =\lf n b_{n}\rf}^{r} \proj_{i-s} \bar{a}_{n, i}^{(m)}\right|\right\|.
    \label{eq:converge_m_approx}
  \end{align}
  According to Lemma 3 in \cite{zhou2010simultaneous}, by triangle inequality, we have
  \begin{align}
    \|\proj_{i-s} \bar{a}_{n, i}^{(m)}\|
    & \leq \sum_{k=1}^p\sum_{l=1}^p \left\{ \frac{1}{nb_n}\lt\|\tilde{\bar{x}}_{i,l,m}-\tilde{\bar{x}}_{i,l,m}^{(i-s)}\rt\|_4 \left|M^{-1}_{l,k}(t_i)\right| \left\|\sum_{j=1}^n K_{b_n}^*\left(t_i-t_j\right) \tilde{L}_{j, k, m}\right\|_4 \right. \notag \\ 
    &+  \left. \frac{1}{nb_n}\lt\|\tilde{\bar{x}}_{i,l,m}^{(i-s)}\rt\|_4 \left|M^{-1}_{l,k}(t_i)\right| \left\|\sum_{j=1}^n K_{b_n}^*\left(t_i-t_j\right) \left(\tilde{L}_{j, k, m}-\tilde{L}_{j, k, m}^{(i-s)}\right)\right\|_4  \right\}\\ 
    &= O \left \{p^2\left(\frac{\chi^{s}}{\sqrt{nb_n}}+\frac{m}{nb_n}\right) \right \}.
  \end{align}
  The last inequality follows from the fact that by Jensen's inequality $\|\tilde{\bar{x}}_{i,l,m}-\tilde{\bar{x}}_{i,l,m}^{(i-s)}\|_4  \leq \|\bar{x}_{i,l} - \bar{x}_{i,l}^{(i-s)}\|_4 = O(\chi^{s})$, Assumption \ref{Ass-W}, and $\tilde{L}_{j, k, m}-\tilde{L}_{j, k, m}^{(i-s)}$ is zero when $j \leq i - s$ and $j \geq i - s + m$.
  Then, since $\proj_{i-s} \bar{a}_{n, i}^{(m)}$ are martingale differences, by Doob's inequality, we obtain
  \begin{align}
    \left\|\max _{\lf n b_n \rf  + 1 \leq r \leq n - \lf n b_n\rf}\left|\sum_{i=\lf n b_n \rf + 1}^{r} \proj_{i-s} \bar{a}_{n, i}^{(m)}\right|\right\|
    &\leq C \left\|\sum_{i=\lf n b_n \rf + 1}^{n - \lf nb_{n}\rf} \proj_{i-s} \bar{a}_{n, i}^{(m)}\right\| \\
     &= O\left\{p^2 \sqrt{n} \left(\frac{\chi^{s}}{\sqrt{nb_n}}+ \frac{m}{nb_n} \right) \right\}, \label{eq:Pabarsum}
  \end{align}
  where $C$ is a positive constant. Plugging \eqref{eq:Pabarsum} into inequality \eqref{eq:converge_m_approx} yields
  \begin{align}
    \left \|\max_{\lf nb_n \rf + 1 \leq r \leq n- \lf  nb_n \rf}\lt| \bar{M}_r^{(m)} - \E(\bar{M}_r^{(m)})\rt| \right \| = O \left \{p^2\left(\frac{m^2}{n^{1/2} b_n} + b_n^{-1/2}\right)\right \}. 
    \label{eq:m_approx4}
  \end{align}
  Finally, from \eqref{eq:m_approx12}, \eqref{eq:m_approx3} and \eqref{eq:m_approx4}, when the dimension $p$ is fixed, taking $m = \lf \log n \rf$, we have proved \eqref{eq:step1}.
 
  Therefore,  by \eqref{eq:null_Gstar} and \eqref{eq:step1}, we have
  \begin{align}
  \max_{\lf nb_n \rf + 1 \leq r \leq n- \lf  nb_n \rf }\left|G^*(r)-\sum_{ i=\lf nb_n \rf+1}^r\tilde e_i\right|
  & = \Op\left(b_n^{-5/4} \log n + nb_n^4 + \sqrt{n}b_n^{5/4} \right),
  \end{align}
which is of smaller order of $\sqrt{n}$.

 \textbf{Step 2:} Recall ${\bs \Sigma}(t_i)$ is the long-run covariance matrix of the process $(\mf x_{i}e_i)$. Since  in our regression we let $\mf x_{i,1}=1$ for $1\leq i\leq n$, $(\bs \Sigma(t_i))_{(1,1)}=\sigma^2_H(t_i)$ is the long-run variance of the process $(e_i)$. We shall show that 
  there exist $i.i.d.$ $N(0, I_{p})$, $\mf V_i=(V_{i,1},...,V_{i,p})^\top$, and 
  \begin{align}\tilde G^*(r)=-\sum_{j=1}^n\left(\frac{1}{nb_n}\sum_{ i=\lf nb_n \rf+1 }^r{\bs \mu}_W^{\T}(t_i)\mf{M}^{-1}(t_i)K_{b_n}^*\left(t_i-t_j\right)\right) {\bs \Sigma}^{1/2}(t_j)\mf{V}_j+\sum_{ i=\lf nb_n \rf+1}^r\sigma_{H}(t_i)V_{i,1}, 
  \end{align}
  such that
  \begin{align}\label{eq:step2}
  \max_{\lf n b_n \rf  + 1 \leq r \leq n - \lf n b_n\rf }|\tilde G^*(r)- G^*(r)|=\Op\left(n^{1/4}\log^2n\right).
  \end{align}
  
  From Corollary 1 in \cite{wu2011gaussian}, we have
  \begin{equation}
  \max _{1 \leq i \leq n}\left|\sum_{j=1}^{i} \mf x_j e_{j}-\sum_{j=1}^{i} {\bs \Sigma}^{1/2}\left(t_{j}\right) \mf V_{j}\right|=o_{\mathbb{P}}\left(n^{1 / 4} \log ^{2} n\right),
  \label{eq:p_gaussian}
  \end{equation}
  and in the first dimension,
  \begin{align}
    \max _{1 \leq i \leq n}\left|\sum_{j=1}^i e_j - \sum_{j=1}^i\sigma_{H}(t_j)V_{j,1}\right|\ = \op\left(n^{1 / 4} \log ^{2} n\right).
  \label{eq:Gaussian}
  \end{align}
  Write $\mf m_{r,j}^{\T} = \frac{1}{nb_n}\sum_{ i=\lf nb_n \rf+1 }^r{\bs \mu}_W^{\T}(t_i)\mf{M}^{-1}(t_i)K_{b_n}^*\left(t_i-t_j\right)$, then 
  \begin{align}
    \max _{\lf n b_n \rf  + 1 \leq r \leq n - \lf n b_n\rf }\left|\tilde G^*(r)- G^*(r)\right| 
    & \leq  \max _{\lf n b_n \rf  + 1 \leq r \leq n - \lf n b_n\rf}\left|\sum_{ i=\lf nb_n \rf+1}^r\sigma_{H}(t_i)V_{i,1}-\sum_{ i=\lf nb_n \rf+1}^re_i\right| \notag\\
    &+\max _{\lf n b_n \rf  + 1 \leq r \leq n - \lf n b_n\rf}\left| \sum_{j=1}^n \mf m_{r,j}^{\T} \mf x_j e_j - \sum_{j=1}^n \mf m_{r,j}^{\T} {\bs \Sigma}^{1/2}\left(t_{j}\right) \mf V_{j} \right|.
    \label{eq:GA1}
  \end{align} 
  Then, \eqref{eq:step2} follows from  \eqref{eq:p_gaussian}, \eqref{eq:Gaussian}, and the summation-by-parts formula.
 
 \textbf{Step 3:}
  Define $\tilde G_{n,b_n}(t) = \tilde G^{*}(\lf n t\rf)/\sqrt{n}$. We shall show that 
  \begin{align}
    \tilde G_{n,b_n}(t) \leadsto U(t) \quad \text{on } D[0,1] \text{ with Skorohod topology}.
  \end{align}
  
  Under the bandwidth condition $nb_n^3/(\log n)^2 \to \infty$, $nb_n^6 \to 0$, we have from Step 1 and Step 2 that
  \begin{align}
  &\max_{\lf nb_n \rf + 1 \leq r \leq n- \lf  nb_n \rf }\left|\tilde G^*(r)-\sum_{ i=\lf nb_n \rf+1}^r\tilde e_i\right|
  =\op(\sqrt{n}).
  \end{align}
  Let $\mu_{W,i}(u)$ denote the $i_{th}$ component of ${\bs \mu}_W(u)$.  Let $m_{r,j,k}$ be the $k_{th}$ element in $\mf m_{r,j}^{\T}$, and $\{\cdot\}_k$ be the $k_{th}$ element in the vector. Under condition \ref{A:Mt} and \ref{A:W}, ${\bs \mu}_W(t)$ and $ \mf{M}^{-1}(t)$ are  Lipschitz continuous.  Since $K^{*}(t)$ can be non-zero only for $t \in [-1, 1]$, elementary calculation shows
    \begin{align}
    m_{r,j, k}
    & = \sum_{i=1}^p  \mu_{W,i}^{\T}(t_j) M_{i,k}^{-1}(t_j) \int_{1-\frac{j}{nb_n}}^{\frac{r-j}{nb_n}} K^{*}\left(y\right) dy + O\left(\frac{1}{n b_n} + b_n\right).
    \label{eq:mrj}
  \end{align}
  Consider $0 \leq t_1 \leq t_2 \leq 1$. Let $s = \lf nt_1 \rf$, $r = \lf nt_2 \rf$.
  The covariance of $\tilde G_{n, b_n}(t)$ is
  \begin{align}
    \E\left\{\frac{\tilde G^*(r) \tilde G^*(s)}{n}\right\} &= \E\lt\{ \sum_{j=1}^n \mf m_{r,j}^{\T}{\bs \Sigma}^{1/2}(t_j)\mf V_j \sum_{j=1}^n \mf m_{s,j}^{\T}{\bs \Sigma}^{1/2}(t_j)\mf V_j \rt\}/n  \notag\\ 
    &- \E\lt\{ \sum_{j=1}^n \mf m_{r,j}^{\T}{\bs \Sigma}^{1/2}(t_j)\mf V_j \sum_{ i=\lf nb_n \rf+1}^{s} \sigma_H(t_i)V_{i,1} \rt\}/n \notag\\ 
    &-\E\lt\{ \sum_{j=1}^n \mf m_{s,j}^{\T}{\bs \Sigma}^{1/2}(t_j)\mf V_j \sum_{ i=\lf nb_n \rf+1}^{r} \sigma_H(t_i)V_{i,1}\rt\}/n \notag\\ 
    &+ \E\lt\{\sum_{ i=\lf nb_n \rf+1}^{r} \sigma_H(t_i)V_{i,1}\sum_{ i=\lf nb_n \rf+1}^{s} \sigma_H(t_i)V_{i,1}\rt\}/n \\ 
    & := I + II + III + IV.
  \end{align}
  Without loss of generality, suppose $\lf nb_n \rf + 1 < s \leq r < n-\lf nb_n \rf$. Let $M_W(t) = {\bs \mu}_W^{\T}(t) \mf{M}^{-1}(t) {\bs \Sigma}(t) \mf{M}^{-1}(t) {\bs \mu}_W (t)$, we have
  \begin{align}
    I & = \sum_{j=1}^n \mf m_{r,j}^{\T}{\bs \Sigma}(t_j)\mf m_{s,j}/n
     = \int_{0}^1 M_W(t) \int_{1-\frac{t}{b_n}}^{\frac{r- t n}{nb_n}} K^{*}\left(y\right) dy \int_{1-\frac{t}{b_n}}^{\frac{s-t n}{nb_n}} K^{*}\left(y\right) dy dt + O\left(\frac{1}{n b_n} + b_n\right).
  \end{align}
  Let $M_{WK}(r,s,t) = M_W(t) \int_{1-\frac{t}{b_n}}^{\frac{r- t n}{nb_n}} K^{*}\left(y\right) dy \int_{1-\frac{t}{b_n}}^{\frac{s-t n}{nb_n}} K^{*}\left(y\right) dy $. When $s /(n b_n) \to \infty$,
  \begin{align}
    \int_{2b_n}^{\frac{s-nb_n}{n}} M_{WK}(r,s,t)  dt 
    &= \int_{0}^{s/n} M_W(t) dt \left\{\int_{-1}^{1} K^{*}(y)dy\right\}^2+O(b_n),
  \end{align}
Since $\int_{1-\frac{t}{b_n}}^{\frac{s - t n}{nb_n}} K^{*}\left(y\right) dy = 0$ for $t > \frac{s+nb_n}{n}$, we have 
  \begin{align}
    I = \int_{0}^{s/n} M_W(r,s,t) dt \left\{\int_{-1}^{1} K^{*}(y)dy\right\}^2+O\left(\frac{1}{n b_n} + b_n\right).
  \end{align}
Similar for the case $s = O(n b_n)  $, since
  \begin{align}
    \int_{0}^{s/n}  M_{WK}(r,s,t) dt = O\left(b_n\right),\quad \int_{0}^{s/n} M_W(t) dt \left\{\int_{-1}^{1} K^{*}(y)dy\right\}^2= O\left(b_n\right),
  \end{align}
  we have, 
  \begin{align}
    I = \int_{0}^{s/n} M_W(t) dt \left\{\int_{-1}^{1} K^{*}(y)dy\right\}^2+O\left(\frac{1}{n b_n} + b_n\right).\label{eq:nullI}
  \end{align}
  
  Similar and tedious calculation shows 
 \begin{align}
    II = - \int_{0}^{s/n} \{{\bs \mu}_W^{\T}(t)\mf{M}^{-1}(t){\bs \Sigma}^{1/2}(t)\}_1 \sigma_H(t)dt \int_{-1}^{1}K^{*}(y)dy +O\left(\frac{1}{nb_n} + b_n\right).\label{eq:nullII}
  \end{align}

  For $III$, if $s/ (nb_n) \to \infty$,
  \begin{align}
    III &= - \frac{1}{n}\sum_{j = \lf nb_n \rf + 1}^{r} \{\mf m_{s,j}^{\T} {\bs \Sigma}(t_j)\}_1 \sigma_H(t_j)\\
    & = - \int_{2b_n}^{s/n-b_n} \int_{-1}^{1}K^{*}(y)dy \{{\bs \mu}_W^{\T}(t)\mf{M}^{-1}(t){\bs \Sigma}^{1/2}(t)\}_1 \sigma_H(t)dt\notag\\ 
    &- \int_{s/n-b_n}^{s/n+b_n} \int_{-1}^{1}K^{*}(y)dy \{{\bs \mu}_W^{\T}(t)\mf{M}^{-1}(t){\bs \Sigma}^{1/2}(t)\}_1 \sigma_H(t)dt \notag\\
    &-\int_{s/n+b_n}^{r/n} \int_{1-\frac{tn}{nb_n}}^{\frac{s-tn}{nb_n}}K^{*}(y)dy \{{\bs \mu}_W^{\T}(t)\mf{M}^{-1}(t){\bs \Sigma}^{1/2}(t)\}_1 \sigma_H(t)dt+ O\left(\frac{1}{nb_n} + b_n\right)\label{eq:III}\\ 
    &= - \int_{0}^{s/n} \{{\bs \mu}_W^{\T}(t)\mf{M}^{-1}(t){\bs \Sigma}^{1/2}(t)\}_1 \sigma_H(t)dt \int_{-1}^{1}K^{*}(y)dy +O\left(\frac{1}{nb_n} + b_n\right).\label{eq:nullIII}
  \end{align}
  For the third term in \eqref{eq:III}, consider  two cases: $r/n >s/n +b_n$ and $s/n \leq r/n \leq s/n + b_n$.
  If $s/n \leq r/n \leq s/n + b_n$, $(s/n+b_n) - r/n \leq b_n$. Then  the third term in \eqref{eq:III} is $O(b_n)$. If $r/n > s/n + b_n$, for $s/n + b_n \leq t \leq r/n$, $\frac{s-tn}{nb_n} \leq -1$. Therefore, this term  is 0. By careful investigation, the result still holds if $s= O(nb_n)$.
  
  The calculation of $IV$ is rather straightforward,
  \begin{align}
    IV =  \int_{0}^{s/n} \sigma_H^2(t) dt + O\left(\frac{1}{nb_n} + b_n\right).\label{eq:nullIV}
  \end{align}
  Then from the approximation of $I-IV$, we have 
  \begin{align}
    \max_{\lf n b_n \rf + 1 <s \leq r <n -\lf n b_n\rf } \left|\E\{\tilde G^*(s)\tilde G^*(r) \}/n - \gamma(s/n, r/n)\right| = O\left(b_n+\frac{1}{nb_n}\right).
    \label{eq:lim_cov}
  \end{align}
  In addition, define $\tilde G^*(s) = \tilde G^*(\lf n b_n\rf)$ if $s < \lf n b_n \rf + 1$ and $\tilde G^*(s) = \tilde G^*(n-\lf n b_n\rf)$ if $n - \lf n b_n\rf < s \leq n $.
By the continuity of $\gamma$, we have
  \begin{align}
    \sup_{0 \leq  t_1 \leq  t_2 \leq 1} \left|\E\left\{\tilde G^*(\lf n t_1 \rf)\tilde G^*(\lf n t_2 \rf) \right\}/n - \gamma(t_1, t_2)\right| =  O\left(b_n+\frac{1}{nb_n}\right).
    \label{eq:lim_cov1}
  \end{align}

 The finite dimension convergence of the Gaussian process $\tilde G_{n,b_n}(t)$ to $U(t)$ then follows from the Cramer Wold device.

  We proceed to show the tightness of $\tilde G_{n,b_n}(t)$.
  For $1 \leq r \leq s \leq n$, since
  \begin{align}
    \tilde G^*(s) - \tilde G^*(r) = -\sum_{j=1}^n\left(\mf m_{s,j}^{\T} - \mf m_{r,j}^{\T}\right) {\bs \Sigma}^{1/2}(t_j)\mf{V}_j+\sum_{i=r+1}^s\sigma_{H}(t_i)V_{i,1}, 
  \end{align}
  it follows from Burholder's inequality that
  \begin{align}
    \left\| \tilde G^*(s) - \tilde G^*(r) \right\|_4^2 &\leq
   K_0 \left( \sum_{j=1}^n \left \|\left(\mf m_{s,j}^{\T} - \mf m_{r,j}^{\T}\right) {\bs \Sigma}^{1/2}(t_j)\mf{V}_j \right\|_4^2+ \sum_{i=r+1}^s\|\sigma_{H}(t_i)V_{i,1}\|_4^2 \right)\\ 
    & \leq K_1 \sum_{j=1}^n \left(\mf m_{s,j}^{\T} - \mf m_{r,j}^{\T}\right) {\bs \Sigma}(t_j) \left(\mf m_{s,j} - \mf m_{r,j}\right)+ K_2(s-r)
  \end{align}
  where $K_0$, $K_1$ and $K_2$ are sufficiently large constants. For the first term, by the result in \eqref{eq:mrj}, we have 
  \begin{align}
     m_{s, j, k} -  m_{r, j, k} = \sum_{i=1}^p  \mu_{W,i}^{\T}(t_j) M_{i,k}^{-1}(t_j) \int_{\frac{r-j}{nb_n}}^{\frac{s-j}{nb_n}} K^*(y) dy  +  O\left(b_n+\frac{1}{nb_n}\right).
  \end{align}
  Observe that
  $m_{s, j, k} -  m_{r, j, k}$ is zero when $j < r - \lf nb_n\rf + 1$ and $j > s + \lf nb_n\rf$. When $ r - \lf nb_n\rf + 1\leq j \leq s + \lf nb_n\rf$,  if $s-r > 2 nb_n$,   $m_{s, j, k} -  m_{r, j, k}$  is $O(1)$  and otherwise $O\left(\frac{s-r}{nb_n}\right)$.
  Hence, if $s-r > 2 nb_n$,
  \begin{align}
    \| \tilde G^*(s) - \tilde G^*(r) \|_4^2  \leq K_3 (s-r + 2 \lf n b_n \rf) + K_2(s-r) = O(|s-r|), \label{eq:tildeG1} 
  \end{align}
  while for $s - r \leq 2 nb_n$,
  \begin{align}
    \| \tilde G^*(s) - \tilde G^*(r) \|_4^2  \leq K_4 \frac{(s-r)^2(s-r + 2 \lf n b_n \rf)}{(nb_n)^2} + K_2(s-r) = O(|s-r|), \label{eq:tildeG2} 
  \end{align}
  where $K_3$, $K_4$  are sufficiently large constants. 
  Hence, for $0 \leq t_1 \leq t \leq t_2 \leq 1$, there exists a sufficiently large constant $K$, s.t.
  \begin{align}
    &\E\left\{|\tilde G_{n, b_n} (t)- \tilde G_{n, b_n}(t_1)|^{2}|\tilde G_{n, b_n}(t_2)- \tilde G_{n, b_n}(t)|^{2}\right\}\notag\\ 
    &\leq \left ( \|\tilde G^*(\lf n t\rf) - \tilde G^*(\lf n t_1\rf)\|_{4} \|\tilde G^*(\lf n t_2\rf) - \tilde G^*(\lf n t\rf)\|_{4}\right)^{2}/n^{2}\\ 
    &\leq K (t_2-t_1)^{2}.\label{eq:null_tightness}
  \end{align}
 Equation (13.2) of \cite{billingsley1999convergence} follows from the the continuity of $U(t)$. By Theorem 13.5 in \cite{billingsley1999convergence}, the $\alpha = \beta = 1$ case, we have the tightness of $\tilde G_{n,b_n}(t)$. 
  The tightness of $\tilde G_{n,b_n}(t)$  and the finite dimension convergence lead to the convergence $\tilde G_{n,b_n}(t) \leadsto U(t)$ on $D[0,1]$ with Skorohod topology.
  Finally, by the continuous mapping theorem, we have proved the convergence of $T_n$ to $\int_0^1 U^2(t) dt$. \hfill $\Box$
\subsubsection{Proof of  \texorpdfstring{\cref{thm:null_dist}}{Theorem 4.1} (ii) }
Let $\mf z(u) := (z_1(u), \cdots, z_p(u))^{\T}$ denote $\mf \Sigma^{1/2}(u)\mf M^{-1}(u) \bs \mu_W(u)$.   Then the covariance structure of the limiting distribution in (i) of  \cref{thm:null_dist} can be written as 
    \begin{align}
      \E(U(r)U(s)) = \int_0^{r \wedge s} \left((\sigma_H(u) - z_1(u))^2 + \sum_{j=2}^p z_j^2(u)\right)  du.
    \end{align}
    Therefore, the limiting distribution in \cref{thm:null_dist} degenerates if and only if $\lambda(\mf z(u) \neq (\sigma_H(u),0 \cdots, 0)^{\T}) >0 $ holds. We shall show that
      \begin{align}
        s_1^{-1}(\tilde T_n/b_n - s_2) \Rightarrow \chi^2_1, \label{deg:step2}
      \end{align}
      where $s_1 = 2\sigma^2_H(0) \int_0^1 \left(\int_{v-1}^1 K^*(t) dt\right)^2 dv$, and $s_2 = 2\int_{0}^1 \sigma^2_H(t) dt \int_0^1 \left(\int_{v}^1 K^*(t) dt\right)^2 dv$.
    Recall in Step 2 of \cref{thm:null_dist}, 
    \begin{align}\tilde G^*(r)=-\sum_{j=1}^n\left(\frac{1}{nb_n}\sum_{ i=\lf nb_n \rf+1 }^r{\bs \mu}_W^{\T}(t_i)\mf{M}^{-1}(t_i)K_{b_n}^*\left(t_i-t_j\right)\right) {\bs \Sigma}^{1/2}(t_j)\mf{V}_j+\sum_{ i=\lf nb_n \rf+1}^r\sigma_{H}(t_i)V_{i,1}.
    \end{align}
    By Step 1 and Step 2 in the proof of \cref{thm:null_dist}, on a richer probability space we have 
    \begin{align}
      \sup_{\nb + 1 \leq r \leq n-\nb} \left|\sum_{i=\nb +1}^r \tilde e_i - \tilde G^*(r)\right| &= \Op(b_n^{- 5/4} \log n + nb_n^4 + \sqrt{n}b_n^{5/4} + n^{1/4}\log^2 n)\\ & = \op(\sqrt{nb_n}/\log n).\label{deg:step1.1-1}
    \end{align}
Define \begin{align}
  \tilde G^{\circ}(r) = \sum_{i=\nb+1}^r \sigma_H(t_i)V_{i,1} - \sum_{j=1}^n \sum_{i=\nb+1}^r \frac{K^*_{b_n}(t_i - t_j)}{nb_n} \sigma_H(t_j)V_{j,1}.
\end{align}
Since $\lambda(\mf z(u) \neq (\sigma_H(u),0 \cdots, 0)^{\T}) =0$, by Taylor series expansion, we have for $\lf nb_n\rf+1\leq r\leq n-\lf nb_n\rf$,
  \begin{align}
    \|\tilde G^*(r) - \tilde G^{\circ}(r) \|^2 & = \left\|\sum_{j=1}^n\left\{\frac{1}{nb_n}\sum_{ i=\lf nb_n \rf+1 }^r \left({\bs \mu}_W^{\T}(t_i)\mf{M}^{-1}(t_i)-{\bs \mu}_W^{\T}(t_j)\mf{M}^{-1}(t_j)\right)K_{b_n}^*\left(t_i-t_j\right)\right\} {\bs \Sigma}^{1/2}(t_j)\mf{V}_j \right\|^2\\ &=  O(nb_n^3).\label{deg:step1.2-1}
    \end{align}

Observe that for $3\nb + 1 \leq  r \leq n$,
\begin{align}\label{deg:coef}
  \sum_{i=\nb+1}^r \frac{K^*_{b_n}(t_i - t_j)}{nb_n} = \begin{cases}
  \int_{\frac{\nb-j}{nb_n}}^1 K^*(t)dt + O((nb_n)^{-1}), &1\leq j \leq r - \nb,\\ 
    1 + O((nb_n)^{-1}), & 2\nb +1 \leq j \leq r - \nb,\\
    \int_{-1}^{\frac{r-j}{nb_n}} K^*(t)dt + O((nb_n)^{-1}), &r - \nb + 1 \leq j \leq r+\nb,\\ 
    0, &\text{otherwise}.
  \end{cases}
\end{align}
Therefore, we have for $\lf nb_n\rf+1\leq r\leq n-\lf nb_n\rf$,
\begin{align}
  \| \tilde G^{\circ}(r) \|^2 = O(nb_n). 
\end{align}
Define 
\begin{align}
  \tilde T_n^{\circ} = \frac{1}{n(n-2\nb)}\sum_{r = \nb + 1}^{n - \nb}(\tilde G^{\circ}(r))^2.
\end{align}
Since $\tilde G^*(r) - \tilde G^{\circ}(r)$ and $\tilde G^{\circ}(r)$ are Gaussian processes,  we have 
\begin{align}
  | \tilde T_n^{\circ} -  T_n| = \Op(b_n).\label{eq:Tncirc}
\end{align}
Write $v_j$ short for $V_{j,1}$.
By \eqref{deg:coef} and some tedious calculation, we have
\begin{align}
  \tilde T_n^{\circ}/b_n &= \frac{1}{nb_n}\left(\sum_{j = \nb + 1}^{2\nb}\left(\int_{-1}^{\frac{j -\nb - 1}{nb_n}}K^*(t)dt\right)\sigma_H(t_j)v_j - \sum_{j = 1}^{\nb}\left(\int_{\frac{j -\nb - 1}{nb_n}}^{1}K^*(t)dt\right)\sigma_H(t_j)v_j \right)^2 \\ & + \frac{1}{n^2 b_n} \sum_{r = 3\nb + 1}^{n - \nb} \left(\sum_{j = r - \nb + 1}^r\left(\int_{\frac{r-j}{nb_n}}^1 K^*(t)dt\right)\sigma_H(t_j)v_j -  \sum_{j = r + 1}^{r + \nb} \left(\int_{-1}^{\frac{r-j}{nb_n}}K^*(t)dt\right)\sigma_H(t_j)v_j \right)^2\\ 
  &+ \Op(b_n^{1/2}) \\ 
  &:= A_n + S_n +  \Op(b_n^{1/2}),
\end{align}
where $A_n$ and $S_n$ are defined in the obvious way. Note that $A_n$ and $S_n$ are independent.
Since $v_i$ are $i.i.d.$ $N(0,1)$,
elementary calculation shows that 
\begin{align}
  A_n/\E A_n \sim \chi^2_1,\label{eq:An1}
\end{align}
and since $K^*$ is symmetric
\begin{align}
  \E A_n = s_1 + O\left(b_n + \frac{1}{nb_n}\right). \label{eq:An2}
\end{align}
Combining \eqref{eq:An1} and \eqref{eq:An2}, by Slutsky's Theorem we have
\begin{align}
  A_n/ s_1 \Rightarrow \chi^2_1,\label{eq:An}
\end{align}
where $s_1$ is as defined in \eqref{deg:step2}.
Observe that 
\begin{align}
  \E S_n = s_2 + O\left(b_n + \frac{1}{nb_n}\right),~\text{and}~\E S_n^2 = O(b_n) + (\E S_n)^2,
\end{align}
where $s_2$ is as defined in \eqref{deg:step2}.
Therefore,
\begin{align}
  |S_n - s_2| \leq |S_n - \E S_n| + |\E S_n - s_2| =  \Op\left(b_n^{1/2} + \frac{1}{nb_n}\right).\label{eq:Sn}
\end{align}
Then, combining \eqref{eq:Tncirc}, \eqref{eq:An}, and \eqref{eq:Sn}, \cref{deg:step2} is proved. \hfill $\Box$

\subsection{Related proofs of \texorpdfstring{\cref{thm:fixlocal}}{Theorem 4.2}}\label{sec:appendG}
This subsection provides the proof of \cref{lm:alt_gaussian} and \cref{prop:5.3}, which are used in the proof of \cref{thm:fixlocal} in \cref{sec:appendA} of the main paper.

\subsubsection{Proof of \texorpdfstring{\cref{lm:alt_gaussian}}{Proposition A.2}}
 
  Observe that
  \begin{equation}
    \sum_{k=1}^{s}{\bs \mu}_W(k/n) e_k^{(d)}=\sum_{k=1}^{s} \sum_{j=0}^{\infty} {\bs \mu}_W(t_k)\psi_j u_{k-j}=\sum_{j=1}^{s} u_{j} \sum_{k=j}^{s} {\bs \mu}_W(t_k)\psi_{k-j} +\sum_{j=0}^{\infty} u_{-j} \sum_{k=1}^{s} {\bs \mu}_W(t_k)\psi_{k+j}
    \end{equation}
Define $Z_j = \sum_{i=0}^j u_{-i}$ with $Z_j = 0$ when $j<0$ and $S_j = \sum_{i=1}^j u_i$ with $S_j = 0$ when $j \leq 0$.
  After a careful check of Corollary 2 of  \cite{wu2011gaussian}, there exist independent variables $v_1, v_2, \cdots , v_n \sim N(0,1)$ and independent  Gaussian variables $v_i$, $i \leq 0$, which are independent of $v_j$, $j>0$, such that
 \begin{equation}
  \zeta_n := \max _{1 \leq i \leq n}\left|\sum_{j=1}^{i} u_{j}-\sum_{j=1}^{i} \sigma_H\left(t_{j}\right) v_{j}\right|=o_{\mathbb{P}}\left(n^{1 / 4} \log ^{2} n\right),
  \label{eq:zeta1}
\end{equation}
\begin{equation}
  \zeta_n^* := \max _{1 \leq i \leq n}\left|\sum_{j=0}^{i} u_{-j}-\sum_{j=0}^{i} \sigma_H\left(t_{-j}\right) v_{-j}\right|=o_{\mathbb{P}}\left(n^{1 / 4} \log ^{2} n\right).
  \label{eq:zeta2}
\end{equation}
Define $\mf R_{k, n}=\sum_{j=0}^{\infty} {\bs \mu}_W(k/n)\psi_j \sigma\left(t_{k-j}\right) v_{k-j}$, $S_j^* = \sum_{i=1}^j \sigma_H(t_i)v_i$, and $Z_j^* = \sum_{i=0}^j \sigma_H(t_{-i})v_{-i}$. Then, by the summation-by-parts formula, we have for some integer $N$,
\begin{equation}
  \begin{aligned}
  \sum_{k=1}^{s}\left({\bs \mu}_W(k/n)e_k^{(d)}- \mf R_{k,n}\right)=& \sum_{j=1}^{s-1}\left(\sum_{k=j}^{s} {\bs \mu}_W(k/n) \psi_{k-j}-\sum_{k=j+1}^{s} {\bs \mu}_W(k/n)\psi_{k-j-1} \right)\left(S_{j}-S_{j}^{*}\right)\\&+\left(S_{s}-S_{s}^{*}\right) {\bs \mu}_W(1) \psi_0\\
  &+\sum_{j=0}^{N-1} \sum_{k=1}^{s}\left({\bs \mu}_W(k/n)\psi_{k+j} - {\bs \mu}_W(k/n)\psi_{k+j+1}\right)\left(Z_{j}-Z_{j}^{*}\right) \\
  &+\left(Z_{N}-Z_{N}^{*}\right) \sum_{k=1}^{s} \psi_{k+N}{\bs \mu}_W(k/n) \\
  &+\sum_{j=N+1}^{\infty} u_{-j} \sum_{k=1}^{s}{\bs \mu}_W(k/n) \psi_{k+j}- \sum_{j=N+1}^{\infty}\sigma_H(t_{-j}) v_{-j} \sum_{k=1}^{s} {\bs \mu}_W(k/n) \psi_{k+j}  \\
  :=& A+B+C+D+E+F.
  \end{aligned}
  \end{equation}
 Let $N = \lf n^{\alpha_0} \rf + 1$, $\alpha_0 > 1$. Condition \ref{A:W} indicates that ${\bs \mu}_W(t)$ is Lipschitz continuous and $\exists C_2>0$, $\sup_{t \in [0,1]}|{\bs \mu}_W(t)|< C_2$. From \eqref{eq:zeta1}, for some postive constant $C_1$, and any $0 < q < 1/4$.
\begin{align}
  \max_{1 \leq s \leq n} |A| &\leq C_1 \zeta_n \max_{1 \leq s \leq n} \sum_{j=1}^{s-1} \left | \sum_{k=j}^{s-1} \psi_{k-j} \left({\bs \mu}_W(k/n) - {\bs \mu}_W((k+1)/n) \right) + \psi_{n-j}{\bs \mu}_W(1)\right |
   =  \Op (n^{1/4 + q + d}).
\end{align}
 Similar techniques and \eqref{eq:zeta2} show that, 
\begin{align}
  \max_{1 \leq s \leq n} |C| 
  =  \Op (n^{\alpha_0/4 + q + \alpha_0 d}),\quad
  \max_{1 \leq s \leq n} |B| \leq |{\bs \mu}_W(1)| \zeta_n \psi_0 = \Op(n^{1/4 + q}),\quad
  \max_{1 \leq s \leq n} |D| \leq 
  \Op(n^{\alpha_0/4 + q + \alpha_0 d}).
\end{align}

Uniformly for $1 \leq s \leq n$, $\tilde \psi_{j,s} := \sum_{k=1}^{s}{\bs \mu}_W(k/n) \psi_{k+j}= O(n|j+1|^{d-1})$, it follows elementary calculation that
\begin{align}
  \left\| \max_{1 \leq s \leq n} |E| \right\|^2 &= \E\left(\sum_{j=N+1}^{\infty} |u_{-j}| \max_s|\tilde \psi_{j,s}| \right)^2 \\ 
  &= \sum_{i=N+1}^{\infty} \sum_{j=N+1}^{\infty} \max_{1 \leq s \leq n} |\tilde \psi_{j,s}|\max_{1 \leq s \leq n} |\tilde \psi_{i,s}| \E(|u_{-j}||u_{-i}| ) = O \left(n^{\alpha_0(2d-1) + 2}\right). \label{eq:E_bound}
\end{align}
Therefore, $\max_{1 \leq s \leq n}|E| = \Op(n^{1 + \alpha_0(d-1/2)})$, and $\max_{1 \leq s \leq n}|F| = \Op(n^{1+\alpha_0(d-1/2)})$.
Finally, it's straightforward to show that $\alpha_0 = 4(1-q) /3$ is the solution of $1+\alpha_0(d-1/2) = \alpha_0/4 + q + \alpha_0 d$, and hence $\alpha_0 \in (1, 4/3).$ \hfill $\Box$

\subsubsection{Proof of \texorpdfstring{\cref{prop:5.3}}{Proposition A.3}}
Observe that 
\begin{align}
\sum_{ i=\lf nb_n \rf+1}^r \mf x_i (e_i^{(d_n)}-  e_i) =  \sum_{j=1}^{L}\sum_{ i=\lf nb_n \rf+1}^r  \mf x_i u_{i-j} \psi_j + \sum_{j= L + 1}^{\infty}\sum_{ i=\lf nb_n \rf+1}^r  \mf x_i u_{i-j} \psi_j : = F_1 + F_2,
\end{align}
where $F_1$ and $F_2$ are defined in the obvious way. We prove the proposition through the following steps:

(i) Show that \begin{align}
\max_{\nb + 1 \leq r \leq n - \nb} \left| F_1 -  \sum_{j=1}^{L}\sum_{ i=\lf nb_n \rf+1}^r  {\bs \mu}_W(t_i) u_{i-j} \psi_j \right| = \Op(\sqrt{n}(\log n)^{-1/2}).\label{eq:F1final}
\end{align}

(ii) 
The second step is to prove that
\begin{align}
\max_{\nb + 1 \leq r \leq n - \nb} \left| F_2 -  \sum_{j= L + 1}^{\infty}\sum_{ i=\lf nb_n \rf+1}^r  {\bs \mu}_W(t_i) u_{i-j} \psi_j\right| = \Op(\sqrt{n}(\log n)^{-1/2}).\label{eq:F2final}
\end{align} 

\textbf{Step (i)}
Let $L = \lf (\log n)^{1/2}\rf$. Then, we have
\begin{align}
\left\|  \max_{\nb + 1 \leq r \leq n - \nb}|F_1 |\right\| \leq   \sum_{j=1}^{L} \psi_j\left\| \max_{\nb + 1 \leq r \leq n - \nb} \left|\sum_{ i=\lf nb_n \rf+1}^r  \mf x_i u_{i-j} \right|\right\| = O(L d_n \sqrt{n}) = O(\sqrt{n}(\log n)^{-1/2}).\label{eq:F111}
\end{align}
Similarly, we can show that 
\begin{align}
\left\|  \max_{\nb + 1 \leq r \leq n - \nb}\left|   \sum_{j=1}^{L}\sum_{ i=\lf nb_n \rf+1}^r  \mf {\bs \mu}_W(t_i) u_{i-j} \psi_j   \right|\right\|
&= O(L d_n \sqrt{n}) = O(\sqrt{n}(\log n)^{-1/2}).\label{eq:F112}
\end{align}
From \eqref{eq:F111} and \eqref{eq:F112}, we have shown \eqref{eq:F1final}.
\par
\textbf{Step (ii)}
Define $\tilde e^{(d_n)}_{i,L} = \sum_{j=L+1}^{\infty}\psi_j u_{i-j}$. We can write 
\begin{align}
  F_2 = \sum_{ i=\lf nb_n \rf+1}^r \mf x_i \tilde e^{(d_n)}_{i,L},\quad  \sum_{j= L + 1}^{\infty}\sum_{ i=\lf nb_n \rf+1}^r  {\bs \mu}_W(t_i) u_{i-j} \psi_j = \sum_{ i=\lf nb_n \rf+1}^r  {\bs \mu}_W(t_i) \tilde e^{(d_n)}_{i,L}.
\end{align}
We approximate $F_2$ following the proof of \cref{prop:5.2}.
Let $\tilde m = \tilde m \log n$, then $\tilde{e}^{(d_n)}_{i,\tilde m} = \sum_{j=\tilde m+1}^{\infty}\psi_j u_{i-j}$. Recall that  $\tilde{\mf{x}}_{i,\tilde m} = \E(\mf x_i| \varepsilon_i,\cdots,\varepsilon_{i-\tilde m})$.

Similar to \eqref{eq:mlong1} and \eqref{eq:mlong11}, we have
\begin{align}
  \left\| \max_{\lf nb_n \rf + 1 \leq r \leq n- \lf  nb_n \rf }\left| F_2 -  \sum_{ i=\lf nb_n \rf+1}^r \tilde{\mf{x}}_{i,\tilde m}\tilde e^{(d_n)}_{i,\tilde m}\right| \right\| = O(\sqrt{n}(\log n)^{-1/2})\label{eq:mlong1L},
\end{align}
i.e. we can approximate $F_2$ by 
$\sum_{ i=\lf nb_n \rf+1}^r \tilde{\mf{x}}_{i,\tilde m}\tilde e^{(d_n)}_{i,\tilde m}.$ 

Secondly, recall the following decomposition similar to \eqref{eq:T1T2},
\begin{align}
\sum_{ i=\lf nb_n \rf+1}^r \tilde{\mf{x}}_{i,\tilde m} \tilde e_{i,\tilde m}^{(d_n)} &= \sum_{l = 0}^{\tilde m} \sum_{i = \lf nb_n \rf + 1}^r \proj_{i-l} \left (\tilde{\mf{x}}_{i,\tilde m} \tilde{e}^{(d_n)}_{i,\tilde m}\right) + \sum_{i = \lf nb_n \rf + 1}^r \E(\tilde{\mf{x}}_{i,\tilde m} \tilde{e}^{(d_n)}_{i,\tilde m}|\FF_{i-\tilde m-1})= T_{1, r} + T_{2, r}.
\label{eq:T1T2L}
\end{align}
We proceed to show that the $T_{1,r}$ is of smaller order of $\sqrt{n}$, and $T_{2, r}$  approximates 
$\sum_{ i=\lf nb_n \rf+1}^r  {\bs \mu}_W(t_i) \tilde e^{(d_n)}_{i,L}$. 

(a) Calculation of $T_{1,r}$.
Similar to the calculation of \eqref{eq:T1}, for $l \leq \tilde m$, we have
\begin{align}
\left\|\sum_{i = \lf nb_n \rf + 1}^{n -\lf nb_n\rf } \proj_{i-l} \left (\tilde{\mf{x}}_{i,\tilde m} \tilde{e}^{(d_n)}_{i,\tilde m}\right) \right\|^2
\leq \sum_{i = \lf nb_n \rf + 1}^{n -\lf nb_n\rf } \left(\|\tilde{\mf{x}}_{i,\tilde m}-\tilde{\mf{x}}^{(i-l)}_{i,\tilde m}\|_4\|\tilde{e}^{(d_n)}_{i,\tilde m}\|_4 \right)^2  \notag.
\end{align}
Notice that uniformly for $\nb + 1 \leq i \leq n - \nb$, 
we have
\begin{align}
\| \tilde{e}^{(d_n)}_{i,\tilde m}\|^2_4  = O\left(\sum_{s=\tilde m+1}^{\infty} \left\| \proj_{i-s}\sum_{j = \tilde m+1}^{\infty} \psi_j u_{i-j} \right\|^2_4 \right)
  = O\left(\sum_{s=\tilde m+1}^{\infty}(s+1)^{2d_n-2}\right) = O((\log n)^{-1}), \label{eq:eim}
\end{align}
where the second equality is from a careful check of Lemma 3.2 in \cite{KOKOSZKA199519}.
Hence, we have 
\begin{align}
\left \| \max_{\lf nb_n \rf + 1 \leq r \leq n- \lf  nb_n \rf }|T_{1, r} | \right \| = O(\sqrt{n}(\log n)^{-1/2}).
\label{eq:T1L}
\end{align}

(b) Calculation of $T_{2,r}$. Since $\tilde e_{i,\tilde m}^{(d_n)}$ is $\FF_{i-\tilde m-1}$ measurable and $\tilde {\mf x}_{i,\tilde m} $ is independent of $\FF_{i-\tilde m-1}$, we have 
\begin{align}
T_{2, r} = \sum_{ i=\lf nb_n \rf+1}^r \E(\tilde{\mf{x}}_{i,\tilde m}|\FF_{i-\tilde m-1}) \tilde{e}^{(d_n)}_{i,\tilde m} = \sum_{ i=\lf nb_n \rf+1}^r \E(\tilde{\mf{x}}_{i,\tilde m}) \tilde{e}^{(d_n)}_{i,\tilde m} =  \sum_{ i=\lf nb_n \rf+1}^r {\bs \mu}_W(t_i) \tilde{e}^{(d_n)}_{i,\tilde m}.
\label{eq:T2L}
\end{align}

Since $\mathcal{P}_{k} \cdot=\mathbb{E}\left(\cdot \mid \mathcal{F}_{k}\right)-\mathbb{E}\left(\cdot \mid \mathcal{F}_{k-1}\right)$, $\tilde e^{(d_n)}_{i,L} - \tilde{e}^{(d_n)}_{i,\tilde m} = \sum_{j = L+1}^{\tilde m} \psi_{j} u_{i-j}$. Similar to \eqref{eq:mlong2} and by Taylor's expansion, we have
\begin{align}
&\left \| \max_{\nb + 1 \leq r \leq n - \nb} \left| T_{2,r} -  \sum_{ i=\lf nb_n \rf+1}^r  {\bs \mu}_W(t_i) \tilde e^{(d_n)}_{i,L}\right| \right\|
= O(\sqrt{n} \tilde m d_n/L) = O(\sqrt{n}(\log n)^{-1/2}).
\end{align}

Therefore, combining the results in \eqref{eq:mlong1L},  \eqref{eq:T1T2L}, \eqref{eq:T1L} and \eqref{eq:T2L}, we have proved \eqref{eq:F2final}. 
\hfill $\Box$
%


\subsection{Proof of \texorpdfstring{\cref{thm:alt_approx}}{Theorem 4.3}}
In order to derive  \cref{thm:alt_approx}, we start by investigating some technical lemmas. \cref{lm:delta_xed} studies the physical dependence of $\mf U^{(d)}(t, \FF_i)$. \cref{lm:consistency} establishes the convergence rate of local linear estimates under the fixed alternative hypothesis. In \cref{lm:alt1}, we derive the uniform Gaussian approximation of the partial sum process of nonparametric residuals. \cref{lm:alt_limit} involves the limiting distribution of a LRD Gaussian process. 

\subsubsection{Some technical lemmas}
\begin{lemma}\label{lm:delta_xed}
  Assuming  
  $ \sup_{ t \in (-\infty,1] }\left\|H\left(t, \mathcal{F}_{0}\right)\right\|_{2p}<\infty
  $, $ \sup_{ t \in [0,1] }\left\|\mf W\left(t, \mathcal{F}_{0}\right)\right\|_{2p}<\infty
  $, $ \delta_{2p}(H, k, (-\infty,1]) = O(\chi^k), \delta_{2p}(\mf W, k) = O(\chi^k), \chi \in (0,1)$,  we have 
  \begin{align}
    \delta_p(\mf U^{(d)}, k) = O(k^{d-1}).
  \end{align}
\end{lemma}
\begin{proof}[Proof]
  Note that for $j \leq i$,
  \begin{align}
    \delta_p(\mf U^{(d)}, i-j) 
    &\leq \|\mf W(t_i,\FF_i)\|_{2p} \delta_{2p}(H^{(d)},i-j)+ \|H^{(d)}(t_i,\FF_{i-j}^*)\|_{2p} \delta_{2p}(\mf W,i-j).\label{eq:Pu}
  \end{align}
  By Burkholder's inequality and \cref{Prop31}, we have 
  \begin{align}
    \|H^{(d)}(t_i,\FF_i)\|^2_{2p}
    &\leq  M\left\|  \sum_{j\in \mathbb{Z}} \left(\mathcal{P}_j H^{(d)} (t_i, \mathcal{F}_i) \right)^2 \right\|_{p} \leq  M \sum_{j\in \mathbb{Z}} \left\| \mathcal{P}_j H^{(d)} (t_i, \mathcal{F}_i) \right\|_{2p}^2  = O(1),
  \end{align}
  where $M$ is a sufficiently large constant.
  
  Then by \cref{Prop31} and \eqref{eq:Pu}, we have proved the desired result.
\end{proof}

\begin{lemma} \label{lm:consistency}
  Under Assumptions \ref{assumptionHp}, \ref{A:K},   \ref{A:beta} and \ref{Ass-W},  $nb_n^2 \to \infty$ and $b_n \to 0$,  we have
  \begin{equation} 
    \sup_{t \in \mathscr{T}}\left|\tilde{{\bs \beta}}_{b_{n}}^{(d)}(t)-{\bs \beta}(t)-\sum_{i=1}^n\frac{\mathbf{M}^{-1}(t)}{n b_{n}} \mathbf{x}_{i} e_{i}^{(d)} K_{b_{n}}^{*}(i / n-t)\right| = \Op\left(\rho^*_n \chi'_n \right),
    \label{eq:expansion_d}
  \end{equation}
  where $\mathscr{T} = [b_n,1-b_n]$, $\rho^*_n = (nb_n)^{d-1/2} \log n \mf 1(0 \leq d \leq 1/26) + (nb_n)^{d-1/2}b_n^{-1/2}\mf 1(1/26 < d < 1/2)+ b_n^2$ and $\chi'_n = n^{-1/2}b_n^{-3/4} + b_n^2$.
  \end{lemma}
  
  \begin{proof}
    According to \cref{thm:fixlocal},when $d \leq 1/26$, take $\alpha_0$ in \cref{thm:fixlocal} to be $31/24 \in (1, 4/3)$,  under the bandwidth condition $nb_n^4/(\log n)^2 \to \infty$, we have $n^{(\alpha_0-1)(\frac{1}{2}-d)}b_n^{d+1/2}\log n \to \infty$. Therefore, by \cref{thm:fixlocal}, we obtain
  \begin{align}
    \sup_{t \in \mathscr{T}}\left| \frac{1}{n b_n} \sum_{i=1}^n \mathbf{x}_i e_i^{(d)} K_{b_n}(t_i-t) \right|  = \Op\left ((nb_n)^{d-1/2}\log n \right).
  \end{align}

  By \cref{lm:delta_xed}, similar arguments in Remark 4 of \cite{wu2007strong} and an application of Propostion B.1 in \cite{dette2018change}, we have 
  \begin{align}
   \sup_{t \in \mathscr{T}}\left| \frac{1}{n b_n} \sum_{i=1}^n \mathbf{x}_i e_i^{(d)} K_{b_n}(t_i-t) \right| = O((nb_n)^{d-1/2}b_n^{-1/2}).
  \end{align}
  The rest of the proof follows from similar procedures in the proof of Theorem 1 in \cite{zhou2010simultaneous}.

  \end{proof}

  \begin{lemma}\label{lm:alt1}
    Define $G_d^*(r)$ as a counterpart of $G^*(r)$,
    \begin{align}
    G_d^*(r)=-\sum_{j=1}^n\left(\frac{1}{nb_n}\sum_{ i=\lf nb_n \rf+1 }^r{\bs \mu}_W^{\T}(t_i)\mf{M}^{-1}(t_i)K_{b_n}^*\left(t_i-t_j\right)\right)\mf x_je_j^{(d)}+\sum_{ i=\lf nb_n \rf+1}^re_i^{(d)}.
    \end{align}
    Under the conditions of \cref{thm:alt_approx}, 
    we have
    \begin{align}
      \max_{\lf nb_n \rf + 1 \leq r \leq n- \lf  nb_n \rf }\lt|G_d^*(r)-\sum_{ i=\lf nb_n \rf+1}^r\tilde e_i^{(d)}\rt| 
      &= \Op (\alpha_n) 
      \end{align}
    where 
      $\alpha_n =  n^{d} b_n^{-2}\log n  + (nb_n)^{d+1/2}\log n + n b_n^3$ when $d \leq 1/26$, 
      $\alpha_n =  n^{d} b_n^{-2}\log n + n^{d+1/2} b_n^{d} + nb_n^3 $, when $1/26 < d < 1/2$.
     and $\tilde e_i^{(d)}$ is the residual under $I(d)$. Under the bandwidth conditions in \cref{thm:alt_approx}, $\alpha_n = o(n^{d+1/2})$.
    \end{lemma}
  
  \begin{proof}
  Similar to \eqref{eq:null_Gstar}, by \cref{lm:consistency}, we have
    \begin{align}
    \max_{\lf nb_n \rf + 1 \leq r \leq n- \lf  nb_n \rf }\lt|G_d^*(r)-\sum_{ i=\lf nb_n \rf+1}^r\tilde e_i^{(d)}\rt|
    \leq \sup_{\lf nb_n \rf + 1 \leq r \leq n- \lf  nb_n \rf } \lt|\tilde M_r^{(d)}\rt|+ \Op(n \rho^*_n \chi'_n), 
    \label{eq:kernel_d}
    \end{align}
    where
    \begin{align}
    \tilde M_r^{(d)}=\sum_{j=1}^n\left(\frac{1}{nb_n}\sum_{ i=\lf nb_n \rf+1 }^r\lt(\mf x^{\T}_i-{\bs \mu}_W^{\T}(t_i)\rt)\mf{M}^{-1}(t_i)K_{b_n}^*\left(t_i-t_j\right)\right)\mf x_je_j^{(d)}.
    \end{align}
    Let ${\bs \xi}_{r,n}(t_j) := \sum_{ i=\lf nb_n \rf+1 }^r \lt (\mf x^{\T}_i-{\bs \mu}_W^{\T}(t_i)\rt)\mf{M}^{-1}(t_i)K_{b_n}^*\left(t_i-t_j\right)$ and ${\bs \xi}_{r,n}(t_0) = 0$, $\sum_{i=1}^{0} \mf x_i e_i^{(d)} = 0$, where $t_j = j/n$.  
     For simplicity, we omit the index of $n$ in ${\bs \xi}_{r,n}(t_j)$.
    Using the summation-by-parts formula, it follows that 
    \begin{align}
      \tilde M_r^{(d)} = \frac{1}{nb_n}{\bs \xi}_{r}(1) \sum_{j=1}^n \mf x_j e_j^{(d)} - \frac{1}{nb_n}\sum_{j=1}^n ({\bs \xi}_{r}(t_j) - {\bs \xi}_{r}(t_{j-1})) \sum_{i=1}^{j-1} \mf x_i e_i^{(d)}
      := Z_1 + Z_2,
    \label{eq:Mrd}
    \end{align}
    where $Z_1$ and $Z_2$ are defined in an obvious way.
    From the proof of  Lemma 6 in \cite{zhou2010simultaneous}, we have for any $1 \leq j \leq n$,
    \begin{align}
     \max_{\lfloor nb_n \rfloor + 1 \leq r \leq n-\lfloor nb_n \rfloor} \left|{\bs \xi}_{r}(t_j)\right| = \Op(\sqrt{n}).
     \label{eq:xi_r}
    \end{align}
  
    From \cref{prop:5.2} and \cref{lm:alt_gaussian}, we have
    \begin{align}
     \max_{1 \leq r \leq n}\left|\sum_{i=1}^r \mf x_i e_i^{(d)}\right|  = \Op(n^{d+1/2}\log n).\label{eq:xed}
    \end{align}
    
    Therefore, 
    \begin{align}\label{eq:Z_1}
      \max_{\lf nb_n \rf  + 1 \leq r \leq n - \lf nb_n \rf} |Z_1| = \frac{1}{nb_n} \max_{1 \leq r \leq n} |{\bs \xi}_r(1)| \left|\sum_{j=1}^n \mf x_j e_j^{(d)}\right| = \Op(b_n^{-1} n^d \log n).
    \end{align}
    Under the continuity of $K^*_{b_n}(\cdot)$, by similar arguments of \eqref{eq:xi_r}, we have for any  $1 \leq j \leq n$, 

    \begin{align}
      &\left \| \max_{\lf nb_n \rf  + 1 \leq r \leq n - \lf nb_n \rf} \left \{\sum_{j=1}^n | {\bs \xi}_r(t_j) - {\bs \xi}_r(t_{j-1})| \right\} \right\|_4
      = O(n^{1/2}b_n^{-1}).\label{eq:xi_diff}
    \end{align}

    Then, by \eqref{eq:xed} and \eqref{eq:xi_diff}, we obtain 
    \begin{align}
       \max_{\lf nb_n \rf  + 1 \leq r \leq n - \lf nb_n \rf} |Z_2|  
       = \Op(b_n^{-2}n^d \log n).\label{eq:Z_2}
    \end{align}
     With \eqref{eq:Mrd}, \eqref{eq:Z_1} and \eqref{eq:Z_2}, it follows that
    \begin{align}\label{eq:Mrtilde}
      \max_{\lf nb_n \rf + 1 \leq r \leq n- \lf  nb_n \rf}| \tilde{M}_r^{(d)}| = \Op (n^{d} b_n^{-2}\log n).
    \end{align}
    From \eqref{eq:Mrtilde} and  \eqref{eq:kernel_d}, we have shown the desired result.
  \end{proof}
  
  \begin{lemma}\label{lm:alt_limit}
    Let $v_i$ be $i.i.d.$ $N(0,1)$ random variables, $$\mf R_{k, n}=\sum_{j=0}^{\infty} {\bs \mu}_W(k/n)\psi_j \sigma_H \left(t_{k-j}\right) v_{k-j},$$   and  $R_{k, n, 1}$ is the first element of $\mf R_{k, n}$. 
    Define $$ \Upsilon_{r,n} =  \sum_{i = \lf nb_n \rf + 1}^r R_{i,n,1} - \sum_{j=1}^n \left (\frac{1}{nb_n} \sum_{ i=\lf nb_n \rf+1}^r \bs \mu^{\T}_W(t_i) \mf{M}^{-1}(t_i) K_{b_n}^*\left(t_i-t_j\right) \right) \mf R_{j,n},$$ and $\Upsilon_{r,n} =  \Upsilon_{\lf nb_n \rf, n} $ for $r < \lf nb_n \rf + 1$, $\Upsilon_{r,n} = \Upsilon_{n- \lf nb_n \rf, n} $ for $r > n - \lf nb_n \rf.$ Let \begin{align}
      \Upsilon_{n}(t) = \Upsilon_{\lf nt \rf,n}\Gamma(d + 1)/n^{d + 1/2}.
    \end{align} \par
    Under the conditions of \cref{thm:alt_approx}, for $d \in (0,1/2)$, we have 
    \begin{align}
      \Upsilon_{n}(t) \leadsto U_d(t) ~\text{on $D[0,1]$ with Skorohod topology}.
    \end{align}
    where $U_d(t)$ is a continuous Gaussian process in $C[0,1]$ with covariance function $(0 \leq r \leq s \leq  1)$
    \begin{align}
      \gamma_d(r,s) = \int_{-\infty}^{+\infty}\sigma^2_H(v)\lambda_d(r,v)\lambda_d(s,v)dv ,
    \end{align}
    where for $v \leq u$, $ u\in [0,1]$,
    \begin{align}
      \lambda_d(u,v) =  d\int_{(-v)_+}^{(u-v)_+}t^{d-1}(\check M_W(t+v)-1)dt,
    \end{align}
      and $\check M_W(t) = {\bs \mu}_W^{\T}(t) \mf{M}^{-1}(t){\bs \mu}_W(t)$, $t \in [0,1]$. 

  \end{lemma}
  
  \begin{proof}
  The limiting distribution of $\Upsilon_{n}(t)$ is derived from the follow procedures. Consider $0 \leq t_1 \leq t_2 \leq 1$. Let $r = \lf nt_1 \rf$, $s = \lf nt_2 \rf$. 
  \begin{enumerate}[label=(\roman*),align=left]
    \item Calculate the covariance $\E \lt \{\frac{\Upsilon_{r,n} \Upsilon_{s,n}}{n^{2d+1}/\Gamma^2(d+1)} \rt\}$  and  establish finite dimensional convergence.
    \item Prove tightness condition.
    \item Show the uniform convergence on $D[0,1]$.
  \end{enumerate}
  
  First, we investigate the terms in $\Upsilon_{r,n}$, when $\lf nb_n\rf + 1 \leq r \leq n- \lf nb_n\rf$,
  \begin{align}
    \Upsilon_{r,n} =  \sum_{i = \lf nb_n \rf + 1}^{r} R_{i,n,1} - \sum_{j=1}^n \mf m_{r,j}^{\T} \mf R_{j,n},\label{eq:decomp}
  \end{align}
  where
  $\mf m_{r,j}^{\T} = \frac{1}{nb_n}\sum_{ i=\lf nb_n \rf+1 }^r{\bs \mu}_W^{\T}(t_i)\mf{M}^{-1}(t_i)K_{b_n}^*\left(t_i-t_j\right)$.\par
  (i) We can write the first term as
  \begin{align}
    \sum_{j=\lf nb_n\rf + 1}^r R_{j,n,1} 
    & = \sum_{j=\lf nb_n\rf + 1}^r \sum_{k = 0}^{\infty} \psi_k \sigma_H(t_{j-k}) v_{j-k}  = \sum_{l=-\infty}^{r} \sigma_H(l/n) v_{l}\sum_{j = (\lf nb_n\rf + 1 -l)_+}^{r-l} \psi_{j} .\label{eq:decomp1}
  \end{align}
  For the second term, 
  let $ \check m_{r,j} = \frac{1}{nb_n}\sum_{ i=\lf nb_n \rf+1 }^r{\bs \mu}_W^{\T}(t_i)\mf{M}^{-1}(t_i)K_{b_n}^*\left(t_i-t_j\right){\bs \mu}_W(t_j)$, we have
  \begin{align}
    \sum_{j=1}^n \mf m_{r, j}^{\T} \mf R_{j,n} 
    & = \sum_{j=1}^n \sum_{k = 0}^{\infty} \psi_k \check m_{r,j} \sigma_H(t_{j-k}) v_{j-k}  = \sum_{l= -\infty}^{n} \sigma_H(l/n) v_{l}\sum_{j = (1 - l)_+ + l}^{n} \psi_{j-l} \check m_{r,j}. \label{eq:decomp2}
  \end{align}

  Consider $\lf n b_n \rf + 1 \leq r \leq s \leq n - \lf n b_n \rf$, the convariance of the Gaussian process is
  \begin{align}
    \E \left \{\frac{\Upsilon_{r,n} \Upsilon_{s,n}}{n^{2d+1}/\Gamma^2(d+1)} \right\} &= \E \left \{ \sum_{j=1}^n \mf m_{r,j}^{\T}\mf R_{j,n} \sum_{j=1}^n \mf m^{\T}_{s,j}\mf R_{j,n} \right\}/(n^{2d+1}/\Gamma^2(d+1))\notag\\ 
    &-  \E \left \{\sum_{j=1}^n \mf m_{r,j}^{\T}\mf R_{j,n}\sum_{ i=\lf nb_n \rf+1}^{s} R_{i,n,1} \right\}/(n^{2d+1}/\Gamma^2(d+1)) \notag\\ 
    &- \E \left \{\sum_{j=1}^n \mf m^{\T}_{s,j}\mf R_{j,n} \sum_{ i=\lf nb_n \rf+1}^{r} R_{i,n,1}\right\}/(n^{2d+1}/\Gamma^2(d+1)) \notag\\ 
    &+  \E \left \{\sum_{ i=\lf nb_n \rf+1}^{r} R_{i,n,1}\sum_{ i=\lf nb_n \rf+1}^{s} R_{i,n,1}\right\}/(n^{2d+1}/\Gamma^2(d+1)) \\ 
    & := I + II + III + IV.
  \end{align}
  
  We first truncate the summands before $l = -N$, $N = \lf n^{\alpha}\rf $, $\alpha \geq 1$. By elementary calculation, we have
  \begin{align}
    I    & = \sum_{l = - N}^{r+ \lf nb_n \rf}\sigma^2_H(l/n) \left(\sum_{i = (1-l)_++l}^{r + \lf nb_n \rf} \psi_{i-l} \check m_{r, i}\right)\left(\sum_{i = (1-l)_+ + l}^{s +  \lf nb_n \rf} \psi_{i-l} \check m_{s, i}\right)/(n^{2d+1}/\Gamma^2(d+1)) + O((N/n)^{2d-1}),
  \end{align}
  where the last equality follows since $\check m_{s,j} = 0$ if $i > s + \nb$.
  
 Next, define \begin{align}
  I^* = \frac{d^2}{n} \sum_{l = -N}^{r}\sigma^2_H(l/n)\int_{(l/n)_+}^{r/n}(t-l/n)^{d-1}\check M_W(t)dt \int_{(l/n)+}^{s/n}(t-l/n)^{d-1}\check M_W(t)dt.
 \end{align}
 We approximate $\sum_{i = (1-l)_++l}^{r + \lf nb_n \rf} \psi_{i-l} \check m_{r, i}$ by considering different regions of $l$, namely $|r-l| \leq nb_n \log n$ and $|r-l| >  nb_n \log n$, $-l \leq nb_n \log n$, and $-l > nb_n \log n$. Then, we have
  \begin{align}
    I = I^* + O(c_n),
  \end{align}
  where $c_n =  b^{d+1}_n\log n  + N/(n^2b_n)+  N b_n/n+ N/n^{d+1} + (N/n)^{2d-1}$.
  Let $$ I^{\diamond} = d^2 \int_{-\infty}^{+\infty}\sigma^2_H(v)\int_{(-v)_+}^{(r/n-v)_+}t^{d-1}\check M_W(t+v)dt \int_{(-v)_+}^{(s/n - v)_+} t^{d-1}\check M_W(t+v )dt dv .$$
   Since $\check M_W(t)$ is nonnegative and Lipschitz continuous on $[0,1]$, we have
  $
    I= I^{\diamond} + O(e_n),
  $
  where
  $e_n = c_n  +  (N/n^2)^{-\frac{1}{d-2}} +  N^{d+1}/n^{d+2}$. The approximation of II-IV follows similarly.
   Now taking $1 \leq  \alpha < \min\left\{\frac{1}{6}, d,  (d+1)^{-1}\right\} + 1$, elementary calculation shows 
  $e_n = o(1)$.
   From the continuity of $\gamma_d$,
     we obtain, for $ 0 \leq t_1, t_2 \leq 1$,
  \begin{align}
    \E\left\{\frac{\Upsilon_{\lf nt_1 \rf,n} \Upsilon_{\lf nt_2 \rf ,n}}{n^{2d+1}/\Gamma^2(d+1)}\right\} \to \gamma_d(t_1, t_2), \quad n \to \infty. \label{eq:alt_cov_lim}
  \end{align}
  Similar to the case under the null hypothesis,
  the finite dimension convergence of the Gaussian process $\Upsilon_n(t)$ to $U_d(t)$ then follows from the Cramer Wold device and equation \eqref{eq:alt_cov_lim}. 
  
  (ii) To prove the tightness  of $\Upsilon_n(t)$, we extend Lemma 2.1 in \cite{taqqu1975weak} to the non-stationary case. We verify equation (13.14) and (13.12) in \cite{billingsley1999convergence}.
  To verify equation (13.14), we need to establish upper bound for 
  \begin{align}
    J_n(t,t_1,t_2) := \E \left| \Upsilon_n(t_2) - \Upsilon_n(t)\right|\left| \Upsilon_n(t) - \Upsilon_n(t_1)\right|,
    \quad 0 \leq t_1 \leq t \leq t_2 \leq 1.
  \end{align}
  By Cauchy-Schwarz inequality, we have 
  \begin{align}
    J_n(t,t_1,t_2) \leq \frac{\Gamma^2(d+1)}{n^{2d+1}} \|\Upsilon_{\lf n t_2 \rf,n}-\Upsilon_{\lf n t \rf,n} \|\|\Upsilon_{\lf n t_1 \rf,n}-\Upsilon_{\lf n t \rf,n} \|.\label{eq:Jn}
  \end{align}
  We proceed to show that, uniformly for $1  \leq r_1 \leq r_2 \leq n$, and $r_1 = \lf n t_1 \rf$, $r_2 = \lf n t_2 \rf$, $t_1, t_2 \in [0,1]$,
  \begin{align}
    \|\Upsilon_{r_2,n} - \Upsilon_{r_1,n}\|^2 = O((r_2 - r_1)^{2d+1}).\label{eq:Upsilon_diff}
  \end{align}
 According to \eqref{eq:decomp}, we have
  \begin{align}
    \|\Upsilon_{r_2,n} - \Upsilon_{r_1,n}\|^2 \leq 2\left(\left\|\sum_{i = r+1}^s R_{i, n,1}\right\|^2 +  \left\|\sum_{j=1}^n (\mf m_{s,j}^{\T} - \mf m_{r,j}^{\T}) \mf R_{j,n}\right\|^2\right).
  \end{align}
  Since $\sum_{l = -\infty}^{r} \left(\sum_{j = r + 1 - l}^{s-l} \psi_j \right)^2 = O((s-r)^{2d+1})$, \eqref{eq:Upsilon_diff} follows from similar calculation in \eqref{eq:tildeG1} and \eqref{eq:tildeG2}.
 
  Then, combining eqution \eqref{eq:Jn} and \eqref{eq:Upsilon_diff},  there exists a sufficiently large positive constant $K$  s.t.
  \begin{align}
    J_n(t,t_1,t_2) \leq K^{2d+1} (t_2 - t)^{d+1/2}(t_1 - t)^{d+1/2} \leq (K t_2 - K t_1)^{2d+1}.
  \end{align}
  Hence, equation (13.14) in \cite{billingsley1999convergence} is satisfied.
  
  We now  verify equation (13.12) in  Theorem 13.5 in \cite{billingsley1999convergence}. 
  For $d < 1/2$, we have
  \begin{align}
    \int_{-\infty}^{\infty}((t - v)_+^d -(-v)_+^d)^2 dv \leq  t^{2d+1} \int_{-\infty}^{\infty}((1+s)_+^d - (s)_+^d)^2 ds = O( t^{2d+1}).\label{eq:fix_int}
  \end{align} 
  Let $\gamma_d^*(t_1, t_2) =:\|U_d(t_2) - U_d(t_1) \|^2. $  Note that $\gamma_d^*(t_1, t_2) = \gamma_d(t_1, t_1) + \gamma_d(t_2, t_2) - 2\gamma_d(t_1, t_2)$.
  Then, since $\check M_W(t)$, $\sigma^2_H(t)$ are bounded and continuous on $[0,1]$ and $(-\infty, 1]$, respectively. For a large constant $C$, it follows that 
  \begin{align}
    \gamma_d^*(t_1, t_2)& = \int_{-\infty}^{\infty} \sigma_H^2(v)(\lambda_d(t_2, v) - \lambda_d(t_1, v))^2 dv\\ 
    & =    \int_{-\infty}^{\infty} \sigma_H^2(v) \left\{(t_2-v)_+^d - (t_1-v)_+^d- d\int_{(t_1 -v)_+}^{(t_2-v)_+}t^{d-1}\check M_W(t+v)dt \right\}^2 dv\\
    &\leq C \int_{-\infty}^{\infty}((\tau-v)_+^d - (-v)_+^d)^2dv = O(\tau^{2d+1}).\label{eq:gammad}
  \end{align}

  For Equation (13.12) in \cite{billingsley1999convergence}, by \eqref{eq:gammad}, we have
  \begin{align}
    \|U_d(t_2) - U_d(t_1)\|^2 = O((t_2 - t_1)^{2d+1}).\label{eq:Ud_diff}
  \end{align}
  Then, by Chebyshev's inequality, it follows that for any $\epsilon>0$,
  \begin{align}
    \lim _{\delta \rightarrow 0} P[U_d:|U_d(1)-U_d(1-\delta)| \geq \epsilon] \leq \lim _{\delta \rightarrow 0} { \delta}^{2d+1}/\epsilon^2=0,
  \end{align}
  which satisfies equation (13.12) in  Theorem 13.5 in \cite{billingsley1999convergence}.
  
  (iii) 
  From (i), we obtain the finite dimensional convergence. From (ii), we've proved that the $\Upsilon_n(t)$ is tight. Kolmogorov-Chentsov theorem in \cite{karatzas1988brownian} and  \eqref{eq:Ud_diff} guarantee that the existence of $U_d(t)$ which has a continuous trajectory. Then, $U_d(t) \in C[0,1] \subset D[0,1].$ According to Theorem 13.5 in \cite{billingsley1999convergence}, we have 
  \begin{align}
    \Upsilon_{n}(t) \leadsto U_d(t) ~ \text{on $D[0,1]$ with Skorohod topology}.
  \end{align}
        \end{proof}
\subsubsection{Proof of \texorpdfstring{\cref{thm:alt_approx}}{Theorem 4.3}}

      Recall that $\mf m_{r,j}^{\T} = \frac{1}{nb_n}\sum_{ i=\lf nb_n \rf+1 }^r{\bs \mu}_W^{\T}(t_i)\mf{M}^{-1}(t_i)K_{b_n}^*\left(t_i-t_j\right)$ as defined in \eqref{eq:GA1}. Define
      \begin{align}
      G_d^*(r)=-\sum_{j=1}^n \mf m_{r,j}^{\T} \mf x_je_j^{(d)}+\sum_{ i=\lf nb_n \rf+1}^re_i^{(d)}.\label{eq:Gd}
      \end{align}
      It follows from \cref{lm:alt1} that
      \begin{align}
        \max_{\lf nb_n \rf + 1 \leq r \leq n- \lf  nb_n \rf }\left|G_d^*(r)-\sum_{ i=\lf nb_n \rf+1}^r\tilde e_i^{(d)}\right| 
        = \Op (\alpha_n),\label{eq:e2G}
      \end{align} 
      where $\alpha_n $ is of smaller order of $n^{d+1/2}$.\    According to \cref{thm:fixlocal}, similar to the proof of \eqref{eq:GA1} using the summation-by-parts formula, there exists
      a series of $i.i.d.$ $N(0,1)$'s $\{v_i\}_{i\in \mathbb Z}$ possibly on a richer probability space,
      such that 
      \begin{align}
          \max_{\lf nb_n \rf + 1 \leq r \leq n- \lf  nb_n \rf }\left|G_d^*(r)- \Upsilon_{r,n}\right|  = \Op(\sqrt{n} (\log n)^d + n^{1+\alpha_0(d-1/2)}),\label{eq:G2U}
      \end{align}
      where \begin{align}\label{eq:UpsilonR}
        \Upsilon_{r,n} =  \sum_{i = \lf nb_n \rf + 1}^r R_{i,n,1} - \sum_{j=1}^n \mf m_{r,j}^{\T} \mf R_{j,n},
      \end{align}
      and $\mf R_{k, n}=\sum_{j=0}^{\infty} {\bs \mu}_W(k/n)\psi_j \sigma\left(t_{k-j}\right) v_{k-j}$, $R_{k, n, 1}$ is the first element of $\mf R_{k, n}$.
      Since by \cref{lm:alt_limit} $\|\Upsilon_{n,n}\|$ is of order $n^{d+1/2}$, we have
      \begin{align}
        \left| T_{n,b_n} ^{(d)}- \Xi_{n,b_n}\right| &\leq \max_{\lf nb_n \rf + 1 \leq r \leq n- \lf  nb_n \rf }\left|\left(\sum_{ i=\lf nb_n \rf+1}^r\tilde e_i^{(d)}\right)^2-\Upsilon^2_{r,n}\right|/n\\
        & \leq  \max_{\lf nb_n \rf + 1 \leq r \leq n- \lf  nb_n \rf }\left|\sum_{ i=\lf nb_n \rf+1}^r\tilde e_i^{(d)}-\Upsilon_{r,n} \right|^2/n +  2 \max_{\lf nb_n \rf + 1 \leq r \leq n- \lf  nb_n \rf }  \left|\sum_{ i=\lf nb_n \rf+1}^r\tilde e_i^{(d)}-\Upsilon_{r,n}\right||\Upsilon_{r,n}|/n\\ 
        &= \Op\left\{n^{d-1/2}\left(\alpha_n +\sqrt{n} (\log n)^d + n^{1+\alpha_0(d-1/2)}\right)\right\} = \op(n^{2d}).
      \end{align}
      The second part of the proof follows from \cref{lm:alt_limit} and the continuous mapping theorem. \hfill $\Box$

      \subsection{Proof of \texorpdfstring{\cref{thm:local}}{Theorem 4.4}}

        As a counterpart of \eqref{eq:Gd}, define
        \begin{align}
        G_{d_n}^*(r)=-\sum_{j=1}^n \mf m_{r,j}^{\T} \mf x_je_j^{(d_n)}+\sum_{ i=\lf nb_n \rf+1}^re_i^{(d_n)}.
        \end{align}
        It follows from \cref{lm:alt1} that
        \begin{align}
          \max_{\lf nb_n \rf + 1 \leq r \leq n- \lf  nb_n \rf }\left|G_{d_n}^*(r)-\sum_{ i=\lf nb_n \rf+1}^r\tilde e_i^{(d_n)}\right| 
          = \Op (b_n^{-2}\log n  + (nb_n)^{1/2}\log n + n b_n^3),
        \end{align} 
      which is of smaller order of $n^{1/2}$  when $nb_n^4/(\log n)^2 \to \infty$, $nb_n^6 \to 0$.
       According to \cref{thm:fixlocal}, using the summation-by-parts formula, there exists
       a series of $i.i.d.$.  Gaussian vectors namely $\{\mf V_i\}_{i\in \mathbb Z}$ possibly on a richer probability space,
       such that 
       \begin{align}
           \max_{\lf nb_n \rf + 1 \leq r \leq n- \lf  nb_n \rf }\left| G_{d_n}^*(r)-  \Upsilon^{\circ}_{r,n}\right|  = \op(\sqrt{n}),
       \end{align}
      where \begin{align}
         \Upsilon^{\circ}_{r,n} =  \sum_{i = \lf nb_n \rf + 1}^r \tilde R_{i,n,1} - \sum_{j=1}^n \mf m_{r,j}^{\T} \mf{\tilde R}_{j,n},~ \mf{\tilde R}_{i,n} &= \sum_{j = 1}^{\infty} {\bs \mu}_W(t_i) \psi_{j}\sigma_H(t_{i-j}) V_{i - j ,1} + {\bs \Sigma}^{1/2}(t_i) \mf V_i:= \mf S_{i, n}  + {\bs \Sigma}^{1/2}(t_i) \mf V_i.
       \end{align}
       with $\tilde R_{i,n, 1} $ and $V_{i,1} $ being the first element of $\mf{\tilde R}_{i,n}$ and $\mf V_{i}$. Extra define  $\Upsilon^{\circ}_{r,n} =  \Upsilon^{\circ}_{\lf nb_n \rf, n} $ for $r < \lf nb_n \rf + 1$, $\Upsilon^{\circ}_{r,n} = \Upsilon^{\circ}_{n- \lf nb_n \rf, n} $ for $r > n - \lf nb_n \rf.$
       Then, let 
       \begin{align}
         \Upsilon^{\circ}_{n}(t) = n^{-1/2}  \Upsilon^{\circ}_{\lf nt \rf,n}.
       \end{align}
       Similar to \cref{lm:alt_limit}, the limiting distribution of $\Upsilon^{\circ}_{n}(t)$ is derived from the follow procedures. Consider $0 \leq t_1 \leq t_2 \leq 1$. Let $r = \lf nt_1 \rf$, $s = \lf nt_2 \rf$. 
       \begin{enumerate}[label=(\alph*),align=left]
         \item Calculate the covariance $n^{-1}\E\{ \Upsilon^{\circ}_{r,n} \Upsilon^{\circ}_{s,n}\}$ and establish finite dimiensional convergence of $ \Upsilon^{\circ}_{n}(t)$.
         \item Prove tightness condition of $ \Upsilon^{\circ}_{n}(t)$.
       \end{enumerate}
      Then, by  (a) and (b),  we have
      \begin{align}
       \Upsilon^{\circ}_n(t) \leadsto U^{\circ}(t) ~ \text{on $D[0,1]$ with Skorohod topology.}
      \end{align}
      \par
      \textbf{Step (a)}.
      Let $S_{k, n, 1}$ be  the first element of $\mf S_{k, n}$.
      Observe that 
      \begin{align}
      \Upsilon^{\circ}_{r,n} = \tilde G^*(r) + \check \Upsilon_{r,n},
      \end{align}
      where 
      \begin{align}
        \tilde G^*(r) = \sum_{i = \lf nb_n \rf + 1}^r \sigma_H(t_i) V_{i,1} - \sum_{j=1}^n \mf m_{r,j}^{\T} {\bs \Sigma}^{1/2}(t_j)\mf V_j, \quad \check \Upsilon_{r,n} =  \sum_{i = \lf nb_n \rf + 1}^r  S_{i,n,1} - \sum_{j=1}^n \mf m_{r,j}^{\T} \mf{S}_{j,n}.
      \end{align}
      The convariance of the Gaussian process is
       \begin{align}
         n^{-1}\E \left \{\Upsilon^{\circ}_{r,n} \Upsilon^{\circ}_{s,n} \right\} &= n^{-1}\E \left\{(\tilde G^{*}(r) + \check \Upsilon_{r,n})(\tilde G^{*}(s) + \check \Upsilon_{s,n}) \right\}.
         \label{eq:local_cov}
       \end{align}
      \par
       \textbf{Step a.1: Compute $\E\{\tilde G^*(s)\tilde G^*(r) \}/n$.}
      According to \eqref{eq:lim_cov1} of the proof of \cref{thm:null_dist}, we have 
      \begin{align}
      \max_{1 \leq  r \leq  s \leq n} \left|\E\{\tilde G^*(s)\tilde G^*(r) \}/n - \gamma(t_1, t_2)\right| = O\left(b_n+\frac{1}{n b_n}\right).\label{eq:stepa1}
      \end{align}
      \par
      \textbf{Step a.2: Compute $\E\{ \check \Upsilon_{r,n}\check \Upsilon_{s,n}\}/n$.}
      Following similar arguments in \cref{lm:alt_limit}
      and some tedious calculation, we have
      \begin{align}
      \sup_{0 \leq t_1 \leq t_2 \leq 1}\left|n^{-1} \E \left \{\check \Upsilon_{\lf nt_1 \rf ,n} \check  \Upsilon_{\lf nt_2 \rf,n} \right\} - \check \gamma(t_1, t_2) \right|=O((\log n)^{-1/2}).\label{eq:stepa2}
      \end{align}
     
      \textbf{Step a.3: Compute $\E\{ \check \Upsilon_{r,n}\tilde G^*(s)\}/n$ and  $\E\{ \check \Upsilon_{s,n}\tilde G^*(r)\}/n$.} Similar to \eqref{eq:lim_cov1} of the proof of \cref{thm:null_dist}
      and some tedious calculation, we have
        \begin{align}
        \sup_{0 \leq t_1 \leq t_2 \leq 1} \left| n^{-1}\E \left \{\check \Upsilon_{\lf nt_1 \rf,n} \tilde G^*(\lf nt_2 \rf) \right\} -  \tilde \gamma(t_1, t_2) \right| = O((\log n)^{-1/2}).
        \end{align}
        Similarly, we have 
        \begin{align}
        \sup_{0 \leq t_1 \leq t_2 \leq 1} \left| n^{-1}\E \left \{\check \Upsilon_{\lf nt_2 \rf,n} \tilde G^*(\lf nt_1 \rf) \right\} -  \tilde \gamma(t_1, t_2) \right| = O((\log n)^{-1/2}).\label{eq:stepa3}
        \end{align}
       Combining \eqref{eq:local_cov}, \eqref{eq:stepa1}, \eqref{eq:stepa2} and \eqref{eq:stepa3},  we have
       \begin{align}
         \sup_{0 \leq t_1 \leq t_2 \leq 1} \left| n^{-1}\E \left \{\Upsilon^{\circ}_{\lf nt_1 \rf,n} \Upsilon^{\circ}_{\lf nt_2 \rf,n} \right\}  -\gamma^{\circ} (t_1, t_2) \right|= O((\log n)^{-1/2}).
       \end{align}
      The finite dimensional convergence of $\Upsilon^{\circ}(t)$  to $U^{\circ}(t)$ then follows from Cramer-Wold device.
      
      \textbf{Step (b)}. 
      Since $\sum_{l=1}^{\infty}((s-r-1+l)^{d_n} - l^{d_n})^2= o(s-r)$, (13.4) of Theorem 13.5 in \cite{billingsley1999convergence} with the $\alpha = \beta = 1$ case follows from \eqref{eq:null_tightness}, and  calculations similar to step (ii) of the proof of  \cref{lm:alt_limit}. 
  Equation (13.2) of \cite{billingsley1999convergence} follows from the continuity of the covariance structure of $U^{\circ}(t)$. Therefore, by Theorem 13.5 in \cite{billingsley1999convergence}, we haven shown the tightness of $\Upsilon^{\circ}(t)$.\hfill $\Box$
 
\section{Theoretical properties of bootstrap tests}\label{bootstrap}

\subsection{Proof of \texorpdfstring{\cref{thm:bootstrap_null}}{Theorem 5.1}}
Proof of (i) Under \cref{ass:lrv}
      the proof follows from  similar but simpler arguments of
      \cref{thm:bootstrapA}. We omit it for brevity.\\
Proof of (ii). 
Given (ii) of \cref{thm:null_dist}, it suffices to show that on a possibly richer probability space,
      \begin{align}
        |\tilde T_n - T_n| = \op(b_n).\label{degnull:step1}
      \end{align}
    Let $\tilde G_r$ be the bootstrap statistic in one iteration defined in \eqref{eq:Gk} in one simulation round. To prove \eqref{degnull:step1}, it's sufficient to show that on a possibly richer probability space,
      \begin{align}
        \sup_{\nb + 1 \leq r \leq n-\nb} \left| \left(\sum_{i=\nb +1}^r \tilde e_i - \tilde G_r\right)\right| = \op(\sqrt{nb_n} / \log n).\label{deg:step1.1}
      \end{align}
      and 
      \begin{align}
        \sup_{\nb + 1 \leq r \leq n-\nb} |\tilde G_r| = \Op(\sqrt{nb_n\log n}). \label{deg:step1.2}
      \end{align}
      Then, the theorem follows from continuous mapping theorem. 
    Recall in Step 2 of \cref{thm:null_dist}, 
    \begin{align}\tilde G^*(r)=-\sum_{j=1}^n\left(\frac{1}{nb_n}\sum_{ i=\lf nb_n \rf+1 }^r{\bs \mu}_W^{\T}(t_i)\mf{M}^{-1}(t_i)K_{b_n}^*\left(t_i-t_j\right)\right) {\bs \Sigma}^{1/2}(t_j)\mf{V}_j+\sum_{ i=\lf nb_n \rf+1}^r\sigma_{H}(t_i)V_{i,1}.
    \end{align}
    Recall in  \eqref{deg:step1.1-1}, we have
    \begin{align}
      \sup_{\nb + 1 \leq r \leq n-\nb} \left|\sum_{i=\nb +1}^r \tilde e_i - \tilde G^*(r)\right| = \op(\sqrt{nb_n}/\log n).\label{deg:step1.1prime}
    \end{align}
    Following the proof of \eqref{eq:totalrate}, we have
    \begin{align}
      \sup_{\nb + 1 \leq r \leq n-\nb} \left|n^{-1/2}(\tilde G_r - \tilde G^*(r)) \right|  
      &= \Op\left(q_n\left( r_n + g_n^{1/2} \right)+ n^{-1/2}b_n^{-3/2}\right)\\ 
      &= \op((b_n^{1/2}/\log n)),\label{deg:step1.1-2}
    \end{align}
    where $r_n  = n^{-1/2} b_n^{-3/4} + b_n^2$, $q_n$ is a sequence goes to infinity at an arbitrarily slow rate, $g_n$ is related to the convergence rate of \cref{ass:lrv} which is $o(b_n/\log^2 n)$.
    Therefore, \cref{deg:step1.1} is proved by \eqref{deg:step1.1prime} and \eqref{deg:step1.1-2}.
Recall that in \eqref{deg:step1.2-1}, we have
\begin{align}
    \|\tilde G^*(r) - \tilde G^{\circ}(r) \|^2 = O(nb_n^3),
\end{align}
where
$
  \tilde G^{\circ}(r) = \sum_{i=\nb+1}^r \sigma_H(t_i)V_{i,1} - \sum_{j=1}^n \sum_{i=\nb+1}^r \frac{K^*_{b_n}(t_i - t_j)}{nb_n} \sigma_H(t_j)V_{j,1},
$
and for $\lf nb_n\rf+1\leq r\leq n-\lf nb_n\rf$,
\begin{align}
  \| \tilde G^{\circ}(r) \|^2 = O(nb_n). 
\end{align}
Since $\tilde G^*(r) - \tilde G^{\circ}(r)$ and $\tilde G^{\circ}(r)$ are Gaussian processes, \cref{deg:step1.2} is proved. \hfill $\Box$

  \begin{remark}\label{deg:relate}
    The bootstrap consistency of V/S-type test follows similarly.  In the following remark, we take K/S-type test as an example, the result of R/S-type test follows similarly. Let 
    \begin{align}
        Y = (|\tilde G^*(3\nb+1), \cdots, \tilde G^*(n - \nb)|)/\sqrt{nb_n},
    \end{align}
    and 
    \begin{align}
        \widetilde{\mathrm{KS}}_n = \max_{3\nb +1 \leq k \leq n - \nb}|\tilde G_k|,\quad \mathrm{KS}_n = \max_{3\nb +1 \leq k \leq n - \nb}|\tilde G_k|,\quad \widetilde{\mathrm{KS}}^*_n = \max_{3\nb +1 \leq k \leq n - \nb}|\tilde G^*_k|.
    \end{align}
    Since by elementary calculation using \eqref{deg:step1.2-1}, \eqref{deg:coef} and degeneracy there exist positive constants $c_1$ and $c_2$ such that for $3\lf nb_n\rf+1\leq r\leq n-\lf nb_n\rf$,
    $ c_1 \leq \|\tilde G^*(r) \|^2/(nb_n)\leq c_2$, by Lemma C.1 in \cite{Dette2021ConfidenceSF}, we have 
    \begin{align}
        \sup_{x\in \mathbb{R}} \left|P(|Y|_{\infty} > x)  - P(\mathrm{KS}_n/\sqrt{nb_n} > x)\right| &\leq  \sup_{x\in \mathbb{R}} \left|P(|Y|_{\infty} > x) - P(\widetilde{\mathrm{KS}}_n/\sqrt{nb_n} > x)\right|\\  & + P(|\widetilde{\mathrm{KS}}_n - \mathrm{KS}_n| > \sqrt{nb_n} \delta) + c \delta \sqrt{\max\{1,\log (n/\delta)\}},
    \end{align}
    where the right-hand side converges to $0$ as $n \to \infty$, if we let $\delta = \log^{-1} n$, and use \eqref{deg:step1.1} and \eqref{deg:step1.1-2}.
  \end{remark}

       \subsection{Proof of \texorpdfstring{\cref{thm:bootstrapA}}{Theorem 5.2}}
Recall that  $\mf m_{r,j}^{\T} = \frac{1}{nb_n}\sum_{ i=\lf nb_n \rf+1 }^r{\bs \mu}_W^{\T}(t_i)\mf{M}^{-1}(t_i)K_{b_n}^*\left(t_i-t_j\right)$.
        Define 
        $ \tilde T^*_{n,m} =  \sum_{r=\lf nb_n \rf + 1}^{ n - \lf nb_n\rf} (\tilde G^*_{r,d})^2 /(n(n - 2\lf nb_n\rf),$ where
        $$ \tilde G^*_{r,d}=-\sum_{j=1}^n \mf m_{r,j}^{\T}  {\bs \Sigma}_d^{1/2}(t_j)\mf{V}_j + \sum_{ i=\lf nb_n \rf+1}^r  \sigma_{Hd}(t_i)V_{i,1}, $$
        with ${\bs \Sigma}_d(t) = \kappa_2(d)\sigma_H^2(t) {\bs \mu}_W(t)\bs \mu^{\T}_W(t)$, $\sigma_{Hd}(t) = \left(\bs \Sigma_d(t)\right)_{(1,1)}$.
      
        The proof consists of three parts.\par
        (a) Obtain the limiting distribution of $n^{-1/2}\tilde G_{r,d}^* $, $\lf nb_n\rf + 1 \leq r \leq n - \lf nb_n\rf$.\par
      
        (b) Show that conditional on data, $m^{-d}n^{-1/2}\tilde G_{r,d} $ and  $n^{-1/2}\tilde G_{r,d}^* $ converge to the same limit, uniformly for $\lf nb_n\rf + 1 \leq r \leq n - \lf nb_n\rf$.\par
        (c) Derive the limiting distribution of $\tilde T_{n}$.
      \par 
       \textbf{Step  (a)}. By similar arguments as \eqref{eq:nullI}, \eqref{eq:nullII}, \eqref{eq:nullII},  and \eqref{eq:nullIV} in the proof of \cref{thm:null_dist}, by Cramer-Wold device, we have the finite dimensional convergence. The tightness follows similarly as in the final part of the proof of \cref{thm:null_dist}. Then, we have $n^{-1/2} \tilde G_{r,d}^* \leadsto \tilde U_d(t)$ on $D[0,1]$ with Skorohod topology.
        \par
        \par
        \textbf{Step of (b)}.
         Let $\mf 1$ denote the indicator function. Let $C$ below denote a sufficiently large constant in the following context. We construct sets independent of $\{\mf{V}_j\}_{j=1}^n$ as follows. Let $q_n$ be a sequence of real numbers so that $q_n \to \infty$ arbitrarily slow.
        Define
        \begin{align}
          W_n := \{\sup_{t \in \I}|\hat {\bs \Sigma}_d^{1/2} (t) m^{-d} - {\bs \Sigma}_d^{1/2} (t)| \leq g^{1/2}_{1,n} q_n\},  \quad
          H_n := \{\sup_{t \in \B}\rho(\hat{\mf M}^{-1} (t) - \mf M^{-1} (t)) \leq r_n q_n\},
        \end{align}
       where $\mathcal{I} = [\gamma_n, 1-\gamma_n] \subset (0,1)$, $\gamma_{n}=\tau_{n}+(m+1) / n$, $\B = [\eta_n , 1-\eta_n]$, $ g_{1, n} = o(1)$ related to the convergence rate of \cref{ass:lrv},  $r_n = n^{-1/2}\eta_n^{-3/4} + \eta_n^2$.
        
        Under \cref{ass:lrv}, 
by Gershgorin’s circle theorem and  Corollary 1 in \cite{yu2015useful}, we obtain
        \begin{align}
          \lim_{n \to \infty} \pp(W_n) = 1, \quad \lim_{n \to \infty} \pp(H_n) = 1.\label{eq:WnHn}
        \end{align}
        Observe that
        \begin{align}
          (m^{-d}\tilde G_{r,d} - \tilde G_{r,d}^*) \mf  1(W_n \cap H_n) & = \sum_{j=1}^n \left\{\mf m_{r,j}^{\T}  {\bs \Sigma}_d^{1/2}(t_j)  - \hat {\mf m}_{r,j}^{\T} m^{-d} \hat {\bs \Sigma}_d^{1/2}(t_j)\right\} 1(W_n \cap H_n)\mf V_j \\ & + \sum_{i=1}^r(\hat \sigma_{Hd}(t_j) m^{-d} - \sigma_{Hd}(t_j)) 1(W_n \cap H_n) V_{i,1}\\
        &= \sum_{j=1}^n \left\{(\mf m_{r,j}^{\T}   - \hat {\mf m}_{r,j}^{\T}) m^{-d} \hat {\bs \Sigma}_d^{1/2}(t_j)\right\} 1(W_n \cap H_n)\mf V_j \\ & +
        \left\{ \sum_{j=1}^n \mf m_{r,j}^{\T}  \lt({\bs \Sigma}_d^{1/2}(t_j)  -  m^{-d} \hat {\bs \Sigma}_d^{1/2}(t_j)\rt) 1(W_n \cap H_n)\mf V_j \right.\\ & +
        \left.\sum_{i=1}^r(\hat \sigma_{Hd}(t_j) m^{-d} - \sigma_{Hd}(t_j)) 1(W_n \cap H_n) V_{i,1}\right\}
        := J_1 + J_2,
        \end{align}
      where $J_1$ and $J_2$ are defined in the obvious way.
         Let $\mathbb{T}_n = [\lf n\gamma_n \rf + 1, n - \lf n\gamma_n \rf] \cap [\lf n\eta_n \rf + 1, n - \lf n\eta_n \rf] $,  and consider $r \in \mathbb{T}_n$. 
         First, we show calculate $J_1$ by the following two steps. Define 
        \begin{align}
          \tilde{\mf{m}}_{r,j}^{\T} = \frac{1}{nb_n}\sum_{ i=\lf nb_n \rf+1 }^r \mf x^{\T}_i \mf{M}^{-1}(t_i)K_{b_n}^*\left(t_i-t_j\right).
      \end{align}
      (1) We shall show that
      \begin{align}
      \left\|\max_{r \in \mathbb{T}_n}\left| \sum_{j=1}^n   \left(\hat {\mf m}_{r,j}^{\T} - \tilde{\mf m}_{r,j}^{\T}\right)m^{-d} \hat{\bs \Sigma}_d^{1/2}(t_j) \mf 1( W_n \cap H_n) \mf V_j \right| \right\| = O(n^{1/2}r_nq_n).\label{eq:J12}
      \end{align}
      (2) Then, we shall derive 
      \begin{align}
      \left\| \max_{r \in \mathbb{T}_n}\left|\sum_{j=1}^n   \left(\tilde {\mf m}_{r,j}^{\T} - \mf m_{r,j}^{\T}\right)m^{-d}\hat{\bs \Sigma}_d^{1/2}(t_j)\mf 1( W_n \cap H_n) \mf V_j\right| \right\| = O(b_n^{-3/2}).\label{eq:J13}
      \end{align}
      Proof of \eqref{eq:J12}. Observe that the left hand side of \eqref{eq:J12} is a martingale
       w.r.t $\mathcal{G}_r = \{\FF_n, \{V_{i,1}\}_{i=1}^r\}$. By Doob's inequality and Burkholder inequality, for a sufficiently large constant $C$, we have
       \begin{align}
           &\left\|\max_{r \in \mathbb{T}_n}\left| \sum_{j=1}^n   \left(\hat {\mf m}_{r,j}^{\T} - \tilde{\mf m}_{r,j}^{\T}\right)m^{-d} \hat{\bs \Sigma}_d^{1/2}(t_j) \mf 1( W_n \cap H_n) \mf V_j \right| \right\|\\ 
           &=  \left\|\max_{r \in \mathbb{T}_n}\left|\frac{1}{nb_n}\sum_{ i=\lf nb_n \rf+1 }^r \mf x^{\T}_i (\hat {\mf M}^{-1}(t_i) - \mf{M}^{-1}(t_i))\sum_{j=1}^n  K_{b_n}^*\left(t_i-t_j\right)   m^{-d} \hat{\bs \Sigma}_d^{1/2}(t_j) \mf 1( W_n \cap H_n) \mf V_j \right| \right\|\\
           & \leq C \left\|\left|\frac{1}{nb_n}\sum_{ i=\lf nb_n \rf+1 }^{n - \nb} \mf x^{\T}_i (\hat {\mf M}^{-1}(t_i) - \mf{M}^{-1}(t_i))\sum_{j=1}^n  K_{b_n}^*\left(t_i-t_j\right)   m^{-d} \hat{\bs \Sigma}_d^{1/2}(t_j) \mf 1( W_n \cap H_n) \mf V_j \right| \right\|\\
           & \leq C \left\{\sum_{j=1}^n \left\|\left(\hat {\mf m}_{n,j}^{\T} - \tilde{\mf m}_{n,j}^{\T}\right)m^{-d} \hat{\bs \Sigma}^{1/2}_d(t_j)\mf 1( W_n \cap H_n) \right\|^2\right\}^{1/2}\\
           & = O(n^{1/2}r_nq_n).
       \end{align}
      Proof of \eqref{eq:J13}.
      Let $\boldsymbol{\mu}_{b_{n}}^{\dagger}(t) =  \frac{1}{nb_n}\sum_{j=1}^n   K^*_{b_n}(t - t_j)m^{-d} \hat{\bs \Sigma}_d^{1/2}(t_j) \mf 1( W_n \cap H_n)\mf V_j $.
      By summation-by-parts formula, 
      \begin{align}
      \sum_{j=1}^n   \left(\tilde {\mf m}_{r,j}^{\T} - \mf m_{r,j}^{\T}\right)m^{-d} \hat{\bs \Sigma}_d^{1/2}(t_j)\mf 1( W_n \cap H_n) \mf V_j  
      & =  \sum_{i=1}^r (\mf x_i - \bs \mu_i)^{\T} \mf M^{-1}(t_i)\boldsymbol{\mu}_{b_{n}}^{\dagger}(r/n) \\
      & = \sum_{i=1}^r (\mf x_i - \bs \mu_i)^{\T}\mf M^{-1}(t_i) \boldsymbol{\mu}_{b_{n}}^{\dagger}(r/n)\\& -\sum_{i=1}^r  (\boldsymbol{\mu}_{b_{n}}^{\dagger}(t_i) -\boldsymbol{\mu}_{b_{n}}^{\dagger}(t_{i-1}) )\sum_{k=1}^{r-1}(\mf x_i - \mu_i)^{\T} \mf M^{-1}(t_i)\\
      & := Z_1 + Z_2
      \end{align}
      Under condition \ref{A:W} and \ref{A:Mt}, similar to proof of \eqref{eq:J12}, by Doob's inequality and Cauchy inequality, we have 
      \begin{align}
      \left\|\max_{r \in \mathbb{T}_n}\left|Z_1 \right|\right\| &\leq  \left\|\left|\sum_{i=1}^{n-\nb} (\mf x_i - \mu_i)^{\T}\mf M^{-1}(t_i)\right| \right\|_4 \left\|\boldsymbol{\mu}_{b_{n}}^{\dagger}((n-\nb)/n) \right\|_4 = O(b_n^{-1/2}),\label{eq:Z1}
      \end{align}
      and  by Cauchy inequality,  we have
      \begin{align}
      \left\|\max_{r \in \mathbb{T}_n}\left|Z_2 \right|\right\| &\leq \left\|\max_{r \in \mathbb{T}_n}\left|\sum_{i=1}^r (\mf x_i - \mu_i)^{\T}\mf M^{-1}(t_i)\right| \right\|_4 \left\|\max_{r \in \mathbb{T}_n}\sum_{i=1}^r \left|\boldsymbol{\mu}_{b_{n}}^{\dagger}(t_i) -\boldsymbol{\mu}_{b_{n}}^{\dagger}(t_{i-1}) \right|\right\|_4 = O(b_n^{-3/2}),\label{eq:Z2}
      \end{align}
      where the last equality follows from  Doob's inequality, triangle and Burkholder inequality.
      By \eqref{eq:Z1} and \eqref{eq:Z2}, \eqref{eq:J13} is proved.
      Similar to the calculation of $J_1$, we have 
      \begin{align}
         \| \max_{r \in \mathbb{T}_n}|J_2|\|  = O(n^{1/2}g_{1,n} q_n).\label{eq:J2}
      \end{align}
      Therefore, combining  \eqref{eq:J12}, \eqref{eq:J13},  and\eqref{eq:J2}, by triangle inequality, 
      \begin{align}
      \left\|\max_{r \in \mathbb{T}_n}\left|  n^{-1/2}(m^{-d}\tilde G_{r,d} - \tilde G_{r,d}^*) \mf  1(W_n \cap H_n) \right|\right\| = O(g_{1, n}^{1/2}q_n + r_n q_n + n^{-1/2}b_n^{-3/2} ).
      \end{align}

      By Proposition A.1 in \cite{wu2018gradient}, since \eqref{eq:WnHn}, we have
      \begin{align}
      \max_{r \in \mathbb{T}_n}\left|  n^{-1/2}(m^{-d}\tilde G_{r,d} - \tilde G_{r,d}^*) \right| = \Op(g_{1, n}^{1/2}q_n + r_n q_n + n^{-1/2}b_n^{-3/2} ).\label{eq:totalrate}
      \end{align}
      \par
      \par
      \textbf{Step (c)}
      Under the bandwidth condition $nb_n^3 \to \infty$,
      by Step (a) and (b), $n^{-1/2}m^{-d} \tilde G_{r,d} \leadsto \tilde U_d(t)$ on $D[0,1]$ with Skorohod topology.
      Therefore, by continuous mapping theorem, we have 
      \begin{align}
      m^{-2d} \tilde T_{n} \Rightarrow \int_0^1 \tilde U^2_d(t) dt.
      \end{align}
      \par

Under \cref{ass:lrv}, following similar arguments in the proof of result (i) of \cref{thm:bootstrapA}, we have result (ii).\hfill $\Box$
      \subsection{Proof of \texorpdfstring{\cref{deg:alt}}{Proposition 5.1}}

    Proof of (i).
     Recall  $\tilde G_{r,d}$ is the bootstrap statistic \eqref{eq:Gk} in one iteration under the fixed alternative.  It's sufficient to show that,
       \begin{align}
        \lim_{n \to \infty} P\left(\left|\sum_{i=\nb +1}^r \tilde e_i^{(d)}\right| >  \left|\tilde G_{r,d} \right|\text{for $r = 
    \nb +1, \cdots n - \nb$}\right) = 1.
      \end{align}
      Similar to the steps in the proof of (ii) of \cref{thm:bootstrap_null}, substituting $\sigma_H(\cdot)$ by $\hat \sigma_d(\cdot)$, using law of total expectation, we have 
      \begin{align}
       \max_{\nb +1 \leq r \leq n - \nb} \|\tilde G_{r,d}\|^2 = O(nb_n m^{2d}).\label{deg:step1}
      \end{align}
      Similar to \cref{lm:consistency}, for the time varying trend model, 
      \begin{equation} 
        \sup_{t \in \mathscr{T}}\left|\tilde{{\beta_1}}^{(d)}(t)-{\beta_1}(t)-\frac{1}{nb_n}\sum_{i=1}^n e_{i}^{(d)} K_{b_{n}}^{*}(i / n-t)\right| = \Op\left(\rho^*_n/(nb_n) \right),
        \label{eq:expansion_d}
      \end{equation}
      where $\mathscr{T} = [b_n,1-b_n]$, $\rho^*_n = (nb_n)^{d-1/2} \log n \mf 1(0 \leq d \leq 1/26) + (nb_n)^{d-1/2}b_n^{-1/2}\mf 1(1/26 < d < 1/2)+ b_n^2$.
      Then, it follows that 
      \begin{align}
       &\max_{\nb + 1 \leq r \leq n-\nb} \left| \sum_{i=\nb+1}^r \tilde e_i^{(d)} -  \sum_{i=\nb+1}^r\left( e_i^{(d)}- \frac{1}{nb_n} \sum_{j=1}^n e_{j}^{(d)} K_{b_{n}}^{*}(i / n-j/n) \right) \right|\\ & = \Op(\rho_n^*/b_n) = \op((nb_n)^{1/2+d}).\label{deg:step2.1}
      \end{align}
      A careful investigation of \cref{lm:alt_gaussian} yields when $nb_n^3 \to \infty$, 
      \begin{align}
        \max_{\nb + 1 \leq r \leq n-\nb}\max_{r - \nb \leq s \leq r + \nb } \left| \sum_{k=r-\nb}^{s} (e_k^{(d)} - R_{k,n,1})\right| &= \Op(n^{\alpha_0(d-1/2) + 1}b_n + n^{\alpha_0/4 + d + q} b_n^d)\\ &= \op((nb_n)^{d+1/2}),\label{deg:alt_approx}
      \end{align}
      where $\alpha_0 \in (1, 4/3)$, $q>0$ can be arbitrarily small. Observe that by \eqref{deg:coef}, uniformly for $3\nb + 1 \leq r \leq n - \nb$, 
      \begin{align}
        &\sum_{i=\nb+1}^r e_i^{(d)}-\frac{1}{nb_n}\sum_{j=1}^n \sum_{i=\nb+1}^r e_{j}^{(d)} K_{b_{n}}^{*}(i / n-j/n)\\  
        &= \sum_{j = \nb + 1}^{2\nb}\left(\int_{-1}^{\frac{j-\nb}{nb_n}} K^*(t)dt\right) e_j^{(d)} + \sum_{j = r - \nb + 1}^{r}\left(\int_{\frac{r-j}{nb_n}}^{1} K^*(t)dt \right) e_j^{(d)} \\ 
        &- \sum_{j = 1}^{\nb}\left(\int_{\frac{j-\nb}{nb_n}}^1 K^*(t)dt\right) e_j^{(d)} - \sum_{j = r + 1}^{r + \nb}\left(\int_{-1}^{\frac{r-j}{nb_n}} K^*(t)dt\right) e_j^{(d)} + \Op(b_n^{-1}),
        \label{deg:sum}
      \end{align}
      and similarly uniformly for $\nb + 1 \leq r \leq 3\nb + 1$, 
      \begin{align}
        &\sum_{i=\nb+1}^r e_i^{(d)}-\frac{1}{nb_n}\sum_{j=1}^n \sum_{i=\nb+1}^r e_{j}^{(d)} K_{b_{n}}^{*}(i / n-j/n)\\  
        &= \sum_{j = \nb + 1}^{r-\nb}\left(\int_{-1}^{\frac{j-\nb}{nb_n}} K^*(t)dt\right) e_j^{(d)} + \sum_{j = r - \nb + 1}^{r}\left(\int_{\frac{r-j}{nb_n}}^{1} K^*(t)dt \right) e_j^{(d)} \\ 
        &- \sum_{j = 1}^{\nb}\left(\int_{\frac{j-\nb}{nb_n}}^1 K^*(t)dt\right) e_j^{(d)} - \sum_{j = r + 1}^{r + \nb}\left(\int_{-1}^{\frac{r-j}{nb_n}} K^*(t)dt\right) e_j^{(d)} + \Op(b_n^{-1}). 
        \label{deg:sum1}
      \end{align}
      Combining \eqref{deg:alt_approx}, \eqref{deg:sum} and \eqref{deg:sum1}, by summation-by-parts formula, we have 
    \begin{align}
      \max_{\nb +1 \leq r \leq n - \nb}\left|\sum_{i=\nb+1}^r (e_i^{(d)} - R_{i,n,1})-\frac{1}{nb_n}\sum_{j=1}^n  \sum_{i=\nb+1}^r K_{b_{n}}^{*}(i / n - j/n)(e_j^{(d)} - R_{j,n,1})\right| = \op((nb_n)^{d+1/2}).\label{deg:step2.2}
    \end{align}
    Using similar technique in \eqref{deg:sum}, it follows from tedious calculation that for $r \geq 2\nb$, there exists a positive constant $c^{\prime}$ such that
    \begin{align}
      \left\| \sum_{i=\nb+1}^r R_{i,n,1} -\frac{1}{nb_n}\sum_{j=1}^n \sum_{i=\nb+1}^r K_{b_{n}}^{*}(i / n - j/n) R_{j,n,1} \right\|^2 \geq c^{\prime} (nb_n)^{2d+1}.\label{deg:step2.3}
    \end{align}
    Finally, since $(nb_n)/m \to \infty$, (i) of proposition follows from \eqref{deg:step1}, \eqref{deg:step2.1}, \eqref{deg:step2.2}, and \eqref{deg:step2.3}.
    
    Proof of (ii). 
     Similar to the steps in \eqref{deg:step1},  we have 
      \begin{align}
       \max_{\nb +1 \leq r \leq n - \nb} \|\tilde G_{r,d}\|^2 = O(nb_n m^{2d_n}),\label{deg:step1+local}
      \end{align}
     and uniformly for $\nb + 1 
    \leq r \leq n - \nb$, there exists a constant $c^{\prime\prime}$,
    \begin{align}
       \left \|\sum_{i=\nb+1}^r \tilde e_i^{(d_n)} \right\| \geq c^{\prime\prime} (nb_n)^{d_n+1/2}.
    \end{align}
    Since $m^{2d_n} = e^{2c\alpha_1}$, and $(nb_n)^{d_n} = e^{2c\beta}$, $\beta > \alpha_1$,  when $c$ is sufficiently large, we have 
    \begin{align}
      \lim_{n \to \infty} P\left(\left|\sum_{i=\nb +1}^r \tilde e_i^{(d_n)}\right| >  \left|\tilde G_{r,d_n}\right| \text{for $r = \nb +1, \cdots, n - \nb$}\right) = 1.    \end{align}
    
    \hfill $\Box$

       \small
 \normalem
\bibliographystyle{apalike}
\bibliography{main}

\begin{thebibliography}{}

\bibitem[Bai and Wu, 2023]{lrvunpublish}
Bai, L. and Wu, W. (2023).
\newblock Supplement to ``difference based variance estimate in time series
  nonparametric regression with applications to specification tests''.
\newblock Manuscript.

\bibitem[Beran, 2009]{BERAN2009900}
Beran, J. (2009).
\newblock On parameter estimation for locally stationary long-memory processes.
\newblock {\em Journal of Statistical Planning and Inference}, 139(3):900 --
  915.

\bibitem[Beran et~al., 2013]{BeranLongmemory}
Beran, J., Feng, Y., Ghosh, S., and Kulik, R. (2013).
\newblock {\em Long-Memory Processes}.
\newblock Springer.

\bibitem[Berkes et~al., 2006]{shao2006aos}
Berkes, I., Horváth, L., Kokoszka, P., and Shao, Q.-M. (2006).
\newblock {On discriminating between long-range dependence and changes in
  mean}.
\newblock {\em The Annals of Statistics}, 34(3):1140 -- 1165.

\bibitem[Billingsley, 1999]{billingsley1999convergence}
Billingsley, P. (1999).
\newblock {\em Convergence of probability measures}.
\newblock Wiley.

\bibitem[Caporale and Gil-Alana, 2013]{caporale2013long}
Caporale, G.~M. and Gil-Alana, L.~A. (2013).
\newblock Long memory and fractional integration in high frequency data on the
  us dollar/british pound spot exchange rate.
\newblock {\em International review of financial analysis}, 29:1--9.

\bibitem[Caporale et~al., 2016]{unemployment2016}
Caporale, G.~M., Gil-Alana, L.~A., and Lovcha, Y. (2016).
\newblock Testing unemployment theories: A multivariate long memory approach.
\newblock {\em Journal of Applied Economics}, 19(1):95--112.

\bibitem[Cavaliere et~al., 2020]{cavaliere2020adaptive}
Cavaliere, G., Ørregaard Nielsen, M., and Taylor, A. M.~R. (2020).
\newblock Adaptive inference in heteroscedastic fractional time series models.
\newblock {\em Journal of Business \& Economic Statistics}, 0(0):1--16.

\bibitem[Chen and Song, 2015]{chen2015function}
Chen, M. and Song, Q. (2015).
\newblock {Simultaneous inference of the mean of functional time series}.
\newblock {\em Electronic Journal of Statistics}, 9(2):1779 -- 1798.

\bibitem[Chen et~al., 2018]{chen2018jbes}
Chen, X.~B., Gao, J., Li, D., and Silvapulle, P. (2018).
\newblock Nonparametric estimation and forecasting for time-varying coefficient
  realized volatility models.
\newblock {\em Journal of Business \& Economic Statistics}, 36(1):88--100.

\bibitem[Cheung and Lai, 1993]{cheung1993gold}
Cheung, Y.-W. and Lai, K.~S. (1993).
\newblock Do gold market returns have long memory?
\newblock {\em Financial Review}, 28(2):181--202.

\bibitem[Craven and Wahba, 1978]{craven1978smoothing}
Craven, P. and Wahba, G. (1978).
\newblock Smoothing noisy data with spline functions.
\newblock {\em Numerische mathematik}, 31(4):377--403.

\bibitem[Dahlhaus, 1997]{dahlhaus1997}
Dahlhaus, R. (1997).
\newblock Fitting time series models to nonstationary processes.
\newblock {\em Ann. Statist.}, 25(1):1--37.

\bibitem[Dahlhaus et~al., 2019]{dahlhaus2019bej}
Dahlhaus, R., Richter, S., and Wu, W.~B. (2019).
\newblock {Towards a general theory for nonlinear locally stationary
  processes}.
\newblock {\em Bernoulli}, 25(2):1013 -- 1044.

\bibitem[Davis and Yau, 2013]{davis2013consistency}
Davis, R.~A. and Yau, C.~Y. (2013).
\newblock {Consistency of minimum description length model selection for
  piecewise stationary time series models}.
\newblock {\em Electronic Journal of Statistics}, 7(none):381 -- 411.

\bibitem[Dehling and Taqqu, 1989]{dehling1989empirical}
Dehling, H. and Taqqu, M.~S. (1989).
\newblock {The Empirical Process of some Long-Range Dependent Sequences with an
  Application to $U$-Statistics}.
\newblock {\em The Annals of Statistics}, 17(4):1767 -- 1783.

\bibitem[Dette et~al., 2017]{Detectinglrdependence}
Dette, H., Preuss, P., and Sen, K. (2017).
\newblock {Detecting long-range dependence in non-stationary time series}.
\newblock {\em Electronic Journal of Statistics}, 11(1):1600 -- 1659.

\bibitem[Dette and Wu, 2019]{dette2019detecting}
Dette, H. and Wu, W. (2019).
\newblock Detecting relevant changes in the mean of nonstationary processes—a
  mass excess approach.
\newblock {\em The Annals of Statistics}, 47(6):3578--3608.

\bibitem[Dette and Wu, 2021]{Dette2021ConfidenceSF}
Dette, H. and Wu, W. (2021).
\newblock Confidence surfaces for the mean of locally stationary functional
  time series.
\newblock {\em arXiv preprint arXiv:2109.03641}.

\bibitem[Dette et~al., 2019]{dette2018change}
Dette, H., Wu, W., and Zhou, Z. (2019).
\newblock Change point analysis of correlation in non-stationary time series.
\newblock {\em Statistica Sinica}, 29(2):611--643.

\bibitem[Duffy and Kasparis, 2021]{Ioa2021}
Duffy, J.~A. and Kasparis, I. (2021).
\newblock {Estimation and inference in the presence of fractional $d=1/2$ and
  weakly nonstationary processes}.
\newblock {\em The Annals of Statistics}, 49(2):1195 -- 1217.

\bibitem[Fan, 1993]{fan1993local}
Fan, J. (1993).
\newblock Local linear regression smoothers and their minimax efficiencies.
\newblock {\em The annals of Statistics}, 21(1):196--216.

\bibitem[Fan and Gijbels, 1996]{fan1996local}
Fan, J. and Gijbels, I. (1996).
\newblock {\em Local polynomial modelling and its applications}.
\newblock Number~66 in Monographs on statistics and applied probability series.
  Chapman \& Hall, London [u.a.].

\bibitem[Fan and Zhang, 2000]{fan2000simultaneous}
Fan, J. and Zhang, W. (2000).
\newblock Simultaneous confidence bands and hypothesis testing in
  varying-coefficient models.
\newblock {\em Scandinavian Journal of Statistics}, 27(4):715--731.

\bibitem[Ferreira et~al., 2018]{ferreira2018estimation}
Ferreira, G., Pi{\~n}a, N., and Porcu, E. (2018).
\newblock Estimation of slowly time-varying trend function in long memory
  regression models.
\newblock {\em Journal of Statistical Computation and Simulation},
  88(10):1903--1920.

\bibitem[Giraitis et~al., 2001]{giraitis2001testing}
Giraitis, L., Kokoszka, P., and Leipus, R. (2001).
\newblock Testing for long memory in the presence of a general trend.
\newblock {\em Journal of Applied Probability}, 38(4):1033--1054.

\bibitem[Giraitis et~al., 2003]{giraitis2003rescaled}
Giraitis, L., Kokoszka, P., Leipus, R., and Teyssi{\`e}re, G. (2003).
\newblock Rescaled variance and related tests for long memory in volatility and
  levels.
\newblock {\em Journal of econometrics}, 112(2):265--294.

\bibitem[Harris and Kew, 2017]{harris2017adaptive}
Harris, D. and Kew, H. (2017).
\newblock Adaptive long memory testing under heteroskedasticity.
\newblock {\em Econometric Theory}, 33(3):755--778.

\bibitem[Harris et~al., 2008]{harris2008testing}
Harris, D., McCabe, B., and Leybourne, S. (2008).
\newblock Testing for long memory.
\newblock {\em Econometric Theory}, 24(1):143--175.

\bibitem[Hu et~al., 2019]{you2019}
Hu, L., Huang, T., and You, J. (2019).
\newblock Estimation and identification of a varying-coefficient additive model
  for locally stationary processes.
\newblock {\em Journal of the American Statistical Association},
  114(527):1191--1204.

\bibitem[Hurst, 1951]{hurst1951long}
Hurst, H.~E. (1951).
\newblock Long-term storage capacity of reservoirs.
\newblock {\em Trans. Amer. Soc. Civil Eng.}, 116:770--799.

\bibitem[Jiang et~al., 2020]{jiang2020time}
Jiang, F., Zhao, Z., and Shao, X. (2020).
\newblock Time series analysis of covid-19 infection curve: A change-point
  perspective.
\newblock {\em Journal of econometrics}.

\bibitem[Karatzas and Shreve, 1988]{karatzas1988brownian}
Karatzas, I. and Shreve, S.~E. (1988).
\newblock Brownian motion.
\newblock In {\em Brownian Motion and Stochastic Calculus}, pages 47--127.
  Springer.

\bibitem[Kokoszka and Young, 2016]{kokoszka2016kpss}
Kokoszka, P. and Young, G. (2016).
\newblock Kpss test for functional time series.
\newblock {\em Statistics}, 50(5):957--973.

\bibitem[Kokoszka and Taqqu, 1995]{KOKOSZKA199519}
Kokoszka, P.~S. and Taqqu, M.~S. (1995).
\newblock Fractional arima with stable innovations.
\newblock {\em Stochastic Processes and their Applications}, 60(1):19 -- 47.

\bibitem[Koutsoyiannis, 2013]{koutsoyiannis2013hydrology}
Koutsoyiannis, D. (2013).
\newblock Hydrology and change.
\newblock {\em Hydrological Sciences Journal}, 58(6):1177--1197.

\bibitem[Kulik and Wichelhaus, 2012]{kulik2012conditional}
Kulik, R. and Wichelhaus, C. (2012).
\newblock Conditional variance estimation in regression models with long
  memory.
\newblock {\em Journal of Time Series Analysis}, 33(3):468--483.

\bibitem[Kwiatkowski et~al., 1992]{kwiatkowski1992testing}
Kwiatkowski, D., Phillips, P.~C., Schmidt, P., and Shin, Y. (1992).
\newblock Testing the null hypothesis of stationarity against the alternative
  of a unit root.
\newblock {\em Journal of econometrics}, 54(1-3):159--178.

\bibitem[Lee and Schmidt, 1996]{lee1996power}
Lee, D. and Schmidt, P. (1996).
\newblock On the power of the kpss test of stationarity against
  fractionally-integrated alternatives.
\newblock {\em Journal of econometrics}, 73(1):285--302.

\bibitem[Lima and Xiao, 2004]{lima2004robustness}
Lima, L. and Xiao, Z. (2004).
\newblock Robustness of stationary tests under long-memory alternatives.
\newblock Technical report, EPGE Brazilian School of Economics and Finance-FGV
  EPGE (Brazil).

\bibitem[Lo, 1989]{lo1989long}
Lo, A.~W. (1989).
\newblock Long-term memory in stock market prices.
\newblock Technical report, National Bureau of Economic Research.

\bibitem[MacNeill, 1974]{MacNeill1974}
MacNeill, I.~B. (1974).
\newblock Tests for change of parameter at unknown times and distributions of
  some related functionals on brownian motion.
\newblock {\em The Annals of Statistics}, 2(5):950--962.

\bibitem[Marinucci and Robinson, 1999]{marinucci1999alternative}
Marinucci, D. and Robinson, P.~M. (1999).
\newblock Alternative forms of fractional brownian motion.
\newblock {\em Journal of statistical planning and inference},
  80(1-2):111--122.

\bibitem[McCloskey and Perron, 2013]{mccloskey2013memory}
McCloskey, A. and Perron, P. (2013).
\newblock Memory parameter estimation in the presence of level shifts and
  deterministic trends.
\newblock {\em Econometric Theory}, 29(6):1196--1237.

\bibitem[Nason et~al., 2000]{nason2000wavelet}
Nason, G.~P., Von~Sachs, R., and Kroisandt, G. (2000).
\newblock Wavelet processes and adaptive estimation of the evolutionary wavelet
  spectrum.
\newblock {\em Journal of the Royal Statistical Society: Series B (Statistical
  Methodology)}, 62(2):271--292.

\bibitem[Pipiras and Taqqu, 2017]{pipiras2017long}
Pipiras, V. and Taqqu, M.~S. (2017).
\newblock {\em Long-range dependence and self-similarity}, volume~45.
\newblock Cambridge university press.

\bibitem[Politis et~al., 1999]{politis1999subsampling}
Politis, D.~N., Romano, J.~P., and Wolf, M. (1999).
\newblock {\em Subsampling}.
\newblock Springer Science \& Business Media.

\bibitem[Preuß and Vetter, 2013]{preu2013}
Preuß, P. and Vetter, M. (2013).
\newblock Discriminating between long-range dependence and non-stationarity.
\newblock {\em Electron. J. Statist.}, 7:2241--2297.

\bibitem[Qu, 2011]{qu2011test}
Qu, Z. (2011).
\newblock A test against spurious long memory.
\newblock {\em Journal of Business \& Economic Statistics}, 29(3):423--438.

\bibitem[Shao and Wu, 2007]{shao2007local}
Shao, X. and Wu, W.~B. (2007).
\newblock Local asymptotic powers of nonparametric and semiparametric tests for
  fractional integration.
\newblock {\em Stochastic processes and their Applications}, 117(2):251--261.

\bibitem[Sibbertsen et~al., 2018]{SIBBERTSEN201833}
Sibbertsen, P., Leschinski, C., and Busch, M. (2018).
\newblock A multivariate test against spurious long memory.
\newblock {\em Journal of Econometrics}, 203(1):33--49.

\bibitem[Taqqu, 1975]{taqqu1975weak}
Taqqu, M.~S. (1975).
\newblock Weak convergence to fractional brownian motion and to the rosenblatt
  process.
\newblock {\em Zeitschrift f{\"u}r Wahrscheinlichkeitstheorie und verwandte
  Gebiete}, 31(4):287--302.

\bibitem[Vogt, 2012]{vogt2012}
Vogt, M. (2012).
\newblock Nonparametric regression for locally stationary time series.
\newblock {\em Ann. Statist.}, 40(5):2601--2633.

\bibitem[Vogt and Dette, 2015]{vogt2015detecting}
Vogt, M. and Dette, H. (2015).
\newblock Detecting gradual changes in locally stationary processes.
\newblock {\em The Annals of Statistics}, 43(2):713--740.

\bibitem[Wu and Zhou, 2018a]{wu2018gradient}
Wu, W. and Zhou, Z. (2018a).
\newblock {Gradient-based structural change detection for nonstationary time
  series M-estimation}.
\newblock {\em The Annals of Statistics}, 46(3):1197 -- 1224.

\bibitem[Wu and Zhou, 2018b]{wu2018}
Wu, W. and Zhou, Z. (2018b).
\newblock Simultaneous quantile inference for non-stationary long-memory time
  series.
\newblock {\em Bernoulli}, 24(4A):2991--3012.

\bibitem[Wu, 2005]{wu2005nonlinear}
Wu, W.~B. (2005).
\newblock Nonlinear system theory: Another look at dependence.
\newblock {\em Proceedings of the National Academy of Sciences},
  102(40):14150--14154.

\bibitem[Wu, 2007]{wu2007strong}
Wu, W.~B. (2007).
\newblock Strong invariance principles for dependent random variables.
\newblock {\em The Annals of Probability}, 35(6):2294--2320.

\bibitem[Wu and Shao, 2006]{wu2006invariance}
Wu, W.~B. and Shao, X. (2006).
\newblock Invariance principles for fractionally integrated nonlinear
  processes.
\newblock In {\em Recent developments in nonparametric inference and
  probability}, pages 20--30. Institute of Mathematical Statistics.

\bibitem[Wu and Zhao, 2007]{wu2007inference}
Wu, W.~B. and Zhao, Z. (2007).
\newblock Inference of trends in time series.
\newblock {\em Journal of the Royal Statistical Society: Series B (Statistical
  Methodology)}, 69(3):391--410.

\bibitem[Wu and Zhou, 2011]{wu2011gaussian}
Wu, W.~B. and Zhou, Z. (2011).
\newblock Gaussian approximations for non-stationary multiple time series.
\newblock {\em Statistica Sinica}, 21(3):1397--1413.

\bibitem[Yu et~al., 2015]{yu2015useful}
Yu, Y., Wang, T., and Samworth, R.~J. (2015).
\newblock A useful variant of the davis--kahan theorem for statisticians.
\newblock {\em Biometrika}, 102(2):315--323.

\bibitem[Zhang et~al., 2011]{ZHANG2011121}
Zhang, Q., Zhou, Y., Singh, V.~P., and Chen, Y.~D. (2011).
\newblock Comparison of detrending methods for fluctuation analysis in
  hydrology.
\newblock {\em Journal of Hydrology}, 400(1):121--132.

\bibitem[Zhou and Wu, 2009]{wu2009quantile}
Zhou, Z. and Wu, W.~B. (2009).
\newblock {Local linear quantile estimation for nonstationary time series}.
\newblock {\em The Annals of Statistics}, 37(5B):2696 -- 2729.

\bibitem[Zhou and Wu, 2010]{zhou2010simultaneous}
Zhou, Z. and Wu, W.~B. (2010).
\newblock Simultaneous inference of linear models with time varying
  coefficients.
\newblock {\em Journal of the Royal Statistical Society: Series B (Statistical
  Methodology)}, 72(4):513--531.

\end{thebibliography}

\end{document}